\documentclass[11pt]{amsart}

\usepackage{paralist, amsmath, amsthm, amssymb, color, graphicx, hyperref}
\usepackage[margin=1.4in]{geometry}
\DeclareGraphicsExtensions{.jpg,.pdf,.eps}
\graphicspath{{./,./Figures/}}

\newtheorem{thm}{Theorem}
\numberwithin{equation}{section}
\numberwithin{thm}{section}
\newtheorem{theorem}[thm]{Theorem}

\newtheorem{prop}[thm]{Proposition}
\newtheorem{lemma}[thm]{Lemma}
\newtheorem{cor}[thm]{Corollary} 
\newtheorem{corollary}[thm]{Corollary} 
\newtheorem{claim}[thm]{Claim}

\theoremstyle{definition}
\newtheorem{definition}[thm]{Definition}

\theoremstyle{remark}

\newtheorem{remark}[thm]{Remark}

\newtheorem{example}[thm]{Example}
\newtheorem{question}[thm]{Question}

\newcommand{\al}{\alpha}
\newcommand{\be}{\beta}
\newcommand{\ga}{\gamma}
\newcommand{\de}{\delta}

\newcommand{\si}{\sigma}
\newcommand{\la}{\lambda}

\newcommand{\De}{\Delta}
\newcommand{\Ga}{\Gamma}
\newcommand{\La}{\Lambda}
\newcommand{\all}{\al^{(0)}}
\newcommand{\bee}{\al^{(1)}}

\newcommand{\CC}{\mathbb{C}}

\newcommand{\QQ}{\mathbb{Q}}
\newcommand{\RR}{\mathbb{R}}
\newcommand{\ZZ}{\mathbb{Z}}
\newcommand{\NN}{\mathbb{N}}

\newcommand{\mA}{\mathcal{A}}
\newcommand{\mB}{\mathcal{B}}
\newcommand{\mC}{\mathcal{C}}
\newcommand{\mD}{\mathcal{D}}

\newcommand{\mG}{\mathcal{G}}

\newcommand{\mP}{\mathcal{P}}
\newcommand{\mR}{\mathcal{R}}
\newcommand{\mT}{\mathcal{T}}
\newcommand{\mU}{\mathcal{U}}
\newcommand{\mW}{\mathcal{W}}

\newcommand{\mZ}{\mathcal{Z}}

\newcommand{\xx}{\textbf{x}}

\newcommand{\kk}{\textbf{k}}
\newcommand{\nn}{\textbf{n}}
\renewcommand{\tt}{\textbf{t}}

\renewcommand{\SS}{\textbf{S}}
\newcommand{\hSS}{{\widehat{\textbf{S}}}}

\newcommand{\hw}{\hat{w}}
\newcommand{\tT}{\tilde{\mT}}
\newcommand{\tD}{\tilde{\mD}}
\newcommand{\tw}{\tilde{w}}
\newcommand{\tphi}{\tilde{\phi}}
\newcommand{\tpsi}{\tilde{\psi}}
\newcommand{\tP}{\widetilde{P}}
\newcommand{\tR}{\tilde{R}}

\newcommand{\bu}{\bullet}
\newcommand{\Cat}{\textrm{Cat}}

\newcommand{\Id}{\textrm{Id}}
\DeclareMathOperator{\one}{\textbf{1}}

\DeclareMathOperator{\ch}{child}

\DeclareMathOperator{\ls}{lsib}
\DeclareMathOperator{\pos}{pos}
\DeclareMathOperator{\pol}{Pol}
\DeclareMathOperator{\lab}{lab}

\DeclareMathOperator{\leaf}{leaf}
\DeclareMathOperator{\wid}{wid}
\DeclareMathOperator{\rank}{rank}
\DeclareMathOperator{\config}{config}
\DeclareMathOperator{\cadet}{cadet}
\DeclareMathOperator{\parent}{parent}
\DeclareMathOperator{\energy}{energy}
\DeclareMathOperator{\drift}{drift}

\DeclareMathOperator{\Seq}{Seq}
\DeclareMathOperator{\asc}{asc}
\DeclareMathOperator{\des}{des}

\DeclareMathOperator{\ee}{\textrm{e}}
\DeclareMathOperator{\vv}{\textrm{v}}
\DeclareMathOperator{\comp}{\textrm{c}}

\newcommand{\ds}{\displaystyle}
\newcommand{\fig}[3]{\begin{figure}[h!]\begin{center}\includegraphics[#1]{#2}\end{center}\caption{#3}\label{fig:#2}\end{figure}}

\renewcommand{\ni}{\noindent}

\title[Deformations of the braid arrangement and Trees]{Deformations of the braid arrangement and Trees\\[2ex]
 \footnotesize\mdseries Dedicated to Ira Gessel for his retirement}
\author{Olivier Bernardi}
\thanks{This work was partially supported by NSF grant DMS-1400859. 
Part of it was completed while visiting the MIT in the Spring 2016. Many thanks to Ira Gessel for suggesting this problem and providing many valuable inputs, and to Sam Hopkins and Alexander Postnikov for interesting discussions. We are also grateful to an anonymous referee for suggesting additional references and making many other valuable comments.}
\address{Olivier Bernardi, Brandeis University, 415 South Street, Waltham, MA 02453, USA.}
\date{\today}

\begin{document}

\begin{abstract}
We establish general counting formulas and bijections for deformations of the braid arrangement. Precisely, we consider real hyperplane arrangements such that all the hyperplanes are of the form $x_i-x_j=s$ for some integer $s$. Classical examples include the braid, Catalan, Shi, semiorder and Linial arrangements, as well as graphical arrangements. We express the number of regions of any such arrangement as a signed count of decorated plane trees. The characteristic and coboundary polynomials of these arrangements also have simple expressions in terms of these trees. 

We then focus on certain ``well-behaved'' deformations of the braid arrangement that we call \emph{transitive}. This includes the Catalan, Shi, semiorder and Linial arrangements, as well as many other arrangements appearing in the literature. For any transitive deformation of the braid arrangement we establish a simple bijection between regions of the arrangement and a set of labeled plane trees defined by local conditions. This answers a question of Gessel. 
\end{abstract}

\maketitle

\section{Introduction}
In this article we establish enumerative and bijective results about classical families of hyperplane arrangements. 
Specifically, we consider real hyperplane arrangements made of a finite number of hyperplanes of the form 
$$H_{i,j,s}=\{(x_1,\ldots,x_n)\in \RR^n ~|~ x_i-x_j=s\},$$ 
with $i,j\in\{1,\ldots,n\}$ and $s\in \ZZ$. We shall call them \emph{deformations of the braid arrangement}. 
In particular, given an integer $n$ and a finite set of integers $S$, the \emph{$S$-braid arrangement in dimension $n$}, denoted $\mA_S(n)$, 
is the arrangement made of the hyperplanes $H_{i,j,s}$ for all $1\leq i<j\leq n$ and all $s\in S$. 
Classical examples include the \emph{braid}, \emph{Catalan}, \emph{Shi}, \emph{semiorder}, and \emph{Linial} arrangements, which correspond to $S=\{0\},~\{-1,0,1\},~\{0,1\},~\{-1,1\}$, and $\{1\}$ respectively. These arrangements are represented\footnote{Here and later, the arrangements are represented by drawing their intersection with the hyperplane $H_0=\{(x_1,\ldots,x_n)~|~x_1+\ldots+x_n=0\}$, which is orthogonal to all the hyperplanes of any deformation of the braid arrangement.} 
in Figure~\ref{fig:classical-arrangements}.
We refer the reader to~\cite{Orlik:hyperplane-arrangements} or \cite{Stanley:hyperplane-arrangements} for an introduction to the general theory of hyperplane arrangements.

\fig{width=\linewidth}{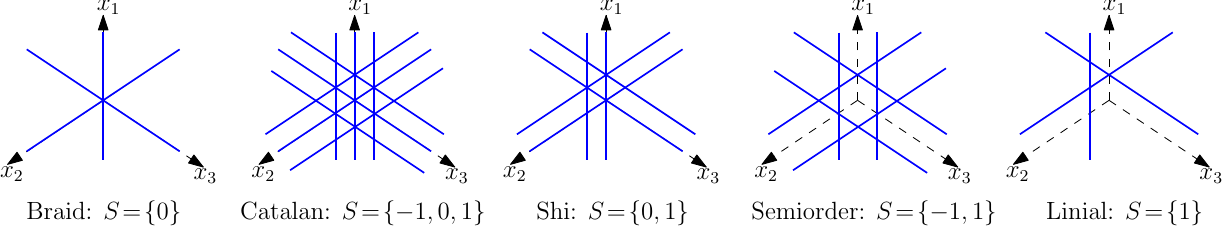}{The braid, Catalan, Shi, semiorder, and Linial arrangements in dimension $n=3$ (seen from the direction $(1,1,1)$).}

There is an extensive literature on counting regions of deformations of the braid arrangement, starting with the work of Shi \cite{Shi:nb-regions,Shi:nb-regions-Weyl}. Important seminal results on this enumerative question were established by Stanley~\cite{Stanley:hyperplane-interval-orders-overview,Stanley:hyperplane-tree-inversions}, Postnikov and Stanley~\cite{Postnikov:coxeter-hyperplanes}, and Athanasiadis~\cite{Athanasiadis:finite-field-method,Athanasiadis:PhD}. Since then, the subject has become quite popular among combinatorialists, and many beautiful counting formulas and bijections were discovered for various families of arrangements; see in particular \cite{Athanasiadis:free-deformations,Athanasiadis:Catalan-Hyperplanes,Athanasiadis:free-deformations,Athanasiadis:bijection-Shi,Ardila:Tutte-hyperplanes,Armstrong:Shi-Ish,Corteel:hyperplane-valued-graphs,Gessel:labeled-trees-Shur-positive,Forge:Linial,Headley:hyperplane-Weyl-groups,Hopkins:bigraphical-arrangements,Hopkins:G-Parking-labelling,Leven:bijections-Ish-Shi} (see also Section~\ref{sec:definitions} for additional references). 

About a decade ago, an interesting pattern was observed by Ira Gessel. In an unpublished manuscript, Gessel obtained an equation for the generating function of labeled binary trees counted according to ascents and descents along left or right edges\footnote{Gessel's result was then rederived in two different ways by Kalikow \cite{Kalikow:symmetries-in-trees} and Drake \cite{Drake:Thesis}.}. By specializing his equation, Gessel observed (based on known results for arrangements) that each of the five classical arrangements families $\mA_S$ defined above (braid, Catalan, Shi, semi-order, Linial) can be associated to a simple family $\mB_S$ of binary trees (characterized by some ascent and descent conditions), in such a way that the regions of $\mA_S(n)$ are equinumerous to trees with $n$ nodes in $\mB_S$ (see Section~\ref{sec:known-arrangements}). This opened the question of explaining these mysterious enumerative identities between regions of arrangements and binary trees; and attempts at answering this question were made for instance in \cite{Corteel:hyperplane-valued-graphs,Forge:Linial}.

It is the goal of this paper to explain that Gessel's observation is not a mere coincidence, but rather the manifestation of a more general theory which unifies and extends many previous results. This paper has two parts. The first part gives enumerative results which apply to every deformation of the braid arrangement. The second part establishes bijections for every deformations of the braid arrangement satisfying a certain \emph{transitivity condition}. In the rest of this introduction, we give a detailed outline of the paper and a preview of our results.

In the first part of the paper (Section~\ref{sec:definitions} to~\ref{sec:proof}), we deal with arbitrary deformations of the braid arrangement. 
For any such arrangement in dimension $n$, we express the number of regions as a \emph{signed} count of some (decorated, labeled, $k$-ary) trees with $n$ nodes, that we call \emph{boxed trees}. 
These results are given in Section~\ref{sec:regions-shaken} for the case of $S$-braid arrangements (Theorem~\ref{thm:signed-count}), and in Section~\ref{sec:multi-shaken} for the  general case (Theorem~\ref{thm:signed-count-multi}). When the arrangement satisfy the \emph{transitivity condition}, our counting result simplifies greatly and the regions are shown to be equinumerous to a family of (labeled, plane) trees satisfying certain ascent and descent conditions (Theorems~\ref{thm:unsigned-count} and~\ref{thm:unsigned-count-multi}).
In Section~\ref{sec:char-poly}, we generalize the previous counting results by expressing the characteristic polynomial and coboundary polynomial (equivalently, Tutte polynomial) of any deformation of the braid arrangement in terms of boxed trees (Theorem~\ref{thm:energy-count-multi}). 
In Section~\ref{sec:GF}, we use the preceding expressions in terms of boxed trees in order to establish equations for the generating function of the number of regions, and more generally for the generating function of coboundary polynomials. We thereby recover and extend many known results.

All the proofs for our counting results are gathered in Section~\ref{sec:proof}. The proof is inspired by statistical mechanics considerations (although some of the arguments can alternately be interpreted in light of the finite field method), and has three steps which could informally be described as follows: 
\begin{compactitem}
\item at the first step (Lemma~\ref{lem:1-multi}) we express the coboundary polynomials of the arrangements as a signed count of some decorated graphs (directly encoding the central subarrangements), 
\item at the second step (Lemma~\ref{lem:2-multi}) we express the generating function of these decorated graphs in terms of a 1-dimensional gas model (by using a version of Mayers' theory of cluster integrals),
\item at the third step (Lemma~\ref{lem:3-multi}) we rearrange the information about the gas model configurations in order to encode them in terms of boxed trees.
\end{compactitem}



In the second part of the paper (Section~\ref{sec:bij}), we establish bijections for regions of \emph{transitive} deformations of the braid arrangement. These are arrangements made up of hyperplanes $H_{i,j,s}$, where the triples $(i,j,s)$ satisfy certain conditions (see Definition~\ref{def:transitive-multi}). Examples of transitive arrangements include the $S$-braid arrangements for $S\subseteq \{-1,0,1\}$, or $S$ an interval containing 1, and the $G$-Shi arrangement for any graph~$G$. As mentioned earlier, for transitive arrangements our enumerative result simplifies and the regions are found to be equinumerous to some simple families of (labeled, plane) trees. In Section~\ref{sec:bij} we establish, for every transitive arrangement, a direct bijection between the regions of the arrangement and the corresponding family of trees (Theorem~\ref{thm:bij-gle}). Our bijection is surprising explicit: given a tree it is very simple to determine all the linear inequalities that define the corresponding region of the arrangement. In order to illustrate this fact, we now present the bijection in the case of the Linial arrangement, which is illustrated in Figure~\ref{fig:Linial-bijection}.

\begin{example} \label{exp:Linial-intro}
The regions of the Linial arrangement $\mA_{\{1\}}(n)$ are in bijection with the set $\mT_{\{1\}}(n)$ of binary trees with $n$ labeled node satisfying the following condition: \emph{for all node $u\in[n]$ having at least one child which is a node, the rightmost such child $v$ is such that $v<u$}. 

The bijection $\Psi$ associates to any tree $T$ in $\mT_{\{1\}}(n)$, the region $\rho(T)$ of $\mA_{\{1\}}(n)$ made of the points $(x_1,\ldots,x_n)$ satisfying the following inequalities for all $1\leq i<j\leq n$: \emph{$x_i-x_j<1$ if and only if either $\drift(i)\leq \drift(j)$ or $\drift(i)=\drift(j)+1$ and $i$ appears before $j$ in the postfix order of $T$}, where $\drift(v)$ is the number of ancestors of $v$ (including $v$) which are right-children. See Figure~\ref{fig:Linial-bijection} for the case $n=3$.
\end{example}
 
\fig{width=.7\linewidth}{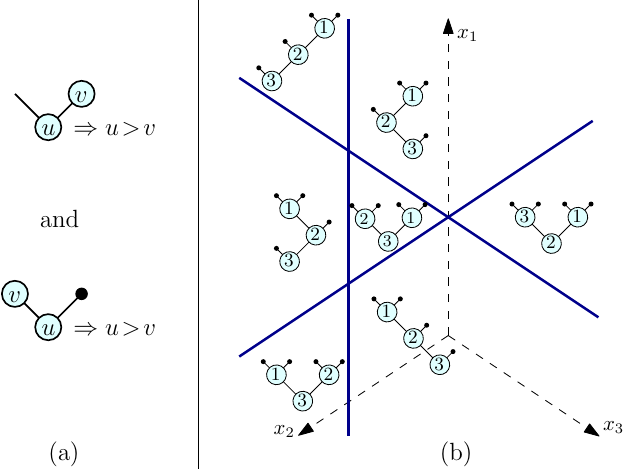}{(a) The condition defining the trees in $\mT_{\{1\}}(n)$. (b) The bijection $\Psi$ between the regions of the Linial arrangement $\mA_{\{1\}}(3)$ and the trees in $\mT_{\{1\}}(3)$.}

In the case of the Catalan arrangement $\mA_{\{-1,0,1\}}(n)$ (and the generalization $\mA_{[-m..m]}(n)$) our bijection builds on a classical construction. In the case of the Shi arrangement $\mA_{\{0,1\}}(n)$ (and the generalization $\mA_{[-m+1..m]}(n)$) our bijection is a close relative to a bijection of Athanasiadis and Linusson~\cite{Athanasiadis:bijection-Shi}; see Section~\ref{sec:bij-link-Shi}. 
But already in the case of the Linial arrangement $\mA_{\{1\}}(n)$, no direct bijection was known between the regions of $\mA_{\{1\}}(n)$ and other combinatorial objects\footnote{The bijections in \cite{Corteel:hyperplane-valued-graphs,Forge:Linial} are not defined on the regions themselves, but rather on some combinatorial objects, called \emph{gain graphs without broken circuit}, which are known to be equinumerous to the regions by a non-bijective argument (Zaslavsky formula \cite{Zaslavsky:region-hyperplanes} allows to express the number of regions as a signed count of gain graphs, and after a suitable sign-reversing involution, one is left with the gain graphs without broken circuit).} (although several families of trees where known to be equinumerous to the regions of $\mA_{\{1\}}(n)$~\cite{Postnikov:coxeter-hyperplanes,Postnikov:intransitive-trees,Postnikov:Thesis}). We give short direct proofs of our bijective results in Section~\ref{sec:bij-m=1} in the case of the Catalan, Shi, semiorder and Linial arrangements. However, in the general case, we only prove surjectivity of our mapping and conclude to bijectivity by invoking the counting results established in the first part of the paper.

In Section~\ref{sec:conclusion}, we explain the relation between the trees described in Example~\ref{exp:Linial-intro} and the \emph{local binary search trees} which were known to be equinumerous to the regions of $\mA_{\{1\}}(n)$, and we conclude with some remarks and open questions.

\medskip

\section{Definitions and known results}\label{sec:definitions}
In this section we set our notation about arrangements and trees, and we recall some known counting results for the regions of the deformations of the braid arrangement.

\medskip
\subsection{Basic definitions.}~\\
A \emph{real hyperplane arrangement} in dimension $n$ is a set $\mA$ of affine hyperplanes of $\RR^n$. For instance, the braid arrangement $\mA_{\{0\}}(n)$ is the set $\{H_{i,j,0}\}_{1\leq i<j\leq n}$ of ${n \choose 2}$ hyperplanes. The \emph{regions of $\mA$} are connected components of $\ds \RR^n\setminus \bigcup_{H\in \mA}H$. 
We denote by $r_\mA$ the number of regions. For instance, it is easy to see that $r_{\mA_{\{0\}}(n)}=n!$.

We denote $\NN=\{0,1,2,\ldots\}$. For $a,b\in \ZZ$, we denote $[a..b]=\{i\in \ZZ~|~a\leq i\leq b\}$, and $[b]=[1..b]$. For a set $S$, we denote by $|S|$ the cardinality. For a ring $R$, we denote by $R[t]$ and $R[[t]]$ respectively the set of polynomials and formal power series in $t$ with coefficients in $R$. We extend the notation to several variables so that $R[y][[t_1,t_2]]$ is the set of formal power series in $t_1,t_2$ with coefficients in $R[y]$. For $G(t)\in R[[t]]$, we denote by $[t^k]G(t)$ the coefficient of $t^k$ in $G$.

\medskip
\subsection{Labeled plane trees.}\label{sec:trees}~\\
A \emph{tree} is a finite connected acyclic graph. A \emph{rooted plane tree} is a tree with a vertex distinguished as the \emph{root}, together with an ordering of the children of each vertex. A vertex in a rooted plane tree is a \emph{leaf} if it has no children, and a \emph{node} otherwise. 
We think of the children of a node $u$ of a rooted plane tree as being ``ordered from left to right'', and we adopt this convention in all our figures. For a child $v$ of $u$ we call \emph{left siblings of $v$} the children of $u$ (including leaves) which are on the left of $v$, that is, smaller than $v$ in the ordering of the children of $u$. 

We denote by $\mT$ the set of rooted plane trees with labeled nodes (if the tree has $n$ nodes, then the nodes have distinct labels in $[n]$, while the leaves are not labeled). We denote by $\mT^{(m)}$ the set of \emph{$(m+1)$-ary} trees in $\mT$ (i.e. the trees such that every node has $m+1$ children). We also denote by $\mT^{(m)}(n)$ the set of trees with $n$ nodes in $\mT^{(m)}$.
 A tree in $\mT^{(2)}(13)$ is represented in Figure~\ref{fig:boxed-tree}(a). 
For a non-root node $v$ of $T\in\mT$, we denote by $\parent(v)$ the parent of $v$, and $\ls(v)$ the number of left-siblings of $v$.
We compare nodes of $T$ according to their labels, so that $u<v$ means that the label of $u$ is less than the label of $v$.

\fig{width=.8\linewidth}{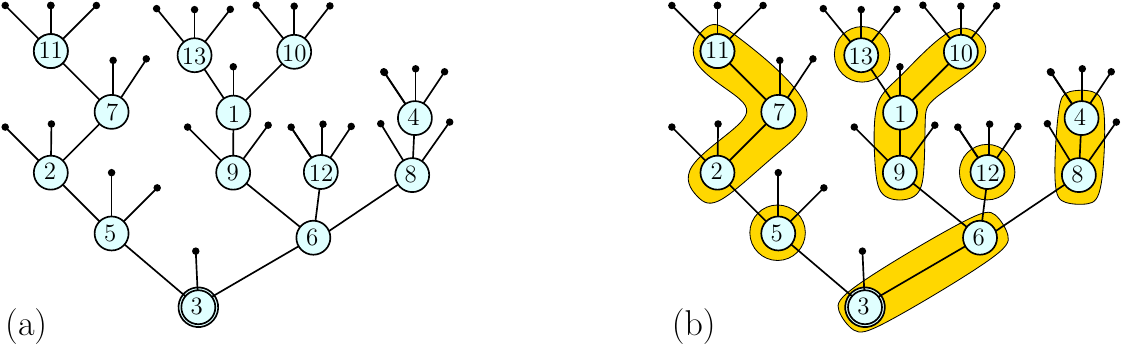}{(a) A (rooted plane node-labeled) tree in $\mT^{(2)}(13)$. (b) A $\SS$-boxed tree for $S=\{-1,2\}$ (note for instance that $-1\in S$ imposes that if a box contains both a node $u$ and its middle child $v$, then $u>v$).}

\medskip
\subsection{Known counting results about deformed braid arrangements.}\label{sec:known-arrangements}~\\ 
As mentioned in the introduction, it has been observed by Gessel, that for every set $S\subseteq \{-1,0,1\}$, the regions of the $S$-braid arrangement $\mA_S(n)$ are equinumerous to a certain family of trees in $\mT^{(1)}(n)$. Up to symmetry, we only need to consider the braid, Catalan, Shi, semiorder, and Linial arrangements, which are represented in Figure~\ref{fig:classical-arrangements}. We now describe the corresponding families of trees (see Figure~\ref{fig:condition-Shi-so-Linial}).
A non-root node $v$ of a tree $T\in\mT^{(1)}$ is called \emph{left node} (resp. \emph{right node}) if $\ls(v)=0$ (resp. $\ls(v)=1$). 
Here are the identities that Gessel observed (based on known counting results about hyperplane arrangements, and a new formula he established for trees in $\mT^{(1)}(n)$ counted according to the number of left and right ascents and descents):
\begin{itemize}
\item The regions of the \emph{Catalan arrangement} $\mA_{\{-1,0,1\}}(n)$ are equinumerous to the trees in $\mT^{(1)}(n)$.
\item The regions of the \emph{Shi arrangement} $\mA_{\{0,1\}}(n)$ are equinumerous to the trees in $\mT^{(1)}(n)$ such that
\begin{compactitem}
\item[(i)] for every right node $v$, $\parent(v)>v$. 
\end{compactitem}
\item The regions of the \emph{semiorder arrangement} $\mA_{\{-1,1\}}$ are equinumerous to the trees in $\mT^{(1)}(n)$ such that 
\begin{compactitem}
\item[(ii)] for every left node $v$, if the right-sibling of $v$ is a leaf then $\parent(v)>v$.
\end{compactitem}
\item The regions of the \emph{Linial arrangement} $\mA_{\{1\}}(n)$ are equinumerous to the trees in $\mT^{(1)}(n)$ such that 
\begin{compactitem}
\item[(iii)] for every left node $v$, $\parent(v)>v$, and for every right node $v$, $\parent(v)\!<\!v$.
\end{compactitem}
\end{itemize}

Based on these observations, Gessel raised the question of finding a uniform, possibly bijective, explanation of these five correspondences between arrangements and binary trees (see~\cite{Gessel:open-problems} or \cite[Section 1]{Gessel:labeled-trees-Shur-positive}). It is the goal of this paper to provide such an explanation (and more).
Let us mention however that the family of trees that will appear in the framework of the current paper for the Linial arrangement is not given by the Condition (iii), but instead by the Condition (iii') represented in Figure~\ref{fig:condition-Shi-so-Linial}. The bijective link between Conditions (iii) and (iii') is explained in Section~\ref{sec:bij-link-Linial}. In fact, Condition (iii') is simply the combination of Conditions (i) and (ii).  The fact that the family of trees associated to the intersection  $\mA_{\{1\}}(n)=\mA_{\{0,1\}}(n)\cap\mA_{\{-1,1\}}(n)$ is the intersection of the family of trees associated to $\mA_{\{0,1\}}(n)$ and $\mA_{\{-1,1\}}(n)$ is an instance of a general feature of the theory developed in the present paper (see Remark \ref{rk:intersection-mTS}). 

\fig{width=\linewidth}{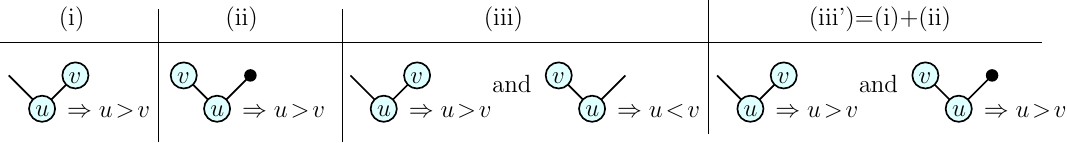}{The conditions (i), (ii), (iii) appearing in the literature for the classes of trees equinumerous to the regions of the Shi, semiorder and Linial arrangements. The characterization (iii') proved in this paper for the Linial arrangement appears to be new (see Section~\ref{sec:bij-link-Linial} for a bijection between (iii) and (iii')). Nodes are represented by labeled discs, while leaves are represented by black dots (here, the nature of some vertices is left unspecified).}

Let us now recall the relevant references for the identities observed above, and some natural generalizations.
 First, observe that $\mT^{(1)}(n)$ has cardinality $n!\Cat(n)!=\frac{(2n)!}{(n+1)!}$, because there are $\Cat(n)$ binary trees with $n$ nodes and $n!$ ways of labeling their nodes. More generally, $\mT^{(m)}(n)$ has cardinality $\frac{((m+1)n)!}{(mn+1)!}$. On the other hand, it is classical that the \emph{$m$-Catalan arrangement} $\mA_{[-m..m]}(n)$ has $\frac{((m+1)n)!}{(mn+1)!}$ regions (see e.g.~\cite[Section 4]{Stanley:hyperplane-interval-orders-overview}). We will recall a bijective proof of this fact in Section~\ref{sec:bij-prelim}. 

Next, observe that the number of trees in $\mT^{(1)}(n)$ satisfying Condition (i) is $(n+1)^{n-1}$ because these trees are easily seen to be in bijection with  Cayley trees with $n+1$ vertices. The fact that the number of regions of the Shi arrangement $\mA_{\{0,1\}}(n)$ is also $(n+1)^{n-1}$ was first established by Shi~\cite{Shi:nb-regions}. In fact, Shi further showed that for all $m$ the  number of regions of the $m$-Shi arrangement $\mA_{[-m+1..m]}(n)$ is  $(mn+1)^{n-1}$, which is  the number of $m$-parking functions of size $n$. Since then, at least two distinct bijective proofs of this fact have been given~\cite{Athanasiadis:bijection-Shi,Stanley:hyperplane-tree-inversions}, besides non-bijective proofs \cite{Postnikov:coxeter-hyperplanes,Corteel:hyperplane-valued-graphs,Athanasiadis98deformationsof}. We discuss these bijections further in Section \ref{sec:bij-link-Shi}. 
As we will see,  $(mn+1)^{n-1}$ is also the number of trees in $\mT^{(m)}(n)$ such that if a node $v$ is the right-most child of a node $u$, then $u>v$ (this generalizes (i)). 

The identity between the regions of the semiorder arrangement and the trees in $\mT^{(1)}(n)$ satisfying Condition (ii) is equivalent to a result of Chandon~\cite{Chandon:semiorder}.
More generally, the \emph{$m$-semiorder arrangement} $\mA_{[-m..m]\setminus \{0\}}(n)$ has regions equinumerous to trees in $\mT^{(m)}(n)$ such that if a node $v$ is the leftmost child of a node $u$, and all its siblings are leaves, then $u>v$ (this generalizes (ii)). This fact is easily deduced from the generating function equation given in~\cite[Theorem 7.1]{Postnikov:coxeter-hyperplanes}.

Lastly, the identity between the regions of the  Linial arrangement and the counting sequence of the trees in $\mT^{(1)}(n)$ satisfying (iii) was conjectured by Linial and Ravid, and proved in~\cite{Postnikov:coxeter-hyperplanes} and independently in~\cite{Athanasiadis:finite-field-method} by equating two generating functions. However, as mentioned above, the family of trees which appear naturally in our framework are characterized by Condition (iii') instead of Condition (iii). More generally,  the \emph{$m$-Linial arrangement} $\mA_{[-m+1..m]\setminus \{0\}}(n)$ has regions equinumerous to the subset of trees in $\mT^{(m)}(n)$ such that  if a node $v$ is the right-most child of a node $u$, then $u>v$, and moreover, if a node $v$ is the leftmost child of a node $u$, and all its siblings are leaves, then $u>v$ (this generalizes (iii')).


Although we cannot give an exhaustive bibliography about the enumerative study of deformations of the braid arrangement, we should mention a few additional references which are relevant to the present article. Formulas for the number of regions of several additional deformations of the braid arrangements (for instance $\mA_{[-\ell..m]}(n)$ for $\ell\geq -1$) are given in~\cite{Postnikov:coxeter-hyperplanes}.
The characteristic and coboundary polynomials of some of the arrangements above have been computed in~\cite{Athanasiadis98deformationsof,Ardila:Tutte-hyperplanes}. In a different direction, several deformations of the braid arrangements associated to a graph $G=([n],E)$ have been considered in the literature. 
The most classical is the \emph{$G$-graphical arrangement} made of the hyperplanes $H_{i,j,0}$, for all $\{i,j\}\in E$. Another important example is the \emph{$G$-Shi arrangement} considered for instance in~\cite{Armstrong:Shi-Ish,Athanasiadis:bijection-Shi}. This arrangement is made of the hyperplanes $H_{i,j,0}$ for all $i,j\in[n]$, and $H_{i,j,1}$ for all $\{i,j\}\in E$ with $i<j$ (so that it is the braid arrangement if $G$ has no edge, and the Shi arrangement if $G=K_n$). In yet another direction, several authors have considered deformed braid arrangements with hyperplanes $H_{i,j,s}$ for generic, non-integer values of $s$ (see e.g.~\cite{Postnikov:coxeter-hyperplanes,Stanley:hyperplane-tree-inversions,Stanley:hyperplane-arrangements,Hopkins:bigraphical-arrangements}), but we will not consider such situations here.



\medskip

\section{Counting regions of $S$-braid arrangements}\label{sec:regions-shaken}
In this section we present our counting results for the regions of $S$-braid arrangements.
Throughout this section, $S$ is a finite set of integers, $m=\max(|s|,~s\in S)$, and $n$ is a non-negative integer.

We start with the definition of \emph{$S$-boxed trees} (Definition \ref{def:boxed-trees}), and then express the number of regions of $\mA_S(n)$ as a signed count of $S$-boxed trees (Theorem \ref{thm:signed-count}). Then we restrict our attention to \emph{transitive sets} $S$ (Definition \ref{def:transitive}) and for them we express the number of regions of $\mA_S(n)$ as an unsigned count of trees (Theorem \ref{thm:unsigned-count}).

In order to define $S$-boxed trees, we first need to define $S$-cadet sequences.

\begin{definition} 
\begin{compactitem}
\item Let $T$ be a tree in $\mT$, and let $u$ be a node. If one of the children of $u$ is a node, we call the rightmost such child the \emph{cadet-node} of $u$, and denote it by $\cadet(u)$.\footnote{The term \emph{cadet} is used here in its genealogical meaning of \emph{youngest heir}.}
\item A \emph{cadet sequence} is a non-empty sequence $(v_1,\ldots,v_{k})$ of nodes such that for all $i$ in $[k-1]$, $v_{i+1}=\cadet(v_i)$.
\item A \emph{$S$-cadet sequence} is a cadet sequence $(v_1,\ldots,v_{k})$ such that for all $1\leq i<j\leq k$, if $\ds \sum_{p=i+1}^{j}\ls(v_p)\in S\cup\{0\}$ then $v_i<v_j$, and if $\ds -\sum_{p=i+1}^{j}\ls(v_p)\in S$ then $v_i>v_j$.
\end{compactitem}
\end{definition}

Note that an $S$-cadet sequence $(v_1,\ldots,v_{k})$ of $T\in\mT^{(m)}$ satisfies in particular $\ls(v_{j})\in[0..m]\setminus \{s\in S~|~-s\in S\}$ for all $j\in[2..k]$. \\

\begin{example} 
Let $T\in\mT^{(m)}$.
\begin{compactitem}
\item For $S=[-m..m]$, the $S$-cadet sequences of $T$ contain a single vertex.
\item For $S=[-\ell..m]$ with $0\leq \ell\leq m$, the $S$-cadet sequences of $T$ are the cadet sequences $(v_1,\ldots,v_k)$ satisfying $v_1<v_2<\cdots<v_k$ and $\ls(v_{p})\in [\ell+1..m]$ for all $p\in[2..k]$.
\item For $S=[-\ell..m]\setminus \{0\}$ with $0\leq \ell\leq m$, the $S$-cadet sequences of $T$ are the cadet sequences $(v_1,\ldots,v_k)$ satisfying $v_1<v_2<\cdots<v_k$ and $\ls(v_{p})\in\{0\}\cup [\ell+1..m]$ for all $p\in[2..k]$.
\item For $S=\{-2,0,1,2\}$, the $S$-cadet sequences of $T$ have size at most 2, and the $S$-cadet sequences of size 2 are of the form $(v_1,v_2)$ with $\ls(v_2)=1$ and $v_1<v_2$.
\end{compactitem}
\end{example}

\begin{definition} \label{def:boxed-trees}
A \emph{boxed tree} is a pair $(T,B)$, where $T$ is in $\mT$, and $B$ is a set of cadet sequences partitioning the set of nodes of $T$ (that is, every node of $T$ is contained in exactly one cadet sequence in $B$). The pair $(T,B)$ is an \emph{$S$-boxed tree} if $T\in\mT^{(m)}$, and $B$ contains only $S$-cadet sequences. We denote by $\mU_S(n)$ the set of $S$-boxed trees with $n$ nodes.
\end{definition}

We represent boxed trees as trees decorated with boxes partitioning the nodes into cadet sequences, as in Figure~\ref{fig:boxed-tree}(b). We can now state the main result of this section.

\begin{thm}\label{thm:signed-count}
Let $S$ be a finite set of integers and $n$ be a positive integer. 
The number of regions of the hyperplane arrangement $\mA_S(n)$ is
\begin{equation}\label{eq:signed-count}
r_S(n)=\sum_{(T,B)\in \mU_S(n)}(-1)^{n-|B|}.
\end{equation}
\end{thm}

The proof of Theorem~\ref{thm:signed-count} is delayed to Section~\ref{sec:proof}. We will now give a simpler expression for $r_S(n)$ in the cases where the set $S$ is ``well-behaved''. 
More precisely, we now introduce the notion of \emph{transitivity} for a set $S$, which implies a drastic simplification of the definition of $S$-cadet-sequences (Lemma \ref{lem:S-config-when-transitive}), and allows one to define a simple sign reversing involution on $S$-boxed trees.

\begin{definition}\label{def:transitive}
A set $S$ of integers is called \emph{transitive} if it satisfies the following conditions for all integers $s,t\notin S$: 
\begin{compactitem}
\item if $st>0$, then $s+t\notin S$,
\item if $s>0$ and $t\leq 0$, then $s-t\notin S$ and $t-s\notin S$.
\end{compactitem}
\end{definition}

\begin{example} 
\begin{compactitem}
\item All the subsets of $\{-1,0,1\}$ are transitive. 
\item All the intervals of integers containing 1 are transitive. 
\item Sets of the form $S=I\setminus k\ZZ$, where $I$ is an interval containing $1$ are transitive. 
\item Sets $S$ such that $[-\lfloor m/2\rfloor..\lfloor m/2\rfloor ]\subseteq S \subseteq [-m..m]$ for some $m$ are transitive. 
\item A set $S$ such that $\{-s,~s\in S\}=S$ is transitive if and only if the set of positive integers not in $S$ is closed under addition (equivalently, 0 together with the positive integer not in $S$ form what is called a \emph{numerical semigroup}; see \cite{Assi-Garcia-Sanchez:transitive-sets} for references on numerical semigroups). 
\item A set $S$ such that $[m]\subseteq S\subseteq [-m..m]$ is transitive if and only if the set of negative integers not in $S$ is closed under addition.
\end{compactitem}

\end{example}

\begin{definition}\label{def:mTS}
We denote by $\mT_S(n)$ the set of trees $T$ in $\mT^{(m)}(n)$ such that all nodes $u,v$ satisfying $\cadet(u)=v$ further satisfies the following:
\begin{compactitem} 
\item[Condition($S$):] if $\ls(v)\notin S\cup\{0\}$ then $u<v$, and if $-\ls(v)\notin S$ then $u>v$. 
\end{compactitem}
\end{definition}

\begin{thm}\label{thm:unsigned-count}
If $S$ is transitive, then regions of the hyperplane arrangement $\mA_S(n)$ are equinumerous to the trees in $\mT_S(n)$.
\end{thm}

\begin{example}
\begin{itemize} 
\item $\mT_{[-m..m]}(n)=\mT^{(m)}(n)$.
\item $\mT_{[-m+1..m]}(n)$ is the set of trees in $\mT^{(m)}(n)$, such that any non-root node having no right-sibling (not even leaves) is less than its parent.
\item $\mT_{[m]}(n)$ is the set of trees in $\mT^{(m)}(n)$, such that such that any cadet-node $v$ is less than its parent.
\item More generally, for $0\leq \ell\leq m$, $\mT_{[-\ell..m]}(n)$ is the set of trees in $\mT^{(m)}(n)$, such that any cadet-node $v$ having more than $\ell$ left-sibling is less than its parent. And $\mT_{[-\ell..m]\setminus \{0\}}(n)$ is the set of trees in $\mT^{(m)}(n)$, such that any cadet-node $v$ having either no left sibling or more than $\ell$ left-siblings is less than its parent.
\end{itemize}
\end{example}

\begin{remark}\label{rk:intersection-mTS}
For any sets $S,S'\subset \ZZ$, $\mT_S(n)\cap\mT_{S'}(n)=\mT_{S\cap S'}(n)$. For instance the set $\mT_{\{1\}}(n)$ of trees associated to Linial arrangement, is the intersection of the set of trees $\mT_{\{0,1\}}(n)$ associated to the Shi arrangement, and the set of trees $\mT_{\{-1,1\}}(n)$ associated to the semiorder arrangement. 

Moreover, each element $s\in[-m..m]\setminus S$ gives a simple condition for trees in $\mT_S(n)$: for $s>0$ the condition is that a cadet node with $s$ left siblings is greater than its parent, while for $s\leq 0$ the condition is that a cadet node with $-s$ left siblings is less than its parent. This is represented in Figure~\ref{fig:condition-s}.
\end{remark}

\fig{width=.6\linewidth}{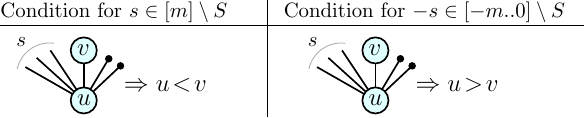}{Conditions for trees to be in $\mT_S(n)$. Each element $s\in[-m..m]\setminus S$ imposes one condition.}

Theorem~\ref{thm:unsigned-count} is an easy consequence of Theorem~\ref{thm:signed-count} and the following lemma.

\begin{lemma}\label{lem:S-config-when-transitive}
Suppose that the set $S$ is transitive. In this case, a cadet sequence $(v_1,\ldots,v_k)$ is an $S$-cadet sequence if and only if 
for all $i\in [k-1]$, 
\begin{compactitem}
\item[(*)] if $\ds \ls(v_{i+1})\in S\cup\{0\}$ then $v_i<v_{i+1}$, and if $\ds -\ls(v_{i+1})\in S$ then $v_i>v_{i+1}$.
\end{compactitem}
\end{lemma}

\begin{proof}
It is clear that the condition (*) is necessary. We now prove that it is sufficient, by induction on $k$. The case $k=1$ is trivial. Now suppose that $k>1$ and $\ga=(v_1,\ldots,v_k)$ satisfies (*). Since $\ga'=(v_1,\ldots,v_{k-1})$ satisfies (*), it is an $S$-cadet sequence. Hence we only need to check that for all $i\in [k-1]$, 
\begin{compactitem}
\item[\hspace{-15mm} (**)] if $\ds \sum_{p=i+1}^{k}\ls(v_p)\!\in\! S\!\cup\!\{0\}$ then $v_i\!<\!v_k$, and if $\ds -\!\sum_{p=i+1}^{k}\!\ls(v_p)\!\in\! S$ then $v_i\!>\!v_k$. 
\end{compactitem}
The case $i=k-1$ of (**) is directly given by (*). We now consider $i\in[k-2]$, and consider several cases. Suppose first that $v_i<v_{k-1}<v_k$. In this case, $-\sum_{p=i+1}^{k-1}\ls(v_p)\notin S$ (since $\ga'$ is an $S$-cadet sequence), $-\ls(v_k)\notin S$ (since $\ga$ satisfies (*)), hence $-\sum_{p=i+1}^{k}\ls(v_p)\notin S$ (since $S$ is transitive), hence Condition (**) holds for $i$. The case $v_i>v_{k-1}>v_k$ is treated similarly.
Suppose next that $v_i>v_{k-1}<v_k$. In this case, $\sum_{p=i+1}^{k-1}\ls(v_p)\notin S\cup \{0\}$ (since $\ga'$ is an $S$-cadet sequence), $-\ls(k)\notin S$ (since $\ga$ satisfies (*)), hence $\sum_{p=i+1}^{k}\ls(v_p)\notin S\cup \{0\}$ and $-\sum_{p=i+1}^{k}\ls(v_p)\notin S$ (since $S$ is transitive), hence Condition (**) holds for $i$. The case $v_i<v_{k-1}>v_k$ is treated similarly. Thus Condition (**) holds for all $i$, and $\ga$ is an $S$-cadet sequence.
\end{proof}

\begin{proof}[Proof of Theorem~\ref{thm:unsigned-count}]
Let $T\in \mT^{(m)}(n)$, and let $v=\cadet(u)$. We claim that that $u$ and $v$ can be in the same box of an $S$-boxed tree $(T,B)$ if and only if $u$ and $v$ do not satisfy Condition($S$) of Definition~\ref{def:mTS}. 
Indeed, by Lemma~\ref{lem:S-config-when-transitive}, in the case $u<v$ (resp. $u>v$) the vertices $u$ and $v$ can be in the same box if and only if $-\ls(v)\notin S$ (resp. $\ls(v)\notin S\cup\{0\}$), and this holds if and only if Condition($S$) does not hold. 

For a tree $T$ in $\mT^{(m)}(n)$, we denote $\mB_{T}=\{B~|~(T,B)\in\mU_S(n)\}$. By Theorem~\ref{thm:signed-count}, 
\begin{equation}\label{eq:two-sums}
r_S(n)=\sum_{T\in \mT_S(n)} ~\sum_{B\in \mB_{T}}(-1)^{n-|B|} ~+~\sum_{T\in \mT^{(m)}(n)\setminus \mT_S(n)}~ \sum_{B\in \mB_{T}}(-1)^{n-|B|}.
\end{equation}
By the above claim, for all $T$ in $\mT_S$, $\mB_{T}$ contain a single element because every node of $T$ must be in a different box. Thus the first sum of~\eqref{eq:two-sums} contributes $|\mT_S(n)|$. We now prove that the second sum is 0 using a sign reversing involution. 
For a tree $T\in \mT^{(m)}(n)\setminus\mT_S(n)$, we pick the smallest vertex $v=\cadet(u)$ such that Condition($S$) does not hold, and define an involution $\varphi$ on $\mB_{T}$ as follows: 
\begin{compactitem}
\item if $u$ and $v$ are in the same box of $B$, then $\varphi(B)$ is obtained by splitting the box containing them between $u$ and $v$, 
\item if $u$ and $v$ are in different boxes of $B$, then $\varphi(B)$ is obtained by merging these boxes.
\end{compactitem}
Lemma~\ref{lem:S-config-when-transitive} ensures that $\varphi(B)\in\mB_T$ in the second situation. Since $\varphi$ is an involution on $\mB_{T}$ changing the number of boxes by $\pm 1$, we get $ \sum_{B\in \mB_{T}}(-1)^{n-|B|}\!=\!0$.
Hence the second sum in~\eqref{eq:two-sums} contributes 0.
\end{proof}


\medskip

\section{General deformations of the braid arrangement}\label{sec:multi-shaken}
In this section we extend the results of Section~\ref{sec:regions-shaken} to general deformations of the braid arrangement. We fix a positive integer $N$ and an ${N\choose 2}$-tuple of finite sets of integers $\SS=(S_{a,b})_{1\leq a<b\leq N}$. 
The arrangement in $\RR^N$ made of the hyperplanes 
$$H_{a,b,s}=\{(x_1,\ldots,x_N)\in \RR^N ~|~ x_a-x_b=s\},$$ 
for all $1\leq a<b\leq N$ and all $s\in S_{a,b}$ is called \emph{$\SS$-braid arrangement}, and is denoted $\mA_{\SS}$. Note that if $S_{a,b}=S$ for all $a,b$, then $\mA_\SS=\mA_S(N)$.

We will now extend Theorem~\ref{thm:signed-count} to $\SS$-braid arrangements. 
Let $m=\max(|s|,~s\in \cup S_{a,b})$. For $1\leq a<b\leq N$, we denote $S_{b,a}=S_{a,b}$, $S_{a,b}^-=\{s\geq 0~|~-s\in S_{a,b}\}$, and $S_{b,a}^-=\{s>0~|~s\in S_{a,b}\}\cup \{0\}$.

\begin{definition} 
A cadet sequence $(v_1,\ldots,v_{k})$ of $T\in\mT^{(m)}(N)$ is an \emph{$\SS$-cadet sequence} if for all $1\leq i<j\leq k$, $\ds \sum_{p=i+1}^{j}\ls(v_p)\notin S_{v_i,v_j}^-$.
An \emph{$\SS$-boxed tree} is a boxed tree $(T,B)$ with $T\in\mT^{(m)}(N)$, and $B$ containing only $\SS$-cadet sequences. We denote by $\mU_\SS$ the set of $\SS$-boxed trees.
\end{definition}

\begin{thm}\label{thm:signed-count-multi}
The number of regions of the hyperplane arrangement $\mA_\SS$ is
\begin{equation}\label{eq:signed-count-multi}
r_\SS=\sum_{(T,B)\in \mU_\SS}(-1)^{n-|B|}.
\end{equation}
\end{thm}

The condition $ \sum_{p=i+1}^{j}\ls(v_p)\notin S_{v_i,v_j}^-$ is equivalent to: if $ \sum_{p=i+1}^{j}\ls(v_p)\in S_{v_i,v_j}\cup \{0\}$ then $v_i<v_j$, and if $ -\! \sum_{p=i+1}^{j}\ls(v_p)\in S_{v_i,v_j}$ then $v_i>v_j$. In particular, if $S_{a,b}=S$ for all $a,b$, then $\mU_\SS=\mU_S(N)$. Hence Theorem~\ref{thm:signed-count-multi} generalizes Theorem~\ref{thm:signed-count}.

Theorem~\ref{thm:signed-count-multi} will be extended in the next section and its proof is delayed to Section~\ref{sec:proof}.
We now generalize Theorem~\ref{thm:unsigned-count} to $\SS$-braid arrangements.

\begin{definition}\label{def:transitive-multi}
The tuple $\SS$ is said \emph{transitive} if for all distinct integers $a,b,c\in[N]$ the following condition holds: if $s\notin S_{a,b}^-$ and $t\notin S_{b,c}^-$, then $s+t\notin S_{a,c}^-$.
\end{definition}

It is easy to see that if $S_{a,b}=S$ for all $a,b\in [N]$, then $\SS$ is transitive if and only if $S$ is transitive. 

\begin{example} \label{exp:transitive-multi}
If for all $a,b\in [N]$, $[-\lfloor m/2 \rfloor \,..\,\lfloor m/2\rfloor] \subseteq S_{a,b}\subseteq [-m..m]$, then $\SS$ is transitive.
\end{example}

\begin{definition}\label{def:mTSS}
We denote by $\mT_\SS$ the set of trees $T$ in $\mT^{(m)}(N)$ such that any pair of nodes $u,v$ such that $\cadet(u)=v$ satisfies $\ls(v)\in S_{u,v}^-$.
\end{definition}

\begin{thm}\label{thm:unsigned-count-multi}
If $\SS=(S_{a,b})_{1\leq a<b\leq N}$ is transitive, then the regions of $\mA_\SS$ are equinumerous to the trees in $\mT_\SS$.
\end{thm}

\begin{remark}
The condition $\ls(v)\notin S_{u,v}^-$ is equivalent to: if $\ls(v)\notin S_{u,v}\cup \{0\}$ then $u<v$, and if $\ds -\ls(v)\in S_{u,v}$ then $u>v$. In particular, if $S_{a,b}=S$ for all $a,b\in[N]$, then $\mT_\SS=\mT_S(N)$. Hence Theorem~\ref{thm:unsigned-count-multi} generalizes Theorem~\ref{thm:unsigned-count}.
\end{remark}

\begin{example}
Let $G=([N],E)$ be a graph and let $S$, $S'$ be two finite sets of integers. Let $G(S,S')$ be the tuple $\SS=(S_{a,b})_{1\leq a<b\leq N}$ defined by $S_{a,b}=S$ if $\{a,b\}\in E$ and $S_{a,b}=S'$ otherwise. Several cases are represented in Figure~\ref{fig:trans-graphical-arrangements}.
\begin{enumerate}
\item For $S=\{-1,0,1\}$ and $S'=\{0,1\}$, the tuple $\SS=G(S,S')$ is transitive for any graph $G$, and $\mT_\SS$ is the set of trees in $\mT^{(1)}(N)$ such that if a node $v$ is the right child of $u$, then either $\{u,v\}\in E$ or $u>v$ (or both). 
\item For $S=\{-1,0,1\}$ and $S'=\{0\}$, the tuple $\SS=G(S,S')$ is transitive for any graph $G$, and $\mT_\SS$ is the set of trees in $\mT^{(1)}(N)$ such that if a node $v$ is the right child of $u$, then $\{u,v\}\in E$. 
\item For $S=\{0,1\}$ and $S'=\{0\}$, the tuple $\SS=G(S,S')$ is transitive for any graph $G$, and $\mT_\SS$ is the set of trees in $\mT^{(1)}(N)$ such that if a node $v$ is the right child of $u$, then $\{u,v\}\in E$ and $u>v$. 
\item For $S=\{0,1\}$ and $S'=\{-1,0\}$, the tuple $\SS=G(S,S')$ is transitive for any graph $G$, and $\mT_\SS$ is the set of trees in $\mT^{(1)}(N)$ such that if a node $v$ is the right child of $u$, then either ($\{u,v\}\in E$ and $u>v$) or ($\{u,v\}\notin E$ and $u<v$). 
\end{enumerate}
\end{example}

\fig{width=\linewidth}{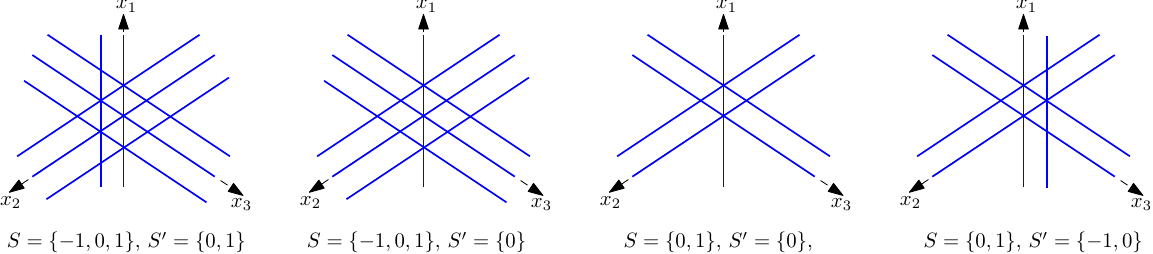}{Some transitive deformations of the braid arrangement. These arrangements have the form $A_{G(S,S')}$, where $G$ is the graph having vertex set $[3]$ and edges $\{1,2\}$ and $\{1,3\}$.}

\begin{proof}[Proof of Theorem~\ref{thm:unsigned-count-multi}]
Given Theorem~\ref{thm:signed-count-multi}, we only need to prove that if $\SS=(S_{a,b})_{1\leq a<b\leq n}$ is transitive, then 
\begin{equation}\label{eq:signed-to-unsigned}
\sum_{(T,B)\in \mU_\SS}(-1)^{|B|}=|\mT_\SS|.
\end{equation}
The proof of \eqref{eq:signed-to-unsigned} is almost identical to that of Theorem~\ref{thm:unsigned-count}, except that Lemma~\ref{lem:S-config-when-transitive} is replaced by the following claim: \emph{if $\SS$ is transitive, a cadet sequence $(v_1,\ldots,v_k)$ of $T\in\mT^{(m)}(N)$ is a $\SS$-cadet sequence if and only if for all $i\in[k-1]$, $\ls(v_{i+1})\notin S_{v_{i},v_{i+1}}^-$.}\\
\ni \textbf{Proof of the claim:} It is clear that $\ls(v_{i+1})\notin S_{v_{i},v_{i+1}}^-$ is necessary. We now prove that it is sufficient, by induction on $k$. The case $k=1$ is trivial. Now suppose that $k>1$, and $\ga=(v_1,\ldots,v_k)$ is cadet sequence such that for all $i\in[k-1]$, $\ls(v_{i+1})\notin S_{v_{i},v_{i+1}}^-$. We want to prove that $\ga$ is a $\SS$-cadet sequence. By the induction hypothesis, $\ga'=(v_1,\ldots,v_{k-1})$ is a $\SS$-cadet sequence, so we only need to prove that for all $i\in[k-1]$, $ \sum_{p=i+1}^{k}\ls(v_p)\notin S_{v_{i},v_{k}}^-$. This is true by hypothesis for $i=k-1$. Moreover, for $i\in[k-1]$, $\sum_{p=i+1}^{k-1}\ls(v_p)\notin S_{i,k-1}^-$ (since $\ga'$ is a $\SS$-cadet sequence), and $\ls(v_k)\notin S_{k-1,k}^-$ (by hypothesis), hence $\sum_{p=i+1}^{k}\ls(v_p)\notin S_{i,k}^-$ (since $\SS$ is transitive). Hence $\ga$ is a $\SS$-cadet sequence. This proves the claim.\\
One can then define a sign reversing involution on $\SS$-boxed trees showing \eqref{eq:signed-to-unsigned}, exactly as in the proof of Theorem~\ref{thm:unsigned-count}.
\end{proof}


\medskip

\section{Characteristic and coboundary polynomials of the deformations of the braid arrangement}\label{sec:char-poly}
In this section we refine our preceding counting results by expressing the characteristic and coboundary polynomials of deformed braid arrangements in terms of boxed trees.

\subsection{Characteristic and coboundary polynomials.}~\\
For a hyperplane arrangement $\mA\subset \RR^n$, we denote by $\chi_\mA(q)$ its \emph{characteristic polynomial} of $\mA$, and by $P_\mA(q,y)$ its \emph{coboundary polynomial}. 
Recall from \cite{Crapo:coboundary-poly} that the coboundary polynomial is defined by 
$$P_\mA(q,y)=\sum_{\mB\subseteq \mA,~\cap_{H\in\mB}H\neq\emptyset}q^{\dim(\cap_{H\in\mB}H)}(y-1)^{|\mB|},$$
and that $\chi_\mA(q)=P_\mA(q,0)$.

The characteristic polynomial $\chi_\mA$ contains a lot of information about the arrangement~$\mA$. In particular, by a result of Zaslavsky~\cite{Zaslavsky:region-hyperplanes}, the number of regions $r_\mA$ and the number of \emph{relatively bounded}\footnote{A region of $\mA$ is \emph{relatively bounded} if its intersection with the subspace generated by the vectors normal to the hyperplanes of $\mA$ is bounded.} regions $b_\mA$ are evaluations of $\chi_\mA(q)$: 
\begin{eqnarray}
r_\mA&=&(-1)^n\chi_\mA(-1),\label{eq:face-characteristic}\\
b_\mA&=&(-1)^{\rank(\mA)}\chi_\mA(1),\nonumber
\end{eqnarray}
where $\rank(\mA)$ is the dimension of the vector space generated by the vectors normal to the hyperplanes of $\mA$.
The characteristic polynomial is also equivalent to the Poincar\'e polynomial of the cohomology ring of the complexification of $\mA$; see~\cite{Orlik:Poincare-hyperplanes}. 
The coboundary polynomial is equivalent to the Tutte polynomial $T_{\mA}(x,y)$ of $\mA$ (that is, the Tutte polynomial of the semi-matroid associated with $\mA$, in the sense of Ardila~\cite{Ardila:Tutte-hyperplanes,Ardila:Semimatroids}): 
$$T_\mA(x,y)=(y-1)^{-\rank(\mA)}P_\mA((x-1)(y-1),y).$$

\medskip
\subsection{Expressing the coboundary polynomial in terms of boxed trees.}~\\
In order to express the coboundary polynomials of deformed braid arrangements in terms of boxed trees, we will consider arrangements of all dimensions.
Let $\hSS=(S_{a,b})_{1\leq a\leq b\leq N}$ be an ${N+1 \choose 2}$-tuple of finite sets, let $m=\max(|s|,~s\in \cup S_{a,b})$, and let $\nn=(n_1,\ldots,n_N)\in\NN^N$. 
We denote $|\nn|=n_1+\ldots+n_N$, and 
$$V(\nn)=\{(a,i)~|~a\in [N],i\in[n_a]\}.$$ 
We endow $V(\nn)$ with the \emph{lexicographical order}, that is, we denote $(a,i)<(b,j)$ if either $a<b$, or $a=b$ and $i<j$.
For $u=(a,i)$ and $v=(b,j)\in V(\nn)$, we denote $S_{u,v}=S_{v,u}=S_{a,b}$, and if $u<v$ we denote $S_{u,v}^-=\{s\geq 0~|~-s\in S_{a,b}\}$ and $S_{v,u}^-=\{s> 0~|~s\in S_{a,b}\}\cup\{0\}$.
Lastly, we define $\mA_\hSS(\nn)$ as the arrangement in $\RR^{|\nn|}$ with hyperplanes 
$$H_{u,v,s}=\{(x_{w})_{w\in V(\nn)}~|~x_{u}-x_{v}=s\},$$
for all $u<v$ in $V(\nn)$ and all $s\in S_{u,v}$.

Note that $\mA_\hSS(\nn)$ identifies with the arrangement $\mA_{\hSS(\nn)}$, where 
\begin{equation}\label{eq:SSnn}
\hSS(\nn)=(S'_{u,v})_{1\leq u<v \leq |\nn|}
\end{equation}
with $S_{u,v}'=S_{a,b}$ for all $\ds u\in \left[1+\sum_{i=1}^{a-1}n_i \,..\, \sum_{i=1}^{a}n_i\right]$ and $\ds v\in \left[1+\sum_{i=1}^{b-1}n_i \,..\, \sum_{i=1}^{b}n_i\right]$.
For instance, $\mA_\hSS(1,1,\ldots,1)= \mA_\hSS$, and $\mA_\hSS(n_1,0,\ldots,0)=\mA_{S_{1,1}}(n_1)$. 
We now describe boxed trees related to the arrangement $\mA_\hSS(\nn)$.

\begin{compactitem}
\item We denote by $\mT^{(m)}(\nn)$ the set of rooted plane $(m+1)$-ary trees with $|\nn|$ nodes labeled with distinct labels in $V(\nn)$. 
\item A cadet sequence $(v_1,\ldots,v_{k})$ of $T\in\mT^{(m)}(\nn)$ is \emph{admissible} if $v_i<v_{i+1}$ for all $i\in[k-1]$ such that $\ls(v_{i+1})=0$. A boxed tree $(T,B)$ is \emph{admissible} if all the sequences in $B$ are admissible. We denote by $\mU^{(m)}(\nn)$ the set of admissible boxed trees $(T,B)$ with $T\in \mT^{(m)}(\nn)$.
\item The \emph{$\hSS$-energy} of a cadet sequence $(v_1,\ldots,v_k)$ of $T$ is the number of pairs $\{i,j\}$ with $1\leq i<j\leq k$, such that $\sum_{p=i+1}^{j}\ls(v_p)\in S_{v_i,v_j}^-$.\footnote{The terminology \emph{energy} used here is related to the interpretation of the current counting problem (about the coboundary polynomial of arrangements) in terms of a gas model which will be introduced in Section \ref{sec:proof}.}
The \emph{$\hSS$-energy} of a boxed tree $(T,B)\in\mU^{(m)}(\nn)$, denoted $\energy_\hSS(T,B)$, is the sum of the energies of the cadet sequences in $B$.
\item We denote by $\mU_\hSS(\nn)$ the set of boxed trees $(T,B)\in \mU^{(m)}(\nn)$ such that $\energy_\hSS(T,B)=0$.
\end{compactitem}

Note that any boxed tree in $\mU_\hSS(\nn)$ is admissible. Moreover, $\mU_\hSS(n_1,0,\ldots,0)=\mU_{S_{1,1}}(n_1)$, and $\mU_\hSS(1,1,\ldots,1)=\mU_\SS$, where $\SS=(S_{a,b})_{1\leq a<b\leq N}$.

We denote by $\mT_\hSS(\nn)$ the set of trees $T$ in $\mT^{(m)}(\nn)$ such that any pair of nodes $u,v$ such that $\cadet(u)=v$ satisfies $\ls(v)\in S_{u,v}^-$. Note that $\mT_\hSS(n_1,0,\ldots,0)=\mT_{S_{1,1}}(n_1)$, and $\mT_\hSS(1,1,\ldots,1)=\mT_\SS$, where $\SS=(S_{a,b})_{1\leq a<b\leq N}$. We say that $\hSS$ is \emph{multi-transitive} if $\hSS(\nn)$ is transitive (in the sense of Definition~\ref{def:transitive-multi}) for all $\nn\in\NN^N$. Note that for $N=1$ $\hSS$ is multi-transitive if and only if $S_{1,1}$ is transitive.

\begin{example} If for all $a,b\in [N]$, $[-\lfloor m/2 \rfloor \,..\,\lfloor m/2\rfloor] \subseteq S_{a,b}\subseteq [-m..m]$, then $\hSS$ is multi-transitive. Also, if $S_{a,a}$ is transitive for all $a\in[N]$, and $S_{a,b}=[-m..m]$ for all $a<b$, then $\hSS$ is multi-transitive. 
\end{example}

We can now express the coboundary polynomial of the arrangements. 
Given indeterminates $t_1,\ldots,t_N$, we denote $\tt=(t_1,\ldots,t_N)$, $\tt^\nn=\prod_{a=1}^N t_a^{n_a}$, and $\nn!=\prod_{a=1}^N n_a!$. 
We denote 
\begin{eqnarray}
P_\hSS(q,y,\tt)&=&\sum_{\nn\in\NN^N}P_{\mA_\hSS(\nn)}(q,y)\frac{\tt^\nn}{\nn!},\label{eq:def-gfP}\\
\chi_\hSS(q,\tt)&=&\sum_{\nn\in\NN^N}\chi_{\mA_\hSS(\nn)}(q)\frac{\tt^\nn}{\nn!},\label{eq:def-gfchi}\\
R_\hSS(\tt)&=&\sum_{\nn\in\NN^N}r_{\mA_\hSS(\nn)}\frac{\tt^\nn}{\nn!},\label{eq:def-gfR}
\end{eqnarray}
In the above definition, we adopt the convention $P_{\mA_\hSS(0,\ldots,0)}(q,y)=1$ (coboundary polynomial of the empty semi-matroid). By~\eqref{eq:face-characteristic}, $\chi_\hSS(q,\tt)=P_\hSS(q,0,\tt)$ and $R_\hSS(\tt)=\chi_\hSS(-1,-\tt)$, where $-\tt=(-t_1,\ldots,-t_N)$.

\begin{thm}\label{thm:energy-count-multi}
Let $\hSS=(S_{a,b})_{1\leq a\leq b\leq N}$ be an ${N+1\choose 2}$-tuple of finite sets of integers, and let $m=\max(|s|,~s\in \cup S_{a,b})$.
Then $P_\hSS(q,y,\tt)$ is related to boxed trees by 
\begin{equation}\label{eq:energy-count-multi}
P_\hSS(q,y,\tt)=\left(\sum_{\nn\in\NN^N}\frac{\tt^\nn}{\nn!}\sum_{(T,B)\in \mU^{(m)}(\nn)}(-1)^{|B|}y^{\energy_\hSS(T,B)}\right)^{-q}.
\end{equation}
In particular, 
\begin{equation}\label{eq:q-is-exponent}
\ds \chi_\hSS(q,\tt)=R_\hSS(-\tt)^{-q},
\end{equation}
and 
\begin{equation}\label{eq:GF-regions-multi}
R_\hSS(\tt)=\sum_{\nn\in\NN^N}\frac{\tt^\nn}{\nn!}\sum_{(T,B)\in \mU_\hSS(\nn)}(-1)^{|\nn|-|B|}.
\end{equation}
Moreover, if $\hSS$ is multi-transitive, then
\begin{equation}\label{eq:GF-regions-transitive-multi}
R_\hSS(\tt)=\sum_{\nn\in\NN^N}\frac{\tt^\nn}{\nn!}|\mT_\hSS(\nn)|.
\end{equation}
\end{thm}

Note that Equation \eqref{eq:GF-regions-multi} implies Theorem~\ref{thm:signed-count} (for $S=S_{1,1}$) by extracting the coefficient of $t_1^n t_2^0\ldots t_N^0$, and Theorem~\ref{thm:signed-count-multi} by extracting the coefficient of $t_1t_2\cdots t_N$. Several applications of Theorem \ref{thm:energy-count-multi} are given in Section~\ref{sec:GF}.

\medskip
\subsection{Remarks about the case of graphical arrangements (case $m=0$).}~\\ 
In this subsection we consider the special case $m=0$ of Theorem \ref{thm:energy-count-multi} which corresponds to graphical arrangements, in order to highlight how our results are related to some known results about graph colorings and acyclic orientations.

Remember that the \emph{$G$-graphical arrangement} associated to a (simple, undirected) graph $G=([n],E)$ is the $n$-dimensional arrangement $\mA(G)$ made of the hyperplanes $H_{i,j,0}$ for all $\{i,j\}\in E$. It follows easily from the definitions that the number of regions of $\mA(G)$ is the number of acyclic orientations of~$G$\footnote{In this natural correspondence, the direction of the edge $(i,j)$ indicates which side of the hyperplanes $H_{i,j,0}$ the region is, or equivalently the inequality between the coordinates $x_i$ and $x_j$.}. Also, as we now recall, the characteristic and coboundary polynomials of $\mA(G)$ are related to the colorings of $G$. Recall that the \emph{partition function of the Potts model} on $G$, is the unique bivariate polynomial $P_G(q,y)$ such that for all positive integer $q$,
$$P_{G}(q,y)=\sum_{f:V\to [q]}y^{\textrm{mono}(f)},$$
where the sum is over all possible colorings of the vertices in $q$ colors, and $\textrm{mono}(f)$ is the number of edges of $G$ with both endpoints of the same color. 
The specialization $\chi_{G}(q)=P_G(q,0)$ counting the \emph{proper colorings} of $G$ in $q$ colors is called \emph{chromatic polynomial} of $G$. It is known \cite{Crapo:coboundary-poly} that for any graph $G$, the coboundary polynomial $P_{\mA(G)}(q,y)$ is equal to $P_{G}(q,y)$, and in particular the characteristic polynomial $\chi_{\mA(G)}(q)$ is equal to the chromatic polynomial $\chi_{G}(q)$.

We can now interpret the case $m=0$ of Theorem \ref{thm:energy-count-multi} in terms of graphs. Let $\hSS=(S_{a,b})_{1\leq a\leq b\leq N}$ be a tuple such that each set $S_{a,b}$ is either empty or equal to $\{0\}$. Note that $\mA_{\hSS(\nn)}$ is the $G$-graphical arrangement for the graph $G=G_{\hSS}(\nn)$ with vertex set $V(\nn)$ and edges $\{u,v\}$ for all $u=(a,i)$ and $v=(b,j)$ such that $S_{a,b}=\{0\}$. 
Moreover, the boxed trees in $\mU^{(0)}(\nn)$ are simply paths with boxes partitioning the vertices, such that in each box the vertices are in increasing order. Hence $\mU^{(0)}(\nn)$ can simply be interpreted as the set of \emph{ordered set partitions} of $V(\nn)$. Lastly, the $\hSS$-energy of a subset $U$ of $V(\nn)$ is the number of edges of $G_{\hSS}(\nn)$ \emph{induced} by $U$ (that is, edges of $G_{\hSS}(\nn)$ with both endpoints in $U$). So the right-hand side of \eqref{eq:energy-count-multi} counts ordered partitions of $V(\nn)$ according to the number of edges it induces.

\begin{example}
Consider the case $N=2$, $S_{1,1}=S_{22}=\emptyset$ and $S_{1,2}=\{0\}$. In this case, $G=G_{\hSS}(n_1,n_2)$ is the \emph{complete bipartite graph} $K_{n_1,n_2}$. Hence, for a subset $U$ of $V(n_1,n_2)$ containing $k_1$ vertices of the form $(1,i)$ and $k_2$ vertices of the form $(2,i)$, the number of edges of $G$ induced by $U$ is $k_1k_2$. Thus \eqref{eq:energy-count-multi} gives
\begin{eqnarray*}
\sum_{(n_1,n_2)\in\NN^2} \!P_{K_{n_1,n_2}}(q,y)\frac{t_1^{n_1}t_2^{n_2}}{n_1!n_2!}&\!=\!&\left(\sum_{(n_1,n_2)\in\NN^2}\!\frac{t_1^{n_1}t_2^{n_2}}{n_1!n_2!}~\!\sum_{(U_1,\ldots,U_r)\atop\textrm{ordered partition of }V(n_1,n_2)}~\prod_{i=1}^r y^{\energy_\hSS(U_i)}\right)^{-q}\\
&\!=\!&\left(\frac{1}{1+\sum_{(k_1,k_2)\in\NN^2\setminus \{(0,0)\}} \!\!\frac{y^{k_1k_2}t_1^{k_1}t_2^{k_2}}{k_1!k_2!}}\right)^{\!-q}\\
&\!=\!&\left(\sum_{(k_1,k_2)\in\NN^2} \!\!\frac{y^{k_1k_2}t_1^{k_1}t_2^{k_2}}{k_1!k_2!}\right)^q,
\end{eqnarray*}
where the second equality directly follows basic generating function principles upon interpreting ordered set partitions as labeled sequences of sets (see \cite[Chapter 2]{Flajolet:analytic}).
\end{example}

Now, for an arbitrary graph $G=([N],E)$, we consider $\hSS=(S_{a,b})_{1\leq a\leq b\leq N}$ with $S_{a,a}=\{0\}$ for all $a\in[N]$, and for $a<b$, $S_{a,b}=\{0\}$ if $\{a,b\}\in E$ and $S_{a,b}=\emptyset$ otherwise. Then, \eqref{eq:q-is-exponent} gives 
\begin{equation}\label{eq:heaps}
\chi_G(q)=[t_1t_2\cdots t_N]R_\hSS(-\tt)^{-q}=(-1)^N[t_1t_2\cdots t_N]R_\hSS(\tt)^{-q}.
\end{equation}
This equation relates the proper colorings of $G$ (left-hand side) to the acyclic orientations of induced subgraphs of $G$ (right-hand side). For instance, it gives $(-1)^N\chi_G(-1)=[t_1t_2\cdots t_N]R_\hSS(\tt)$ which is the number of regions of $\mA_{\hSS(1,1,\ldots,1)}=\mA(G)$, or equivalently the number of acyclic orientations of $G$. This is precisely the interpretation of $\chi_G(-1)$ given by Stanley in~\cite{Stanley:acyclic-orientations}. More generally, for $U=\{u_1,\ldots,u_k\}\subseteq [N]$, the coefficient $[t_{u_1}t_{u_2}\cdots t_{u_k}]R_\hSS(\tt)$ is by definition the number of regions of the $G[U]$-graphical arrangement, where $G[U]$ is the subgraph of $G$ induced by $U$. Thus $[t_{u_1}t_{u_2}\cdots t_{u_k}]R_\hSS(\tt)$ is the number of acyclic orientations of $G[U]$. Using this interpretation, one can recover from \eqref{eq:heaps} the interpretation of $(-1)^N\chi(-q)$ given in~\cite{Stanley:acyclic-orientations} (it counts the total number of acyclic orientations induced on the blocks of an ordered set partition of length $q$ of the vertices of $G$).

For the readers with some familiarity with the theory of heaps (see for instance~\cite{Viennot:Heaps,Krattenthaler:Heaps}), we mention an interpretation of \eqref{eq:heaps} in this context. 
Consider the heap structure associated to the graph $G$: the pieces are the vertices of $G$, and two pieces overlap if they correspond to adjacent vertices.
In this context, $R_\hSS(\tt)$ can be interpreted as the generating function of the heaps of pieces associated to $G$ (where the variable $t_i$ counts the number of pieces associated to the vertex $i$ of $G$). Indeed, the regions of $\mA_{\hSS(\nn)}$ are in one-to-one correspondence with the acyclic orientations of $G_{\hSS(\nn)}$, which are in one-to-one correspondence with the heaps having $n_i$ pieces  associated to the vertex $i$ of $G$.
Hence, by~\cite[Proposition 5.3]{Viennot:Heaps}, $I(\tt):=R_\hSS(-\tt)^{-1}$ is the generating function of \emph{trivial heaps}, or equivalently, \emph{independent sets} of $G$ (that is, sets of non-adjacent vertices). Thus, through the theory of heaps \eqref{eq:heaps} becomes transparent: it simply expresses the fact that a proper $q$-coloring of $G$ is a $q$-tuple of independent sets partitioning the vertices.
More generally, we get a simple expression for the generating function of chromatic polynomials $\chi_\hSS(q,\tt)$:
$$\chi_\hSS(q,\tt)=I(\tt)^q=\left(\sum_{U\subseteq [N],~\textrm{independent set of } G}~\prod_{i\in U}t_i\right)^q.$$


\medskip

\section{Generating functions}\label{sec:GF}
In this section, we use Theorem \ref{thm:energy-count-multi} in order to give equations for the generating functions $P_\hSS(q,y,\tt)$, $\chi_\hSS(q,\tt)$, and $R_\hSS(\tt)$. These equations simply translate the decomposition of boxed trees obtained by deleting the box containing the root. 

\subsection{A universal generating function equation.}~\\
Let $\hSS=(S_{a,b})_{1\leq a\leq b\leq N}$ be an ${N+1 \choose 2}$-tuple of finite sets, and let $m=\max(|s|,~s\in \cup S_{a,b})$. In order to characterize the generating functions associated to the arrangements $\mA_\hSS(\nn)$, we need to define some combinatorial structures encoding admissible cadet sequences for trees in $\mT^{(m)}(\nn)$.

\begin{definition}\label{def:S-config-multi}
A \emph{$(m,N)$-configuration} of \emph{size} $\kk\in \NN^N$, is a pair $\ga=((d_1,\ldots,d_{|\kk|-1}),(u_1,\ldots,u_{|\kk|}))$ such that $\{u_1,\ldots,u_{|\kk|}\}=V(\kk)$, $d_1,\ldots,d_{|\kk|-1}\in [0.. m]$, and if $d_i=0$ then $u_i<u_{i+1}$. 

We denote by $|\ga|=\kk$ the size of $\ga$. The \emph{width} of $\ga$ is $\wid(\ga)=d_1+\cdots+d_{|\kk|-1}+m+1$, and the \emph{$\hSS$-energy} of $\ga$, denoted $\energy_\hSS(\ga)$, is the number of pairs $\{i,j\}$ with $1\leq i<j \leq |\kk|$ such that $ \sum_{p=i}^{j-1}d_p\in S_{u_i,u_j}^-$. 
\end{definition}

\begin{remark}\label{rk:cadet-seq=config}
To an admissible cadet sequence $(v_1,\ldots,v_k)$ of a tree $T\in\mT^{(m)}(\nn)$, we associate an $(m,N)$-configuration $\ga=((d_1,\ldots,d_{k-1}),(u_1,\ldots,u_{k}))$ defined as follows. Denoting $\kk=(k_1,\ldots,k_N)$ with $k_a=|\{i\in[n_a]~|~(a,i)\in \{v_1,\ldots,v_k\}\}|$, we set 
\begin{compactitem}
\item $d_i=\ls(v_{i+1})$ for all $i\in[k-1]$, 
\item $(u_1,\ldots,u_k)$ is the unique order-preserving relabeling of $(v_1,\ldots,v_k)$ in $V(\kk)$.
\end{compactitem}
It is clear that $\ga$ is an $(m,N)$-configuration of size $\kk$, and that $\energy_\hSS(\ga)$ is the $\hSS$-energy of the cadet sequence $(v_1,\ldots,v_k)$. Moreover, $\wid(\ga)$ is the number of children of $v_1,\ldots,v_k$ which are neither in $\{v_2,\ldots,v_k\}$ nor right-siblings of $v_2,\ldots,v_k$.
\end{remark}

For $\kk\in \NN^N$, we denote by $\mC^{(m)}(\kk)$ the set of $(m,N)$-configurations of size $\kk$, and we denote by $\mC_\hSS(\kk)$ the subset of configurations having $\hSS$-energy 0.  We also denote $\ds \mC^{(m,N)}=\bigcup_{\kk\neq (0,\ldots,0)}\mC^{(m)}(\kk)$ and $\ds \mC_\hSS=\bigcup_{\kk\neq (0,\ldots,0)}\mC_\hSS(\kk)$, where the unions are over non-zero tuples in $\NN^N$. Finally, we denote 
$$\Gamma_\hSS(x,y,\tt)=\sum_{\ga \in\mC^{(m,N)}}x^{\wid(\ga)}y^{\energy_S(\ga)}\frac{\tt^{|\ga|}}{|\ga|!},$$
and 
$$\Gamma_\hSS(x,\tt)=\Gamma_\hSS(x,0,\tt)=\sum_{\ga \in\mC_\hSS}x^{\wid(\ga)}\frac{\tt^{|\ga|}}{|\ga|!}.$$
We now state the general form of the generating function equation.

\begin{thm} \label{thm:GF-multi}
The generating function of coboundary polynomials $P_\hSS(q,y,\tt)$ (defined by \eqref{eq:def-gfP}) is equal to $\tP_\hSS(y,\tt)^{-q}$, where $\tP_\hSS(y,\tt)$ is the unique series in $\QQ[y][[t_1\ldots,t_N]]$ satisfying 
\begin{equation}\label{eq:GF-P-multi}
\tP_\hSS(y,\tt)=1-\Gamma_\hSS(\tP_\hSS(y,\tt),y,\tt).
\end{equation}
In particular, the generating function of regions $\ds R_\hSS(\tt)$ (defined by \eqref{eq:def-gfR}) is the unique series in $\QQ[[t_1,\ldots,t_N]]$ satisfying 
\begin{equation}\label{eq:GF-R-multi}
R_\hSS(\tt)=1-\Gamma_\hSS(R_\hSS(\tt),-\tt).
\end{equation}
\end{thm}

\begin{example}
Let $N=2$ and $S_{1,1}=[-2..2]$, $S_{1,2}=[-1..2]$, and $S_{2,2}=\{-2,0,1,2\}$. 
Then we have $m=2$ and $\mC_S=\{\ga_1,\ga_2,\ga_3,\ga_4,\ga_5\}$, where $\ga_1=((),(v_1))$, $\ga_2=((),(v_2))$ $\ga_3=((1),(v_2,v_3))$, $\ga_4=((2),(v_1,v_2))$, and $\ga_5=((2,1),(v_1,v_2,v_3))$ with $v_1=(1,1)$, $v_2=(2,1)$, $v_3=(2,2)$.
Thus $\Ga_\hSS(x,\tt)=(t_1+t_2)x^3+t_2^2x^4/2+t_1t_2x^5+t_1t_2^2x^6/2$ and 
$$R_\hSS(\tt)=1+(t_1+t_2)R_\hSS(\tt)^3-t_2^2R_\hSS(\tt)^4/2-t_1t_2R_\hSS(\tt)^5+t_1t_2^2R_\hSS(\tt)^6/2.$$
This gives 
$$R_\hSS(\tt)=1+t_{{1}}+t_{{2}}
+3\,{t_{{1}}}^{2}+5\,t_{{1}}t_{{2}}+5/2\,{t_{{2}}}^{2}
+12\,{t_{{1}}}^{3}+28\,{t_{{1}}}^{2}t_{{2}}+25\,t_{{1}}{t_{{2}}}^{2}+17/2\,{t_{{2}}}^{3}
+\ldots\,.
$$
\end{example}

\begin{proof}
By Theorem~\ref{thm:energy-count-multi}, $P_\hSS(q,y,\tt)=\tP_\hSS(y,\tt)^{-q}$ for 
$$\tP_\hSS(y,\tt):=\sum_{\nn\in\NN^N}\frac{\tt^\nn}{\nn!}\sum_{(T,B)\in \mU^{(m)}(\nn)}(-1)^{|B|}y^{\energy_\hSS(T,B)}.$$
We now consider the decomposition boxed trees $(T,B)\in\mU^{(m)}(\nn)$ for $|\nn|>0$. Consider the cadet sequence $\be=(v_1,\ldots,v_k)\in B$ containing the root $v_1$ of $T$. By Remark~\ref{rk:cadet-seq=config}, we can associate to $\be$ an $(m,N)$-configuration $\ga$. Deleting the vertices $v_1,\ldots,v_k$ of $T$ and the right-siblings of $v_2,\ldots,v_{k-1}$ (which, by definition, are leaves), gives a sequence of $\wid(\ga)$ subtrees. Hence, the class $\mU^{(m,N)}=\bigcup_{\nn\in\NN^N} \mU^{(m)}(\nn)$ admits the following recursive equation 
$$\mU^{(m,N)}=1+\sum_{\ga\in \mC^{(m,N)}}\{\ga\}\star \Seq_{\wid(\ga)}(\mU^{(m,N)}),$$
where $\star$ denotes the \emph{product}, and $\Seq_\ell$ denotes the \emph{$\ell$-sequences} construction for \emph{labeled combinatorial classes} (see e.g.~\cite[Chapter 2]{Flajolet:analytic}). This gives 
$$\tP_\hSS(y,\tt)=1+\sum_{\ga\in \mC^{(m,N)}}\left(-y^{\energy_\hSS(\ga)}\frac{\tt^{|\ga|}}{|\ga|!}\right)\times \left(\tP_\hSS(y,\tt)\right)^{\wid(\ga)},$$
which is \eqref{eq:GF-P-multi}. Moreover, \eqref{eq:GF-P-multi} implies \eqref{eq:GF-R-multi}, because \eqref{eq:face-characteristic} gives $R_\hSS(\tt)=\tP_\hSS(0,-\tt)$.
\end{proof}

\medskip
\subsection{Generating functions for $S$-braid arrangements (case N=1).}~\\
In this subsection, we explore in more details the case of $S$-braid arrangements (case $N=1$ of Theorem~\ref{thm:GF-multi}). When $S$ is transitive we obtain simpler equations for the generating function of regions. We then recover and extend several classical results.

For a set of integers $S$, we denote $\mC_S=\mC_{(S)}$. Hence, $\mC_S$ is the set of pairs $\ga=((d_1,\ldots,d_{k-1}),(v_1,\ldots,v_{k}))$ such that 
\begin{compactitem} 
\item $\{v_1,\ldots,v_{k}\}=[k]$, 
\item $d_1,\ldots,d_{k-1}\in [0..m]$, where $m=\max(|s|,~s\in S)$
\item for all $0\leq i<j\leq k$ either $v_i<v_j$ and $-\sum_{p=i}^{j-1}d_p\notin S$, or $v_i>v_j$ and $\sum_{p=i}^{j-1}d_p\notin S\cup \{0\}$.
\end{compactitem}
Equation \eqref{eq:GF-R-multi} then gives the following characterization of $\ds R_S(t)=\sum_{n\geq 0}r_{\mA_S(n)}\frac{t^n}{n!}$:
\begin{equation}\label{eq:GF-R}
R_S(t)=1-\Gamma_S(R_S(t),-t),
\end{equation}
where $\ds \Gamma_S(x,t)=\sum_{\ga \in\mC_S}x^{\wid(\ga)}\frac{t^{|\ga|}}{|\ga|!}$.


\begin{example}
For $S=[-3..3]\setminus \{-2,1\}$, we have $m=3$ and $\mC_S=\{\ga_k,~k\geq 1\}\cup\{\ga'\}$, where $\ga_k=((2,2,\ldots,2),(1,2,\ldots,k))$ and $\ga'=((1),(2,1))$. 
Thus $\Gamma_S(x,t)=\sum_{k\geq 1} t^k\frac{x^{2k+2}}{k!}+t^2x^5/2=x^2(e^{tx^2}-1)+t^2x^5/2$ and 
$$R_S(t)=1+R_S(t)^2(1-e^{-t R_S(t)^2})-t^2R_S(t)^5/2.$$
\end{example}

Next, we give an expression for $\Gamma_S(x,\tt)$ when $S$ is transitive.
For $k>1$ we denote by $\mathfrak{S}_k$ the set of permutations of $[k]$. For $\pi\in \mathfrak{S}_k$, we let $\asc(\pi)=\{i\in[k-1]~|~\pi(i)<\pi(i+1)\}$ and $\des(\pi)=\{i\in[k-1]~|~\pi(i)>\pi(i+1)\}$ be the number of ascents and descents of $\pi$ respectively. We denote 
$$\La(u,v,t)=\sum_{k=1}^\infty\sum_{\pi\in \mathfrak{S}_k}u^{\asc(\pi)}v^{\des(\pi)}\frac{t^k}{k!}.$$
This is the generating function of the \emph{homogeneous Eulerian polynomials}. Clearly,
$\La(u,u,t)=\frac{t}{1-tu}$ and $\La(u,0,t)=\frac{e^{tu}-1}{u}$. More generally, it is known~\cite{Carlitz:Eulerian-numbers} that 
\begin{equation}\label{eq:Eulerian-poly}
\La(u,v,t)=\frac{e^{tu}-e^{tv}}{u\,e^{tv}-v\,e^{tu}}.
\end{equation}

\begin{prop}
If $S\subseteq \ZZ$ is transitive, and $m=\max(|s|,~s\in S)$ then
\begin{equation}\label{eq:Gamma-transitive}
\Gamma_S(x,t)=x^{m+1}\La(\mu(x),\nu(x),t),
\end{equation}
where $\ds \mu(x)=\sum_{-d\in [-m..0]\setminus S}x^d$, and $\ds\nu(x)=\sum_{d\in [m]\setminus S}x^d$. 
Thus, $R_S(t)$ is the unique solution of 
\begin{equation}\label{eq:eqGF-transitive}
R_S(t)=1-R_S(t)^{m+1}\La(\mu(R_S(t)),\nu(R_S(t)),-t).
\end{equation}
\end{prop}

\begin{example} For $S=[-m..m]$, we have $\mu(x)=\nu(x)=0$. Hence $\Gamma_S(x,t)=tx^{m+1}$, and $R_{S}(t)=1+t\,R_{S}(t)^{m+1}$.
\end{example}

\begin{proof}
Lemma~\ref{lem:S-config-when-transitive} gives a simple characterization of $\mC_S$. Namely $\ga=((d_1,\ldots,d_{k-1}),(u_1,\ldots,u_{k}))$ is in $\mC_S$
if and only if $u_1,\ldots,u_{k}$ is a permutation of $[k]$, and for all $i\in[k-1]$, $d_i\in[0..m]$ and either ($v_i<v_{i+1}$ and $-d_i\notin S$) or ($v_i>v_{i+1}$ and $d_i\notin S\cup\{0\}$). Thus, for each ascent $i$ of the permutation $u_1,\ldots,u_{k}$, $-d_i$ is in $[-m..0]\setminus S$, and for each descent $i$, $d_i$ is in $[m]\setminus S$. This gives \eqref{eq:Gamma-transitive}.
\end{proof}

Let us first consider the special case $[m]\subseteq S$. Equation~\eqref{eq:Gamma-transitive-interval} below is~\cite[Theorem 9.1]{Postnikov:coxeter-hyperplanes} of Postnikov and Stanley.

\begin{cor}
If $[m]\subseteq S\subset [-m..m]$ and $\{s<0,s\notin S\}$ is closed under addition, then 
\begin{equation}\label{eq:Gamma-transitive-positive}
R_S(t)=1+R_S(t)^{m+1}\frac{1-e^{-t\,\mu(R_S(t))}}{\mu(R_S(t))},
\end{equation}
where $\ds\mu(x)=\sum_{-d\in [-m..0]\setminus S}x^d$. In particular, if $S=[-\ell..m]$ with $\ell\in[-1..m-1]$, then 
\begin{equation}\label{eq:Gamma-transitive-interval}
R_S(t)^{m-\ell}=\exp\left(t\frac{R_{S}(t)^{m+1}-R_{S}(t)^{\ell+1}}{R_{S}(t)-1}\right).
\end{equation}
\end{cor}

\begin{proof}
A set $S$ satisfying the assumptions is transitive. Moreover $\nu(x)=0$ so that $\La(\mu(x),\nu(x),t)=\frac{e^{t\mu(x)}-1}{\mu(x)}$. Thus \eqref{eq:eqGF-transitive} gives \eqref{eq:Gamma-transitive-positive}. 
In the particular case $S=[-\ell..m]$ we also have $\mu(x)=\frac{x^{m+1}-x^{\ell+1}}{x-1}$, and \eqref{eq:Gamma-transitive-positive} readily gives \eqref{eq:Gamma-transitive-interval}.
\end{proof}

In the special case $S=-S$, we recover~\cite[Theorem 2.3]{Stanley:hyperplane-interval-orders-overview} of Stanley and~\cite[Theorem 2.4]{Stanley:hyperplane-interval-orders-overview} (written in a slightly different form) which is credited to Athanasiadis.

\begin{cor}[\cite{Stanley:hyperplane-interval-orders-overview}]\label{cor:Ssym}
If $S\subseteq \ZZ$ satisfies $0\in S$, $\{-s,~s\in S\}=S$, and $\NN\setminus S$ is closed under addition, then for all $n>0$,
\begin{equation}\label{eq:nb-regions-Ssym}
r_{A_S(n)}=(n-1)![x^{n-1}]\left(1+x\sum_{d\in[0..m]\cap S}(x+1)^d\right)^n.
\end{equation}
where $m=\max(|s|,~s\in S)$. Moreover, 
\begin{equation}\label{eq:Ssym-0}
R_{S\setminus \{0\}}(t)=R_S(1-e^{-t}). 
\end{equation}
\end{cor}

\begin{proof}
A set $S$ satisfying the assumptions is transitive. Moreover, $\mu(x)=\nu(x)$, so that \eqref{eq:eqGF-transitive} becomes
\begin{equation}\label{eq:eqGF-Ssym}
\ds R_S(t)=1+\frac{t\,R_S(t)^{m+1}}{1+t\,\nu(R_S(t))},
\end{equation}
where $\nu(x)=\sum_{d\in [m]\setminus S}x^d$. 
This gives 
$\ds \tR(t)=t\Theta(\tR(t)),$
where $\tR(t)=R_S(t)-1$ and $\Theta(x)=(x+1)^{m+1}-x\,\nu(x+1)=1+x\sum_{d\in[0..m]\cap S}(x+1)^d$. Hence, Lagrange inversion formula gives \eqref{eq:nb-regions-Ssym}. 
Moreover, $S\setminus \{0\}$ is also transitive, and 
$$\Ga_{S\setminus \{0\}}(x,t)=x^{m+1}\La(\nu(x)+1,\nu(x),t)=x^{m+1}\frac{e^t-1}{1-(e^t-1)\nu(x)}=\Ga_{S}(x,e^t-1).$$
Thus \eqref{eq:eqGF-transitive} becomes 
$$R_{S\setminus \{0\}}(t)=1+\frac{(1-e^{-t})\,R_{S\setminus \{0\}}(t)^{m+1}}{1+(1-e^{-t})\,\nu(R_{S\setminus \{0\}}(t))}.$$
Comparing this equation with \eqref{eq:eqGF-Ssym} gives \eqref{eq:Ssym-0}.
\end{proof}

\medskip
\subsection{Generating functions for transitive arrangements in the case $N> 1$.}~\\
In this subsection, we return to the general case $N\geq 1$ and consider several transitive arrangements.

\begin{thm}\label{thm:GF-transitive}
Suppose $\hSS=(S_{a,b})_{1\leq a\leq b\leq N}$ is multi-transitive. 
Let $\Ga_1(x,\tt),\ldots,\Ga_N(x,\tt)$ be the series defined by the system of linear equations
\begin{equation}\label{eq:Gaa}
\Ga_a(x,\tt)=\La(\mu_{a,a}(x),\nu_{a,a}(x),t_a)\cdot\left(1+\sum_{b=1}^{a-1}\nu_{b,a}(x)\Ga_b(x,\tt)+\sum_{b=a+1}^{N}\mu_{a,b}(x)\Ga_b(x,\tt) \right),
\end{equation}
where $\La$ is defined by \eqref{eq:Eulerian-poly}, and for all $1\leq a\leq b\leq N$, $ \mu_{a,b}(x)=\sum_{-d\in [-m..0]\setminus S_{a,b}}x^d$, and $\nu_{a,b}(x)=\sum_{d\in [m]\setminus S_{a,b}}x^d$.
Then $\Ga_\hSS(x,\tt)=x^{m+1}\sum_{a=1}^N\Ga_a(x,\tt)$ so that 
$$R_\hSS(\tt)=1-R_\hSS(\tt)^{m+1}\sum_{a=1}^N\Ga_a(R_\hSS(\tt),-\tt).$$
\end{thm}

\begin{proof}
Let $\mC_{\hSS,a}$ be the set of configurations $\ga=((d_1,\ldots,d_{k-1}),(v_1,\ldots,v_k))\in\mC_\hSS$ such that $v_1$ has the form $(a,i)$ for some $i$.
We claim that $\Ga_a(x,\tt)$ is the generating function of configurations in $\mC_{\hSS,a}$. More precisely,
$$\Ga_{a}(x,\tt)=\sum_{\ga \in\mC_\hSS}x^{\wid(\ga)-m-1}\frac{\tt^{|\ga|}}{|\ga|!}.$$
Indeed, $\mC_{\hSS,a}$ has a simple description (see proof of Theorem~\ref{thm:unsigned-count-multi}), and \eqref{eq:Gaa} simply translates the decomposition of configurations in $\mC_{\hSS,a}$, at the first $p\in[k]$ such that $v_p$ has the form $(b,j)$ with $b\neq a$ (the term 1 in \eqref{eq:Gaa} corresponds to the case where there is no such $p$).
\end{proof}

As an illustration of Theorem \ref{thm:GF-transitive}, we treat two examples inspired by \cite{Gessel:Counting-binary-trees-descents}.

\begin{example} 
Suppose first that $S_{a,a}=\{-1,0,1\}$ for all $a\in [N]$ and $S_{a,b}=\{-1,0\}$ for all $1\leq a<b\leq N$. Then \eqref{eq:Gaa} reads 
$\ds \Ga_a(x,\tt)=t_a(1+\sum_{b=1}^{a-1}x\Ga_b(x,\tt))$. 
This gives $\ds \Ga_a(x,\tt)=\sum_{k>0}x^{k-1}\sum_{1\leq i_1<i_2<\cdots<i_k=a}\prod_{j=1}^k t_{i_j}$, so that 
$$\Ga_{\hSS}(x,\tt)=x\sum_{k>0}x^{k}\sum_{1 \leq i_1<i_2<\cdots<i_k\leq N}\prod_{j=1}^n t_{i_j}=x\left(\prod_{a=1}^N 1+t_a x\right)- x.$$
Thus, \eqref{eq:eqGF-transitive} gives
\begin{equation}\label{eq:Gessel1}
R_\hSS(\tt)=\prod_{a=1}^N \frac{1}{1-t_aR_{\hSS}(\tt)}.
\end{equation}
Now consider $\hSS'=(S_{a,b}')_{1\leq a\leq b\leq N}$ with $S_{a,a}'=\{0\}$ for all $a\in [N]$ and $S_{a,b}'=\{0,1\}$ for all $1\leq a<b\leq N$.
Equation \eqref{eq:Gaa} reads 
$\ds \Ga_a(x,\tt)=\frac{t_a}{1-t_ax}(1+\sum_{b=a+1}^Nx\Ga_b(x,\tt))$.
hence $\ds \Ga_a(x,\tt)=\sum_{k>0}x^{k-1}\sum_{a=i_1\leq i_2 \leq \cdots \leq i_k\leq N}\prod_{j=1}^k t_{i_j}$, and
$$\Ga_{\hSS'}(x,\tt)=x\sum_{k>0}x^{k}\sum_{1 \leq i_1\leq i_2\leq\cdots \leq i_k\leq N}\prod_{j=1}^n t_{i_j}=x\left(\prod_{a=1}^N\frac{1}{1-t_a x}\right)- x.$$
Thus, \eqref{eq:eqGF-transitive} gives
\begin{equation}\label{eq:Gessel2}
R_{\hSS'}(\tt)=\prod_{a=1}^N 1+t_aR_{\hSS'}(\tt).
\end{equation}
Note that \eqref{eq:Gessel1} and \eqref{eq:Gessel2} imply that $R_{\hSS}(\tt)$ and $R_{\hSS'}(\tt)$ are symmetric functions in $(t_1,\ldots,t_N)$, which is not obvious from the definition. This is a special case of a result proved in \cite{Gessel:Counting-binary-trees-descents}. In fact, it follows from  \eqref{eq:GF-regions-transitive-multi} that $R_{\hSS}(\tt)=1+B(1,1,0,1,\tt)$ and $R_{\hSS'}(\tt)=1+B(1,1,1,0,\tt)$ for the series $B(u_1,u_2,v_1,v_2,\tt)$ considered in \cite{Gessel:Counting-binary-trees-descents} which counts trees in $\bigcup_{\nn\in\NN^N}\mT^{(1)}(\nn)$ according to certain ascent and descent statistics. Accordingly, \eqref{eq:Gessel1} and \eqref{eq:Gessel2} are special cases of the equation given for $B(u_1,u_2,v_1,v_2,\tt)$ in \cite{Gessel:Counting-binary-trees-descents}.
\end{example}

We now state the extensions of Corollary~\ref{cor:Ssym} to $N>1$.

\begin{cor}\label{cor:GF-Ssym-multi}
Suppose that $\hSS=(S_{a,b})_{1\leq a\leq b\leq N}$ is multi-transitive, and that for all $1\leq a\leq b\leq N$, the set $S_{a,b}$ contains 0 and satisfies $\{-s,~s\in S_{a,b}\}=S_{a,b}$. Then for all $\nn\neq (0,\ldots,0)$,
$$r_{\mA_\hSS(\nn)}=[\tt^\nn]\left(\left(\sum_{a=1}^N t_a\right)\left(\prod_{b=1}^N g_b(\tt)^{n_b-1}\right)\,\det\!\left(\delta_{i,j}\,g_i(\tt)-t_j\frac{\partial g_i(\tt)}{\partial t_j}\right)_{i,j\in[N]}\right),$$
where $\ds g_a(\tt)=1+\sum_{b=1}^N t_b \sum_{d\in [0..m]\cap S_{a,b}}\bigg(1+\sum_{c=1}^N t_c\bigg)^d$, and $\de_{i,j}$ is the Kronecker delta.
\end{cor}

\begin{example} Let $N=2$ and let $\hSS=(S_{a,b})_{1\leq a\leq b\leq 2}$ with $S_{1,1}=S_{2,2}=\{-1,0,1\}$ and $S_{1,2}=\{0\}$. Then 
$g_1(t_1,t_2):=1+t_1(2+t_1+t_2)+t_2$, $g_2(t_1,t_2):=1+t_2(2+t_1+t_2)+t_1$, the determinant is $(1+t_1+t_2)(1-t_1^2-t_2^2)$ and Corollary~\ref{cor:GF-Ssym-multi} gives
$$r_{\mA_\hSS(n_1,n_2)}=[t_1^{n_1}t_2^{n_2}]\left((t_1+t_2)(1+t_1+t_2)(1-t_1^2-t_2^2)\,g_1(t_1,t_2)^{n_1-1}\,g_2(t_1,t_2)^{n_2-1}\right).$$
\end{example}

Corollary \ref{cor:GF-Ssym-multi} is an application of the multivariate Lagrange inversion formula~\cite{Gessel:multivariate-Lagrange,Good:multivariate-Lagrange} that we now recall for the readers' convenience.

\begin{lemma}[Lagrange inversion formula]\label{lem:Lagrange-inversion} Let $\tt=(t_1,\ldots,t_N)$ be indeterminates. Let $g_1(\tt),\ldots,g_N(\tt)$ be series in $\CC[[\tt]]$ with non-zero constant terms. Let $f_1(\tt),\ldots,f_N(\tt)$ be the unique series in $\CC[[\tt]]$ such that for all $i\in[N]$ $f_i(\tt)=t_ig_i(f_1(\tt),\ldots,f_N(\tt))$. Then for all $a\in [N]$ and for all tuples $\nn=(n_1,\ldots,n_N)\in \NN^N$, 
$$[\tt^\nn]f_a(\tt)=[\tt^\nn]\left(t_a\cdot \left(\prod_{b=1}^N g_b(\tt)^{n_b-1}\right)\cdot \det\!\left(\delta_{i,j}\,g_i(\tt)-t_j\frac{\partial g_i(\tt)}{\partial t_j}\right)_{i,j\in[N]}\right).$$
\end{lemma}

\begin{proof}[Proof of Corollary \ref{cor:GF-Ssym-multi}] 
Since for all $a\in[N]$, $\La(\mu_{a,a}(x),\nu_{a,a}(x),t_a)=\frac{t_a}{1-t_a\nu_{a,a}(x)}$, Equation \eqref{eq:Gaa} can be rewritten as
$\Ga_a(x,\tt)=t_a(1+\sum_{b=1}^N\nu_{a,b}(x)\Ga_b(x,\tt))$.
Hence, denoting $R(\tt)=R_{\hSS}(\tt)-1$ and $ R_a(\tt)=-R_{\hSS}(\tt)^{m+1}\Ga_a(R_{\hSS}(\tt),-\tt)$, we get 
$$R(\tt)=\sum_{a=1}^N R_a(\tt), \textrm{ and for all $a\in[N]$, } \ds R_a(\tt)=t_a\, g_a(R_1(\tt),\ldots,R_N(\tt)),$$ 
where 
$$g_a(r_1,\ldots,r_N)=\bigg(1+\sum_{c=1}^N r_c\bigg)^{m+1}-\sum_{b=1}^N r_b\,\nu_{a,b}\bigg(1+\sum_{c=1}^N r_c\bigg)=1+\sum_{b=1}^N r_b \sum_{d\in [0..m]\cap S_{a,b}}\!\!\bigg(1+\sum_{c=1}^Nr_c\bigg)^d.$$
Applying Lemma \ref{lem:Lagrange-inversion} gives the result.
\end{proof}

\begin{cor}
Suppose that $\hSS=(S_{a,b})_{1\leq a\leq b\leq N}$ is multi-transitive, and that for some $a\in[N]$, $S_{a,a}$ contains 0 and satisfies $\{-s,~s\in S_{a,a}\}=S_{a,a}$. Let $\hSS'$ be the same tuple as $\hSS$ except $S_{a,a}$ is replaced by $S_{a,a}\setminus \{0\}$. Then,
$$R_{\hSS'}(\tt)=R_{\hSS}(t_1,\ldots,t_{a-1},1-e^{-t_a},t_{a+1},\ldots,t_N).$$
\end{cor}

\begin{proof} Let $\Ga_1(x,\tt),\ldots,\Ga_N(x,\tt)$ be the series satisfying \eqref{eq:Gaa} for the tuple $\hSS$ and $\Ga_1'(x,\tt),\ldots,\Ga_N'(x,\tt)$ their analogues for $\hSS'$. 
As in the proof of Corollary~\ref{cor:Ssym}, we get for all $i\in[N]$ $\Ga_i'(x,\tt)=\Ga_i'(x,\tt')$, where $\tt'=(t_1,\ldots,t_{a-1},e^{t_a}-1,t_{a+1},\ldots,t_N)$. Thus $\Ga_{\hSS'}(x,\tt)=\Ga_\hSS(x,\tt')$, and $\Ga_{\hSS'}(x,-\tt)=\Ga_{\hSS}(x,-\tt'')$, for $\tt''=(t_1,\ldots,t_{a-1},1-e^{-t_a},t_{a+1},\ldots,t_N)$. Hence \eqref{eq:GF-R-multi} implies $R_{\hSS'}(\tt)=R_{\hSS}(\tt'')$.
\end{proof}

\begin{cor}\label{cor:Sab=SSym}
Suppose that $\hSS=(S_{a,b})_{1\leq a\leq b\leq N}$ is multi-transitive, and that for all $a<b$, $S_{a,b}=S$ for some set $S$ containing 0 and such that $\{-s,~s\in S\}=S$. Then
$$\Ga_{\hSS}(x,\tt)=x^{m+1}\frac{\De(x,\tt)}{1-\nu(x)\De(x,\tt)},$$
where $\ds \De(x,\tt)=\sum_{a=1}^N\frac{\La(\mu_{a}(x),\nu_{a}(x),t_a)}{1+\nu(x)\La(\mu_{a}(x),\nu_{a}(x),t_a)}$, $\ds \mu_a(x)=\sum_{-d\in [-m..0]\setminus S_{a,a}}x^d$, $\ds \nu_a(x)=\sum_{d\in [m]\setminus S_{a,a}}x^d$, and $\ds \nu(x)=\sum_{d\in [m]\setminus S}x^d$. 
\end{cor}

\begin{example}
Let $N=2$ and $\hSS=(S_{a,b})_{1\leq a\leq b\leq 2}$ with $S_{1,1}=\{-1,0,1\}$, $S_{2,2}=\{0,1\}$ and $S_{1,2}=\{0\}$. Then with the notation of Corollary~\ref{cor:Sab=SSym}, we have $\nu(x)=x$, $\ds \De(x,\tt)=\frac{t_1}{1+t_1x}+\frac{e^{t_2x}-1}{x+(e^{t_2x}-1)x}$ and 
$$\Ga_\hSS(x,\tt)=-{\frac {x \left( 2\,t_{{1}}x{{\rm e}^{t_{{2}}x}}-t_{{1}}x+{{\rm e}^{t_{{2}}x}}-1 \right) }{t_{{1}}x{{\rm e}^{t_{{2}}x}}-t_{{1}}x-1}}.$$
\end{example}


\begin{proof}
Equation \eqref{eq:Gaa} can be rewritten as 
$$\Ga_a(x,\tt)=\La(\mu_{a}(x),\nu_{a}(x),t_a)(1+\nu(x)(\Ga_\hSS(x,\tt)/x^{m+1}-\Ga_a(x,\tt)).$$
Thus $\ds \Ga_a(x,\tt)=\frac{\La(\mu_{a}(x),\nu_{a}(x),t_a)}{1+\nu(x)\La(\mu_{a}(x),\nu_{a}(x),t_a)}(1+\nu(x)\Ga_\hSS(x,\tt)/x^{m+1})$, and 
$$\Ga_\hSS(x,\tt)=x^{m+1}\sum_{a=1}^N\Ga_a(x,\tt)=\De(x,\tt)\,(x^{m+1}+\nu(x)\Ga_\hSS(x,\tt)).$$
\end{proof}

\begin{remark} Our results for the number of regions of a multi-transitive arrangement $\mA_\hSS(\nn)$ have been derived from \eqref{eq:GF-regions-multi} using the decomposition of boxed trees in $\mU_\hSS(\nn)$.
They could alternately be obtained from \eqref{eq:GF-regions-transitive-multi} using the decomposition of trees in $\mT_\hSS(\nn)$. When $S_{a,b}=-S_{a,b}$ for all $a,b\in[N]$ one can simply use the decomposition obtained by deleting the root. In the general case, the decomposition would correspond to deleting all the vertices in the longest cadet sequence starting at the root.
\end{remark}


\medskip

\section{Proofs}\label{sec:proof}
As explained above, Theorem~\ref{thm:energy-count-multi} implies Theorems~\ref{thm:signed-count} and Theorems~\ref{thm:signed-count-multi}. It remains to prove Theorem~\ref{thm:energy-count-multi}. The proof breaks into three steps corresponding to Lemmas~\ref{lem:1-multi},~\ref{lem:2-multi}, and~\ref{lem:3-multi} below. 

We first express the coboundary polynomial of any deformation of the braid arrangement as a weighted count of graphs. We denote by $\mG_n$ the set of graphs (without loops nor multiple edges) with vertex set $[n]$.

\begin{lemma}\label{lem:1-multi}
Let $n\in \NN$, and let $\SS=(S_{u,v})_{1\leq u<v\leq n}$ be an ${n\choose 2}$-tuple of finite sets of integers, and let $m=\max(|s|,~s\in \cup S_{u,v})$.
The coboundary polynomial of $\mA_{\SS}$ is 
\begin{equation*}
P_{\mA_\SS}(q,y)=\sum_{G\in\mG_n}(y-1)^{\ee(G)}q^{\comp(G)}|\mW_\SS(G)|,
\end{equation*}
where $\ee(G)$ and $\comp(G)$ are the number of edges and connected components of $G$ respectively, and 
 $\mW_\SS(G)$ is the set of tuples $(x_1,\ldots,x_n)\in\ZZ^n$ such that 
\begin{compactitem}
\item for all edge $\{u,v\}$ of $G$ with $u<v$, $x_u-x_v$ is in $S_{u,v}$,
\item for all vertex $v$ of $G$ such that $v$ is smallest in its connected component, $x_v=0$.
\end{compactitem}
\end{lemma}

\begin{example} Let $n=3$ and $S_{1,2}=S_{1,3}=\{-2,1\}$ and $S_{2,3}=\{-2,-1,1\}$. Let $G$ be the graph with vertex set $[3]$ and edges $\{1,2\}$ and $\{2,3\}$. Then\\ $\ds \mW_\SS(G)=\{(0,-1,-2), (0,-1,0), (0,-1,1), (0,2,1), (0,2,3), (0,2,4)\}$.
\end{example}

\begin{proof}
To a subarrangement $\mB\subseteq \mA_\SS$, we associate the graph $G_\mB\in \mG_n$ with arcs $\{u,v\}$ for all $u<v$ such that there exists $s\in S_{u,v}$ such that $H_{a,b,s}\in\mB$. We say that $\mB$ is \emph{central} if $\cap_{H\in\mB}H\neq \emptyset$. 
If $\mB$ is central, for each edge $\{u,v\}$ of $G_\mB$, with $u<v$, there is a unique value $s\in S_{u,v}$ such that $H_{u,v,s}\in\mB$, and we denote this value $\mB(u,v)$. We also denote $\mB(v,u)=-\mB(u,v)$. Clearly, a point $(x_1,\ldots,x_n)$ is in $\bigcap_{H\in \mB}H$ if and only if for any path $v_1,v_2,\ldots,v_k$ in $G_\mB$,
$$\ds x_{v_1}-x_{v_k}=\sum_{i=1}^{k-1}\mB(v_i,v_{i+1}).$$ 
Hence, there is a unique point $\xx_\mB=(x_1,\ldots,x_n)$ in $\bigcap_{H\in \mB}H$ such that $x_v=0$
for all $v\in[n]$ such that $v$ is the smallest vertex in its connected component of $G_\mB$. Moreover, $\dim(\bigcap_{H\in \mB}H)=\comp(G_{\mB})$.
Note also that $\xx_\mB$ is in $\mW_\SS(G_{\mB})$, and that $\mB$ is uniquely determined by the pair $(G_\mB,\xx_\mB)$. Lastly, any pair $(G,\xx)$ where $G\in \mG_n$ and $\xx\in \mW_\SS(G)$ comes from a central subarrangement $\mB$. Thus,
\begin{eqnarray*}
P_{\mA_\SS}(q,y)&=&\sum_{\mB\subseteq \mA_\SS\textrm{ central}}(y-1)^{|\mB|}q^{\dim\left(\bigcap_{H\in \mB}H\right)}\\
&=&\sum_{(G,\xx),~ G\in \mG_n,~\xx\in \mW_\SS(G)}(y-1)^{\ee(G)}q^{\comp(G)}\\
&=&\sum_{G\in\mG_n}(y-1)^{\ee(G)}q^{\comp(G)}|\mW_\SS(G)|.
\end{eqnarray*}
\end{proof}

Our second step relates the generating function $P_\hSS(q,y,\tt)$ of coboundary polynomials, to a generating function $Z_{\hSS}(\de,y,\tt)$ of tuples of integers. We fix $N>0$ and an ${N+1\choose 2}$-tuple $\hSS=(S_{a,b})_{1\leq a<b\leq N}$ of finite sets of integers. As before, $m=\max(|s|,~s\in \cup S_{a,b})$ and for $\nn\in \NN^N$ and $u,v\in V(\nn)$ with $u=(a,i)$, $v=(b,j)$ and $u<v$, we denote
$S_{u,v}=S_{a,b}$, $S_{v,u}=S_{a,b}$, $S_{u,v}^-=\{s\geq 0~|~-s\in S_{a,b}\}$, and $S_{v,u}^-=\{s> 0~|~s\in S_{a,b}\}\cup\{0\}$.

For a positive integer $\de$, and $\nn\in \NN^N$, we denote
$$Z_{\hSS,\nn}(\de,y)=\sum_{\xx=(x_{v})_{v\in V(\nn)}\in [\de]^{|\nn|}}y^{\energy_\hSS(\xx)},$$
where $\energy_\hSS(\xx)$ is number of pairs $(u,v)\in V(\nn)^2$, with $u<v$, such that $x_{u}-x_{v}\in S_{u,v}$.
For instance, the $\hSS$-energy of the tuple $\xx$ in Figure~\ref{fig:particle-config-multi} is 1.
By convention, we set $Z_{\hSS,(0,\ldots,0)}(\de,y)=1$.

\fig{width=.7\linewidth}{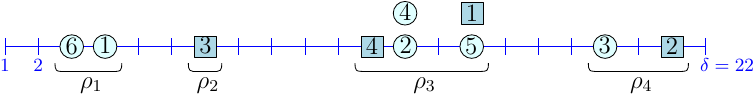}{Let $\de=22$, $N=2$, and $\nn=(6,4)$. In the figure, the tuple $\xx=(x_{1,1},\ldots,x_{1,6},x_{2,1},\ldots,x_{2,4})=(4,13,19,13,15,3,15,21,7,12)$ is represented by indicating a ``round-particle'' labeled $i$ in position $x_{1,i}$ for all $i\in[6]$, and a ``square-particle'' labeled $i$ in position $x_{2,i}$ for all $i\in[4]$. For $\hSS=(S_{a,b})_{1\leq a\leq b\leq N}$ with $S_{1,1}=S_{2,2}=\{-1,2\}$ and $S_{1,2}=\{-1,0,2\}$, the $\hSS$-energy is $\energy_\hSS(\xx)=1$, given by the pair $\{(1,5),(2,1)\}$ (because $x_{1,5}-x_{2,1}=0\in S_{(1,5),(2,1)}=\{-1,0,2\}$). The tuple $\xx$ has four runs $\rho_1,\ldots,\rho_4$.}

\begin{lemma}\label{lem:2-multi}
The generating functions $P_\hSS(q,y,\tt)$ 
and 
$$\ds Z_{\hSS}(\de,y,\tt)=\sum_{\nn\in \NN^N} Z_{\hSS,\nn}(\de,y) \frac{\tt^{\nn}}{\nn!},$$
are related by
\begin{equation}\label{eq:R-Vs-Z-multi}
\frac{1}{q}\log(P_\hSS(q,y,\tt))=\lim_{\de\to \infty} \frac{1}{\de}\log(Z_{\hSS}(\de,y,\tt)).
\end{equation}
Equation~\eqref{eq:R-Vs-Z-multi} is to be understood as an identity for formal power series in $t_1,\ldots,t_N$: 
\begin{itemize}
\item the limit is taken coefficient by coefficient in $t_1,\ldots,t_N$, 
\item for a series in formal power series $A(t_1,\ldots,t_N)$ such that $A(0,\ldots,0)=1$, we denote by $\log(A(\tt))$ the formal power series $\sum_{n=1}^\infty\frac{(-1)^{n-1}(A(\tt)-1)^n}{n}$.
\end{itemize}
\end{lemma}

\begin{proof}
Let $\mG_\nn$ be the set of graphs with vertex set $V(\nn)$, and let $\mG^{(N)}=\bigcup_{\nn\in\NN^N}\mG_\nn$. 
For $G\in \mG_\nn$, we denote by $\mW_\hSS(G)$ the set of tuples $(x_{v})_{v\in V(\nn)}\in\ZZ^{|\nn|}$ such that 
\begin{compactitem}
\item for all edge $\{u,v\}$ of $G$ with $u<v$, $x_{u}-x_{v}\in S_{u,v}$,
\item for all vertex $v$ of $G$ which is smallest in its connected component, $x_{v}=0$.
\end{compactitem}
Recall that the arrangement $\mA_\hSS(\nn)$ identifies with the arrangement $\mA_{\hSS(\nn)}$, where the tuple $\hSS(\nn)$ is given by~\eqref{eq:SSnn}. Up to this identification, Lemma~\ref{lem:1-multi} gives
$$P_{\mA_\hSS(\nn)}(q,y)=\sum_{G\in\mG_\nn}(y-1)^{\ee(G)}q^{\comp(G)}|\mW_\hSS(G)|,$$
hence
\begin{equation}\label{eq:PSS}
P_\hSS(q,y,\tt)=\sum_{\nn\in\NN^N}\frac{\tt^{\nn}}{\nn!}\sum_{G\in\mG_\nn}(y-1)^{\ee(G)}q^{\comp(G)}|\mW_\hSS(G)|.
\end{equation}
We now apply the multivariate exponential formula. We think of $\mG^{(N)}$ as the combinatorial class of graphs with $N$ types of vertices, with the vertices of each type being \emph{well-labeled} (that is, the $n_a$ vertices of type $a$ have distinct labels in $[n_a]$). The \emph{size} of $G\in \mG_\nn$ is $\nn=(n_1,\ldots,n_N)$ and the \emph{weight} of $G$ is $(y-1)^{\ee(G)}q^{\comp(G)}|\mW_\hSS(G)|$. The weight is multiplicative over connected components (and unchanged by order preserving relabeling of the vertices of each type). Hence the multivariate exponential formula applies (see e.g.~\cite{Stanley:volume2}), and taking the logarithm of $P_\hSS(q,y,\tt)$ amounts to selecting the connected graphs in $\mG^{(N)}$. This gives,
\begin{equation}\label{eq:logPSS}
\frac{1}{q}\log(P_\hSS(q,y,\tt))=\sum_{\nn\in\NN^N}\frac{\tt^{\nn}}{\nn!}\sum_{G\in\mG_\nn\textrm{ connected}}(y-1)^{\ee(G)}|\mW_\hSS(G)|.
\end{equation}

Next, observe that
\begin{eqnarray*}
|\mZ_{\hSS,\nn}(\de,y)|&=&
\sum_{(x_{v})_{v\in V(\nn)}\in [\de]^{|\nn|}}~\prod_{u,v\in V(\nn),~u<v}\left(1+(y-1)\cdot \one_{x_{u}-x_{v}\in S_{u,v}}\right),\\
&=&\sum_{(x_{v})_{v\in V(\nn)}\in [\de]^{|\nn|}}~\sum_{G\in \mG_\nn}(y-1)^{\ee(G)}\left(\prod_{\{u,v\} \textrm{ edge of }G,~u<v}\one_{x_{u}-x_{v}\in S_{u,v}}\right),\\
&=& \sum_{G\in \mG_\nn}(y-1)^{\ee(G)}|\mW_{\hSS,\de}(G)|,
\end{eqnarray*}
where $\one$ is the indicator function, and $\mW_{\hSS,\de}(G)$ is the set of tuples $(x_{v})_{v\in V(\nn)}\in [\de]^{|\nn|}$ such that for all edge $\{u,v\}$ of $G$ with $u<v$, $x_{u}-x_{v}\in S_{u,v}$.
The graph weight $(y-1)^{\ee(G)}|\mW_{\hSS,\de}(G)|$ is multiplicative over connected components, hence by the multivariate exponential formula, 
\begin{eqnarray}
\log(Z_{\hSS,\de}(\tt))&=& \log\left(\sum_{\nn\in \NN^N} \frac{\tt^{\nn}}{\nn!}\sum_{G\in \mG_\nn}(y-1)^{\ee(G)}|\mW_{\hSS,\de}(G)|\right),\nonumber\\
&=&\sum_{\nn\in \NN^N}\frac{\tt^{\nn}}{\nn!}\sum_{G\in\mG_\nn, \textrm{ connected}}(y-1)^{\ee(G)}|\mW_{\hSS,\de}(G)|.\label{eq:ZSS}
\end{eqnarray}
It only remains to prove that for any connected graph $G\in \mG_\nn$,
\begin{equation}\label{eq:lim-WSd}
\lim_{\de\to\infty}\frac{1}{\de}|\mW_{\hSS,\de}(G)|=|\mW_{\hSS}(G)|.
\end{equation}
It is easy to see that, the tuples in $\mW_{\hSS,\de}(G)$ are translations of tuples in $\mW_\hSS(G)$, and the number of translations is of order $\de$. More precisely, 
$$\mW_{\hSS,\de}(G)=\{(x_v+\theta)_{v\in V(\nn)}~|~(x_v)_{v\in V(\nn)}\in W_\hSS(G),\textrm{ and }1-\min(x_v)_{v\in V(\nn)}\leq\theta\leq \de-\max(x_v)_{v\in V(\nn)}\}.$$
The tuples above are all distinct, and for any $(x_v)_{v\in V(\nn)}\in W_\hSS(G)$, $\max(x_v)_{v\in V(\nn)}-\min(x_v)_{v\in V(\nn)}\leq m \cdot |\nn|$. Thus 
 $$\left(\de-m\cdot|\nn|\right)|\mW_{\hSS}(G)|\leq |\mW_{\hSS,\de}(G)|\leq \de\,|\mW_{\hSS}(G)|.$$ 
This shows~\eqref{eq:lim-WSd}, and completes the proof.
\end{proof}

\begin{remark} The proof of Lemma~\ref{lem:2-multi} is reminiscent of Mayers' theory of cluster integrals (see e.g.~\cite{OB:Mayers-theory,Leroux:mayers-theory1,Mayer:theory}). In this perspective, the right-hand side of~\eqref{eq:R-Vs-Z-multi} corresponds to the pressure of the infinite volume limit of a discrete gas model (where particles of type $a$ and $b$ interacts according to a soft-core potential of shape $S_{a,b}$, and energy of interaction $y$). Alexander Postnikov also pointed out to the author that Lemma~\ref{lem:2-multi} could alternatively be obtained by using the finite field method pioneered by Athanasiadis~\cite{Athanasiadis:finite-field-method} and adapted to the calculation of coboundary polynomials in~\cite{Ardila:Tutte-hyperplanes}. However, the situation of deformed braid arrangement is distinguished by the fact that the parameter $q$ appears merely as an exponent of the generating function $P_\hSS(q,y,\tt)$; see \eqref{eq:q-is-exponent}. This fact, which is a direct consequence of Lemma~\ref{lem:1-multi} and already appears in~\cite[Theorem 1.2]{Stanley:hyperplane-interval-orders-overview}, allows one to focus the remaining analysis on a single value of $q$, namely $+\infty$.
\end{remark}

Our last step, relates the generating function $Z_{\hSS}(\de,y,\tt)$ of point configurations to the generating function of boxed trees.

\begin{lemma}\label{lem:3-multi} 
Let
$$U_\hSS(y,\tt)=\sum_{\nn\in\NN^N}\frac{\tt^\nn}{\nn!}\sum_{(T,B)\in \mU^{(m)}(\nn)}(-1)^{|B|}y^{\energy_\hSS(T,B)},$$ 
and 
$$U^\bu_\hSS(y,\tt)=\sum_{\nn\in\NN^N}\frac{\tt^\nn}{\nn!}\sum_{(T,B)\in \mU_\SS^{(m)}(\nn)}(|B|+\leaf(T))(-1)^{|B|}y^{\energy_\hSS(T,B)},$$ 
where $\leaf(T)$ is the number leaves of $T$.
For all $\nn\in\NN^N$, and for all $\de>m\cdot |\nn|$, 
\begin{equation} \label{eq:Zcoeff-as-boxed-multi}
Z_{\hSS,\nn}(\de,y)=[\tt^{\nn}]U_\hSS(y,\tt)^{-\de-m-2}U_\hSS^\bu(y,\tt).
\end{equation}
\end{lemma}

We will use a counting result about tuples of plane trees. We recall that the \emph{prefix order} of the vertices of a rooted plane tree is the total order for which any vertex $v$ is less than its children and all the descendants of $v$ are less than the right siblings of $v$.

\begin{claim}\label{claim:tau} 
Let $\al,w_1\ldots,w_r$ be positive integers. Let $\tau(\al;w_1,\ldots,w_r)$ be the set of tuples $(T_1,\ldots,T_\al)$, where $T_1,\ldots,T_\al$ are rooted plane trees, and $T_\al$ has a marked vertex, such that, denoting $c_{i,1},\ldots, c_{i,k_i}$ the number of children of the nodes of $T_i$ in prefix order, one has 
$$\ds (c_{1,1},\ldots, c_{1,k_1},c_{2,1},\ldots, c_{2,k_2},\ldots, c_{\al,1},\ldots, c_{\al,k_\al})=(w_1,\ldots,w_r).$$ 
Then,
$$|\tau(\al;w_1,\ldots,w_r)|={\al+w_1+\cdots+w_r\choose r}.$$
\end{claim}
\begin{proof} The proof is represented in Figure~\ref{fig:decompose-path}. Let $\mP$ be the set of lattice paths on $\ZZ$ starting at 0, and having every step greater or equal to $-1$. Let $\mP_{-1}$ be the set of paths in $\mP$ ending at $-1$, and let $\mP_{-1}^{+}\subset \mP_{-1}$ be the subset of paths remaining non-negative until the last step. Recall that the map $\phi$ which associates to a rooted plane tree $T$ the path $P$ with steps $c_1-1,c_2-1,\ldots,c_n-1$ where $c_1,\ldots,c_n$ are the number of children of the vertices of $T$ taken in prefix order is a bijection between rooted plane trees and $\mP_{-1}^+$ (see e.g.~\cite[Chapter 5.3]{Stanley:volume2}). Moreover by the cycle lemma, there is a $n$-to-1 correspondence between the paths with $n$ steps in $\mP_{-1}$ and the paths with $n$ steps in $\mP_{-1}^+$. Thus the map $\phi$ induces a bijection between $\mP_{-1}$ and rooted plane trees with a marked vertex. 

Now let $\mP(\al;w_1,\ldots,w_r)$ be the set of path in $\mP$ having $\al+w_1+\ldots+w_r-r$ steps $-1$, and $r$ non-negative steps $w_1-1,\ldots,w_r-1$ in this order. Clearly, 
$|\mP(\al;w_1,\ldots,w_r)|={\al+w_1+\cdots+w_r\choose r},$
and paths in $\mP(\al;w_1,\ldots,w_r)$ ends at $-\al$. We consider the decomposition of paths in $\mP(\al;w_1,\ldots,w_r)$ at the first time they reach $-1,-2,\ldots,-\al+1$, as represented in Figure~\ref{fig:decompose-path}. This gives a bijection between $\mP(\al;w_1,\ldots,w_r)$ and the set of tuples $(P_1,\ldots,P_{\al})$ such that $P_1,\ldots,P_{\al-1}\in\mP_{-1}^+$, $~P_{\al}\in\mP_{-1}$ and there is a total of $\al+w_1+\ldots+w_r-r$ steps $-1$, and $r$ non-negative steps $w_1,\ldots,w_r$ in this order. Combining this decomposition with $\phi$ gives a bijection between $\mP(\al;w_1,\ldots,w_r)$ and $\tau(\al;w_1,\ldots,w_r)$, thereby proving the claim.
\end{proof}
\fig{width=\linewidth}{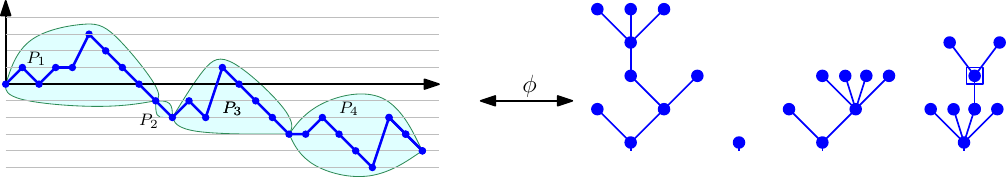}{A path $P$ in $\mP(4;2,2,1,3,2,4,1,2,4)$, and its decomposition into four paths $P_1,\ldots,P_4$ with $P_1,P_2,P_3\in\mP_{-1}^+$, and $P_4$ in $\mP_{-1}$. Here the vertical direction corresponds to the value of the path, while the horizontal direction represents time.}

\begin{proof}[Proof of Lemma~\ref{lem:3-multi}]
We will relate both sides of~\eqref{eq:Zcoeff-as-boxed-multi} to the generating function of $(m,N)$-configurations (see Definition~\ref{def:S-config-multi}). Let $(T,B)\in \mU_\SS^{(m)}(\nn)$. We associate to $(T,B)$ a rooted plane tree $R$ obtained by contracting each box into a node. More precisely, if $\be=(v_1,\ldots,v_k)$ is a cadet sequence in $B$, then we delete all the right siblings of the nodes $v_2,\ldots,v_k$ (these are leaves of $T$) and contract the edges $(v_i,v_{i+1})$ of all $i\in[k-1]$. If $\be$ corresponds to the $p$th node of $R$ in prefix order of $R$, we denote $L_p=\{v_1,\ldots,v_k\}$ and we denote by $\ga_p$ the $(m,N)$-configuration corresponding to $\be$ in the sense of Remark~\ref{rk:cadet-seq=config}. 
Clearly, the boxed tree $(T,B)$ is uniquely determined by the triple $(R,(\ga_1,\ldots,\ga_{|B|}),(L_1,\ldots,L_{|B|}))$.
Hence, the boxed trees $(T,B)\in \mU_\SS^{(m)}$ are in bijection with triples $(R,(\ga_1,\ldots,\ga_{n}),(L_1,\ldots,L_{n}))$,
where
\begin{itemize} 
\item $R$ is a rooted plane tree, and $n$ is the number of nodes of $R$,
\item $\ga_1,\ldots,\ga_{n}$ are $(m,N)$-configurations of width $\ch(v_1),\ldots,\ch(v_{n})$ respectively, where $v_1,\ldots,v_n$ are the nodes of $R$ in prefix order, and $\ch(v)$ is the number of children of the node $v$ (including leaves), 
\item and $(L_1,\ldots,L_{n})$ is a set partition of $V(|\ga_1|+\cdots +|\ga_{n}|)$ such that for all $p\in[n]$, $|\ga_p|=(k_{p,1},\ldots,k_{p,N})$, where $\ds k_{p,a}=|\{i~|~(a,i)\in L_p\}|$.
\end{itemize}
Lastly, there are 
$$\frac{\left(|\ga_1|+\cdots +|\ga_{n}|\right)!}{\prod_{p=1}^{n}|\ga_{p}|!}$$
ways to choose the set partition $(L_1,\ldots,L_{n})$. Hence 
\begin{equation*}
U_\SS(y,\tt)=\sum_{R\in \mR}~\prod_{~v\textrm{ node of }R}\left(-\sum_{\ga\in \mC^{(m,N)}~|~ \wid(\ga)=\ch(v)}y^{\energy_\SS(\ga)}\frac{\tt^{|\ga|}}{|\ga|!}\right),
\end{equation*}
where $\mR$ is the set of rooted plane trees and $\mC^{(m,N)}$ is the set of $(m,M)$-configurations. Similarly,
\begin{equation*}
U_\SS^\bu(y,\tt)=\sum_{R\in \mR}\vv(R)\prod_{v\textrm{ node of }R}\left(-\sum_{\ga\in \mC^{(m,N)}~|~ \wid(\ga)=\ch(v)}y^{\energy_\SS(\ga)}\frac{\tt^{|\ga|}}{|\ga|!}\right),
\end{equation*}
where $\vv(R)$ is the number of vertices of $R$.

Next, we express $Z_{\hSS,\nn}(\de,y)$ in terms of $(m,N)$-configurations. 
Let $\nn\in\NN^N$, and let $\xx=(x_{v})_{v\in V(\nn)}\in [\de]^{|\nn|}$.
Intuitively, we think of each coordinate $x_{v}$ as the position of a particle in the space $[\de]$, and we will distinguish \emph{runs} which are groups of particles that are close to one another. This is represented in Figure~\ref{fig:particle-config-multi}. 
Let $v'_1,\ldots,v'_{|\nn|}\in V(\nn)$ be defined by $\{v'_1,\ldots,v'_{|\nn|}\}=V(\nn)$, and the conditions $\ds x_{v'_i}\leq x_{v'_{i+1}}$, and if $x_{v'_i}= x_{v'_{i+1}}$ then $v'_i<v'_{i+1}$. We denote $d_i=x_{v'_{i+1}}-x_{v'_{i}}$ for all $i\in[|\nn|-1]$, and adopt the convention $d_0=d_{|\nn|}=\infty$. 
A \emph{run} of $\xx$ is a subsequence $\rho=(v'_i,v'_{i+1},\ldots,v'_j)$, with $1\leq i\leq j\leq|\nn|$, such that $d_{i-1}>m$, $d_{j}>m$, and for all $i\leq p<j$, $d_{p}\leq m$.
We define the \emph{position} of $\rho$ as $\pos(\rho)=x_{v'_i}$, 
the \emph{width} of $\rho$ as $\wid(\rho)=x_{v'_j}-x_{v'_i}+m+1$, 
the \emph{labels} of $\rho$ as $\lab(\rho)=\{v'_i,\ldots,v'_j\}$, and
the \emph{size} of $\rho$ as $|\rho|=\kk=(k_1,\ldots,k_N)$ where $k_a=|\{i~|~(a,i)\in \lab(\rho)\}|$. 
For instance, the tuple $\xx$ represented in Figure~\ref{fig:particle-config-multi} has four runs having position 3, 7, 12, and 19 respectively, width 4, 3, 6, and 5 respectively, and size $(2,0)$, $(0,1)$, $(3,2)$, and $(1,1)$ respectively.
Lastly, the \emph{configuration} of $\rho$ is 
$\config(\rho)=\left((d_i,d_{i+1},\ldots, d_{j-1}),(u_i,\ldots,u_j)\right)$, 
where $(u_i,\ldots,u_j)$ is the unique order preserving relabeling of $(v'_i,\ldots,v'_j)$ in $V(\kk)$. For instance, in Figure~\ref{fig:particle-config-multi}, $\config(\rho_3)=((1,0,2,0),((2,2),(1,1),(1,2),(1,3),(2,1)))$. 
Note that $\ga=\config(\rho)$ is in $\mC^{(m)}(\kk)$, and $\wid(\rho)=\wid(\ga)$. Moreover it is easy to see that 
$$\ds \energy_\hSS(\xx)=\sum_{i=1}^r\energy_\hSS(\config(\rho_i)),$$ 
where $\rho_1,\ldots,\rho_r$ are the runs of $\xx$ (because pairs of particles at distance greater than $m$ do not contribute to the $\hSS$-energy). 
Moreover, the tuple $\xx$ is completely determined by the positions, labels, and configurations of its runs $\rho_1,\ldots,\rho_r$. 
The configurations of the runs are arbitrary, and given the configurations $\ga_1,\ldots,\ga_r$ of the runs there are 
$${\de +m+r-\wid(\ga_1)-\ldots-\wid(\ga_r)\choose r}$$ 
ways to choose the positions $(p_1,\ldots,p_r)\in[\de]^r$ (since the only constraints are $p_i+\wid(\ga_i)\leq p_{i+1}$ for all $i\in[r-1]$, and $p_r+\wid(\ga_r)\leq \de +m+1$). Also, there are
$\ds \frac{\nn!}{\prod_{i=1}^{r}|\ga_i|!}$
ways to choose the labels. Thus,
\begin{equation}\label{eq:cardinalityZ}
Z_{\hSS,\nn}(\de,y)=\nn!\sum_{r=0}^{\infty}\sum_{\ga_1,\ldots, \ga_r\in \mC^{(m,N)}\atop |\ga_1|+\cdots+|\ga_r|=\nn}{\de +m+r-\wid(\ga_1)-\cdots-\wid(\ga_r)\choose r}\prod_{i=1}^{r}\frac{y^{\energy_{\hSS}(\ga_i)}}{|\ga_{i}|!}.
\end{equation}

In order to prove~\eqref{eq:Zcoeff-as-boxed-multi}, we will now consider negative values of $\de$. Let us denote $\pol_r(x)=\frac{x(x-1)\cdots(x-r+1)}{r!}$. This is a polynomial in $x$, such that for all $x\in \NN$, ${x \choose r}=\pol_r(x)$. Let 
\begin{equation}\label{eq:cardinalityZtilde}
\widetilde{Z}_{\hSS,\nn}(\de,y)=\nn!\sum_{r=0}^{\infty}\sum_{\ga_1,\ldots, \ga_r\in \mC^{(m,N)}\atop |\ga_1|+\cdots+|\ga_r|=\nn}\pol_r\left(\de +m+r-\wid(\ga_1)-\cdots-\wid(\ga_r)\right)\prod_{i=1}^{r}\frac{y^{\energy_{\hSS}(\ga_i)}}{|\ga_{i}|!}.
\end{equation}
This is a polynomial in $\de$ and $y$ which coincides with $Z_{\hSS,\nn}(\de,y)$ for all integer $\de>m\cdot|\nn|$, because any $\ga\in\mC^{(m,N)}$ satisfies $\wid(\ga)-1\leq m\cdot|\ga|$. It remains to prove that for all $\de$,
\begin{equation}\label{eq:Zcoeff-as-boxed-multi2}
\widetilde{Z}_{\hSS,\nn}(\de,y)=[\tt^{\nn}]U_\hSS(y,\tt)^{-\de-m-2}U_\hSS^\bu(y,\tt).
\end{equation}
We observe that both sides of~\eqref{eq:Zcoeff-as-boxed-multi2} are polynomials in~$\de$. 
Indeed, $U_\hSS(y,(0,\ldots,0))=1$, hence the series 
$$U_\hSS(y,\tt)^{-\de}=\exp(-\de\log(U_\SS(y,\tt))=\exp\left(\de \sum_{k\geq 1}\frac{(1-U_S(\tt))^k}{k}\right)$$ 
has coefficients which are polynomial in~$\de$.
Thus, in order to prove~\eqref{eq:Zcoeff-as-boxed-multi2}, it suffices to prove it for infinitely many values of $\de\in \CC$.
Let $\al$ be a positive integer, let $\de=-m-1-\al$, and let
\begin{eqnarray*}
\widetilde{Z}_{\hSS}(\de,y,\tt)&=&\sum_{\nn\in\NN^N}\frac{\tt^\nn}{\nn!}\widetilde{Z}_{\hSS,\nn}(\de,y).
\end{eqnarray*}
We have 
\begin{eqnarray*}
\widetilde{Z}_{\hSS}(\de,y,\tt)&=&\sum_{r=0}^\infty\sum_{\ga_1,\ldots, \ga_r\in \mC^{(m,N)}}\pol_r\left(r-1-\al-\wid(\ga_1)-\cdots-\wid(\ga_r)\right)\prod_{i=1}^r\frac{y^{\energy_\hSS(\ga_i)}t^{|\ga_i|}}{|\ga_i|!}\\
&=&\sum_{r=0}^\infty\sum_{\ga_1,\ldots, \ga_r\in \mC^{(m,N)}}(-1)^r{\al+\wid(\ga_1)+\cdots+\wid(\ga_r)\choose r}\prod_{i=1}^r\frac{y^{\energy_\hSS(\ga_i)}t^{|\ga_i|}}{|\ga_i|!}.
\end{eqnarray*}
Using Claim~\ref{claim:tau} gives
\begin{eqnarray*}
\widetilde{Z}_{\hSS}(\de,y,\tt)&=&\sum_{r=0}^\infty(-1)^r\sum_{\ga_1,\ldots, \ga_r\in \mC^{(m,N)}}\tau(\al;\wid(\ga_1),\ldots,\wid(\ga_r))\prod_{i=1}^r\frac{y^{\energy_\hSS(\ga_i)}t^{|\ga_i|}}{|\ga_i|!}\\
&=&\left(\sum_{R\in \mR}\prod_{v\textrm{ node of }R}-\sum_{\ga\in \mC^{(m,N)},~ \wid(\ga)=\ch(v)}\frac{y^{\energy_\hSS(\ga)}t^{|\ga|}}{|\ga|!}\right)^{\al-1}\\
&&~~\times \left(\sum_{R\in \mR}\vv(R)\prod_{v\textrm{ node of }R}-\sum_{\ga\in \mC^{(m,N)},~ \wid(\ga)=\ch(v)}\frac{y^{\energy_\hSS(\ga)}t^{|\ga|}}{|\ga|!}\right)\\
&=&U_\hSS(y,t)^{-\de-m-2}U_\hSS^\bu(y,t).
\end{eqnarray*}
for all $\de\leq -m-2$. Thus, 
$\ds \widetilde{Z}_{\hSS}(\de,y,\tt)=U_\hSS(y,t)^{-\de-m-2}U_\hSS^\bu(y,t)$ for all $\de\in\CC$, and extracting the coefficient of $\tt^{\nn}$ gives~\eqref{eq:Zcoeff-as-boxed-multi2}.
\end{proof}

We can now complete the proof of Theorem~\ref{thm:energy-count-multi}. By Lemma~\ref{lem:3-multi},
$$\lim_{\de\to \infty}\frac{Z_{\hSS,\de}(\de,y,\tt)}{U_\hSS(y,\tt)^{-\de-m-2}U_\hSS^\bu(y,\tt)}=1,$$
where the limit is taken coefficient by coefficient in $\tt$. Thus, by Lemma~\ref{lem:2-multi},
\begin{eqnarray*}
\frac{1}{q}\log(P_\hSS(q,y,\tt))&=&\lim_{\de\to \infty} \frac{1}{\de}\log(Z_{\hSS,\de}(\de,y,\tt))\\
&=& \lim_{\de\to \infty} \frac{-\de-m-2}{\de}\log\left(U_\hSS(y,\tt)\right)+ \lim_{\de\to \infty} \frac{1}{\de}\log\left(U_\hSS^\bu(y,\tt)\right)\\
&& ~+\lim_{\de\to \infty} \log\left(\frac{Z_{\hSS,\de}(\de,y,\tt)}{U_\hSS(y,\tt)^{-\de-m-2}U_\hSS^\bu(y,\tt)}\right)\\
&=& -\log(U_\hSS(y,\tt)).
\end{eqnarray*}
Hence, 
$$\ds P_\hSS(q,y,\tt) =(U_\hSS(y,\tt))^{-q}=\left(\sum_{\nn\in\NN^N}\frac{\tt^\nn}{\nn!}\sum_{(T,B)\in \mU^{(m)}(\nn)}(-1)^{|B|}y^{\energy_\hSS(T,B)}\right)^{-q}.$$
This gives~\eqref{eq:energy-count-multi}, and setting $y=0$ gives~\eqref{eq:GF-regions-multi}. Lastly, in the case where $\hSS$ is multi-transitive, \eqref{eq:GF-regions-transitive-multi} follows from~\eqref{eq:GF-regions-multi} via~\eqref{eq:signed-to-unsigned}.

\medskip

\section{A simple bijection for regions of transitive deformations of the braid arrangement}\label{sec:bij}
In this section we present a simple bijection between regions of the arrangement $\mA_\SS$ and the trees in $\mT_\SS(n)$ for any transitive tuple $\SS$. We first recall a bijection between the regions of $\mA^{(m)}(n)$ and labeled parenthesis systems in Subsection~\ref{sec:bij-prelim}, and combine it with a convenient bijection between parenthesis systems and trees. Then, we treat the cases of the Shi, semiorder and Linial arrangements in Subsection~\ref{sec:bij-m=1}, before treating the general case in Subsection~\ref{sec:bij-gle}.

\subsection{Preliminary: bijection between regions of $\mA_{[-m..m]}(n)$ and~$\mT^{(m)}(n)$.}\label{sec:bij-prelim}~\\
We first recall a classical encoding of the regions of $\mA_{[-m..m]}(n)$ by labeled, non-nesting, parenthesis systems. An example is represented in Figure~\ref{fig:encoding-by-sketch}.

A $m$-\emph{parenthesis system} of \emph{size} $n$ is a word $w$ on the alphabet $\{\al,\be\}$ with $n$ letters $\al$ and $mn$ letters $\be$, 
such that no prefix of $w$ contains more $\be$'s than $m$ times the number of $\al$'s. It is well known that there are $\Cat^{(m)}(n)=\frac{((m+1)n)!}{n!(mn+1)!}$ $m$-parenthesis systems of size $n$, and that such words bijectively encode rooted plane $(m+1)$-ary trees with $n$ nodes (see e.g.~\cite[Chapter 5.3]{Stanley:volume2}). A \emph{$m$-sketch} of size $n$ is a word $\tw$ obtained from a parenthesis system $w$ by replacing the $i$th letter $\al$ by the letter $\al_{\pi(i)}$ for some permutation $\pi$ of $[n]$. We denote by $\tD^{(m)}(n)$ the set of $m$-sketches of size~$n$. 
Clearly, $|\tD^{(m)}(n)|=n!\Cat^{(m)}(n)=\frac{((m+1)n)!}{(mn+1)!}=|\mT^{(m)}(n)|$. We now describe bijections between the regions of $\mA_{[-m..m]}(n)$, and the sets $\tD^{(m)}(n)$ and $\mT^{(m)}(n)$. The case $m=1$, $n=3$ is represented in Figures~\ref{fig:Catalan-arrang-paren+mini} and~\ref{fig:Catalan-arrang-trees+mini}.

We first need to \emph{annotate} our sketches. Let $A^{(m)}(n)$ be the alphabet made of the $(m+1)n$ letters $\{\al_i^{(s)}~|~i\in[n],s\in[0..m]\}$. We call \emph{$\al$-letters} the letters $\al^{(0)}_i$ for $i\in[n]$, and \emph{$\be$-letters} the letters $\al^{(s)}_i$ for $i\in[n],s\in[m]$. 
For a word $\hw$ on the alphabet $A^{(m)}(n)$, we say that the letter $\al_i^{(s)}$ is \emph{active} in $\hw$ if $s<m$, and $\al_{i}^{(s)}$ appears in $\hw$ but $\al_i^{(s+1)}$ does not.
The \emph{annotation} of a sketch $\tw=\tw_1\cdots\tw_{(m+1)n}$ is the word $\hw=\hw_1\cdots\hw_{(m+1)n}$ obtained by applying the following rule for $p=1,2,\ldots,(m+1)n$: if $\tw_p=\al_i$ then set $\hw_p=\al^{(0)}_i$, while if $\tw_p=\be$ then set $\hw_p=\al^{(s+1)}_i$, where $\al_i^{(s)}$ is the first active letter in $\hw_1\cdots \hw_{p-1}$ (it is easy to see that there is always such a letter). 
We denote by $\mD^{(m)}(n)$ the set of annotated $m$-sketches of size $n$.
It is easy to see that a word $\hw=\hw_1\cdots \hw_{(m+1)n}$ is in $\mD^{(m)}(n)$ if and only if 
\begin{compactitem} 
\item[(a)] $\{\hw_1,\ldots,\hw_{(m+1)n}\}=A^{(m)}(n)$,
\item[(b)] for all $i\in[n]$ and all $s\in[m]$, the letter $\al_i^{(s-1)}$ appears before $\al_i^{(s)}$, 
\item[(c)] for all $i,j\in[n]$ and all $s,t\in[m]$, if $\al_i^{(s-1)}$ appears before $\al_j^{(t-1)}$, then $\al_i^{(s)}$ appears before $\al_j^{(t)}$.
\end{compactitem}

\fig{width=.9\linewidth}{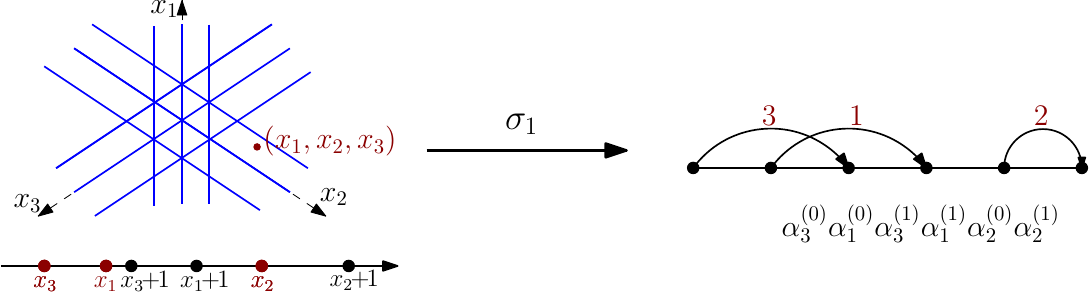}{The mapping $\si_1$ associating an annotated 1-sketch to any point $(x_1,\ldots,x_n)$ in $\RR^n\setminus \bigcup_{H\in \mA_{[-1..1]}(n)}H$. Graphically, the annotated 1-sketch $\hw=\hw_1\cdots \hw_{2n}$ is represented by a set of non-nesting parentheses on $2n$-points corresponding to the letters. More precisely, if the letters $\al^{(0)}_i$ to $\al^{(1)}_i$ are in position $p$ and $q$ of $\hw$, then parenthesis labeled $i$ goes from the $p$th point to the $q$th point.}

We now associate an annotated $m$-sketch of size $n$ to each region of $\mA_{[-m..m]}(n)$. Let $\xx=(x_1,\ldots,x_n)$ be a point in $\RR^n\setminus \bigcup_{H\in \mA_{[-m..m]}(n)}H$. Observe that the condition $\xx\notin \bigcup_{H\in \mA_{[-m..m]}(n)}H$ is equivalent to the fact that the numbers $\{x_i+s~|~i\in[n],s\in[0..m]\}$ are all distinct. We define $z_1,\ldots,z_{(m+1)n}$ by the conditions $z_1<\ldots<z_{(m+1)n}$ and $\{z_1,\ldots,z_{(m+1)n}\}= \{x_i+s~|~i\in[n],s\in[0..m]\}$. Then, we define $\si_m(\xx)=\hw_1\hw_2\cdots \hw_{(m+1)n}$, where $\hw_p=\al_i^{(s)}$ if $z_p=x_{i}+s$. 
Here are basic properties of the mapping $\si_m$.
\begin{compactenum}
\item[(i)] For any $\xx\notin \bigcup_{H\in \mA_{[-m..m]}(n)}H$, the word $\si_m(\xx)$ is an annotated $m$-sketch. Indeed, it clearly satisfies the conditions (a), (b), (c). 
\item[(ii)] The mapping $\si_m$ is constant over each region of $\mA_{[-m..m]}(n)$. Indeed, the order of the numbers $\{x_i+s~|~i\in[n],s\in[0..m]\}$ cannot change when $\xx$ moves continuously inside $\RR^n\setminus \bigcup_{H\in \mA_{[-m..m]}(n)}H$.
\item[(iii)] The sketch $\si_m(\xx)$ identifies the region containing $\xx$. 
Indeed, for all $i,j\in[n]$, and all $s\in[0..m]$, $x_i-x_j<s$ if $\al^{(0)}_i$ appears before $\al^{(s)}_j$ in $\si_m(\xx)$ and $x_i-x_j>s$ otherwise.
\item[(iv)] For any annotated $m$-sketch $\hw=\hw_1\cdots \hw_{(m+1)n}$, there exists $\xx\in \RR^n\setminus \bigcup_{H\in \mA_{[-m..m]}(n)}H$ such that $\si(\xx)=\hw$. 
Indeed, one can simultaneously define $\xx\in\si^{-1}(\hw)$ and $z_1,\ldots,z_{(m+1)n}$ by applying the following rules for $p=1,2,\ldots,(m+1)n$:
if $\hw_p=\al^{(0)}_i$, then set $z_p=z_{p-1}+1/2^p$ (with $z_0=0$) and set $x_i=z_p$, while if $\hw_p=\al^{(s)}_i$ with $s\neq 0$, then set $z_p=x_{i}+s$.
\end{compactenum}
Properties (i) and (ii) show that $\si_m$ is a mapping from the regions of $\mA_{[-m..m]}(n)$ to $\mD^{(m)}(n)$. Properties (iii) and (iv) imply that $\si_m$ is a bijection. In particular, this shows that $\mA_{[-m..m]}(n)$ has $\frac{((m+1)n)!}{(mn+1)!}$ regions.
The bijection $\si_1$ is represented in Figure~\ref{fig:Catalan-arrang-paren+mini}. Note that (iii) gives the inverse bijection $\si_m^{-1}$ in terms of inequality, while (iv) gives an explicit point in $\si_m^{-1}(\hw)$.

\fig{width=\linewidth}{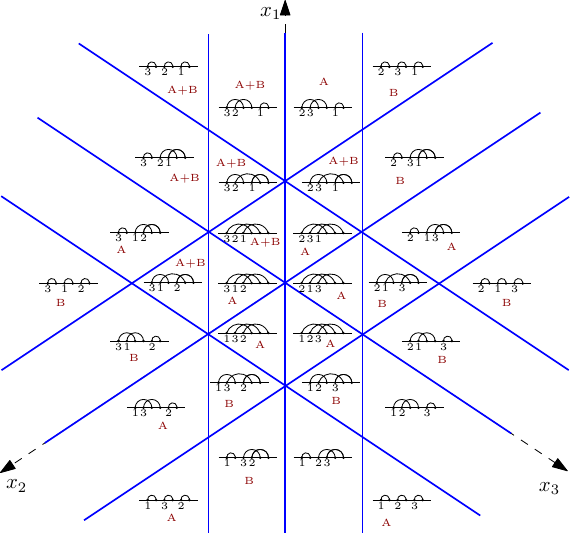}{The Catalan arrangement $\mA_{\{-1,0,1\}}(3)$, and the annotated 1-sketches corresponding to each region. The sketches marked $A$ are Shi maximal, the sketches marked $B$ are semiorder maximal, and those marked both $A$ and $B$ are Linial maximal.}



Next, we describe a bijection $\phi_m$ between $\mD^{(m)}(n)$ and the set $\mT^{(m)}(n)$ of $(m+1)$-ary trees with labeled nodes\footnote{Of course, any classical bijection between $m$-parenthesis systems and $(m+1)$-ary trees induces a bijection between $\mD^{(m)}(n)$ and $\mT^{(m)}(n)$ (by sending labels from the parentheses to the nodes). The non-classical bijection $\phi_m$ is chosen because it is well adapted to the ``non-nesting'' nature of annotated sketches.}.
Let $T$ be a rooted plane tree. We define the \emph{drift} of a vertex $v$ of $T$ as $\drift(v)=\ls(u_1)+\cdots+\ls(u_k)$, where $u_0,u_1,\ldots,u_k=v$ are the vertices on the path from the root $u_0$ to $v$. We define the total order $\prec_T$ on vertices of $T$ by setting $u\prec_T v$ if either $\drift(u)<\drift(v)$, or $\drift(u)=\drift(v)$ and $u$ appears before $v$ in the postfix order of $T$ (recall that the \emph{postfix order} is the order of appearance of the vertices when turning counterclockwise around the tree starting at the root). The order $\prec_T$ is represented in Figure~\ref{fig:phi}(a).
 
\fig{width=\linewidth}{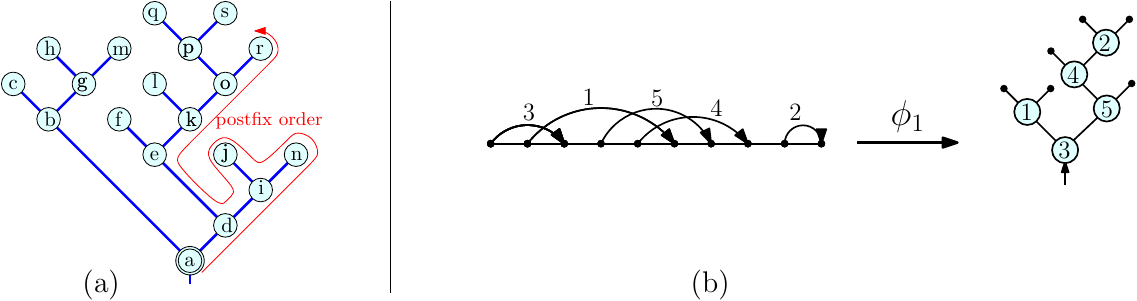}{(a) The order $\prec_T$ for this binary tree $T$ is $a\!\prec_T \!b\!\prec_T\!c\!\prec_T\!d\!\prec_T\cdots\prec_T s$. (b) The annotated 1-sketch $\hw=\all_3\all_1\bee_3\all_5\all_4\bee_1\bee_5\bee_4\all_2\bee_2$, and the associated tree $\phi_1(\hw)\in\mT^{(1)}(5)$.}

Let $\tT(n)$ be the set of rooted plane trees, with (at most $n$) nodes labeled with distinct numbers in $[n]$, and some special leaves called \emph{buds} (see Figure~\ref{fig:phi-steps}). 
If $T$ has some buds, we call \emph{first bud} the least bud of $T$ for the order $\prec_T$.
Let $\hw=\hw_1\ldots\hw_{(m+1)n}\in\mD^{(m)}(n)$, and let $\hw^p=\hw_1\ldots\hw_{p}$ be its prefix of length $p$. We define the trees $\tphi_m(\hw^0),\ldots,\tphi_m(\hw^{(m+1)n})\in\tT(n)$ as follows: 
\begin{compactitem}
\item $\tphi_m(\hw^0)$ is the tree with one bud and no other vertex,
\item if $p>0$ and $\hw_p=\al^{(s)}_i$ with $s>0$, then $\tphi_m(\hw^p)$ is obtained from $\tphi(\hw^{p-1})$ by replacing its first bud by a (non-bud) leaf,
\item if $p>0$ and $\hw_p=\al^{(0)}_i$, then $\tphi_m(\hw^p)$ is obtained from $\tphi(\hw^{p-1})$ by replacing its first bud by a node labeled $i$ with $m+1$ children, all of them buds. 
\end{compactitem}
The trees $\tphi(\hw^{p})$ are represented in Figure~\ref{fig:phi-steps}. It is clear, by induction on $p$, that $\tphi_m(\hw^p)$ has $1+m\, n_\al-n_\be$ buds, where $n_\al$ and $n_\be$ are respectively the number of $\al$-letters and $\be$-letters in $\hw^p$. 
In particular $\tphi_m(\hw^p)$ has at least one bud, so that $\tphi_m$ is well defined, and $\tphi_m(\hw)$ has exactly one bud. We denote by $\phi_m(\hw)$ the tree in $\mT^{(m)}(n)$ obtained from $\tphi_m(\hw)$ by replacing its bud by a leaf; see Figure~\ref{fig:phi}(b). Before showing that $\phi_m$ is a bijection, we describe the inverse mapping $\psi_m$. 
Let $T\in \mT^{(m)}(n)$ and let $u_0\prec_T u_1\prec_T\cdots \prec_T u_{(m+1)n}$ be the vertices of $T$ ($T$ has $n$ nodes and $mn+1$ leaves). 
Let $\psi_m(T)$ be the word $\hw=\hw_1\ldots\hw_{(m+1)n}$ defined as follows: for all $p\in[(m+1)n]$, if $u_p$ is the $(s+1)$st child of the node $i$, then $\hw_p=\al_i^{(s)}$. 

\fig{width=\linewidth}{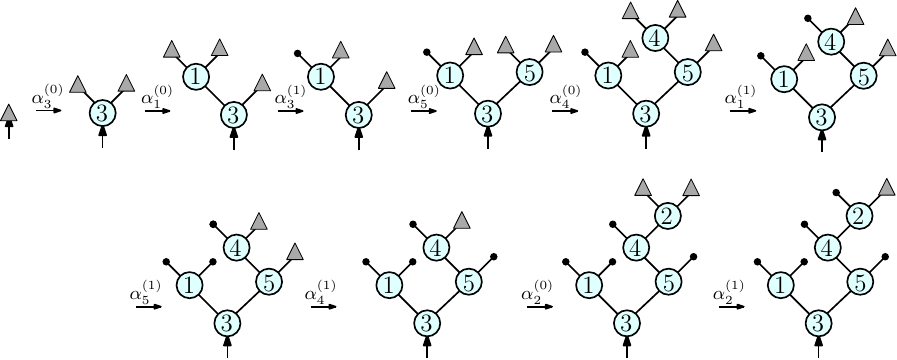}{The trees $\tphi_m(\hw^0),\ldots,\tphi_m(\hw^{(m+1)n})$ for $\hw=\all_3\all_1\bee_3\all_5\all_4\bee_1\bee_5\bee_4\all_2\bee_2$. Buds are represented by triangles, and leaves by dots.}

\begin{prop}\label{prop:bij-sketch-to-tree}
The mapping $\phi_m$ is a bijection between $\mD^{(m)}(n)$ and $\mT^{(m)}(n)$. The inverse mapping is $\psi_m$. Moreover, for $\hw\in\mD^{(m)}(n)$, if the letter following $\al_i^{(s)}$ in $\hw$ is $\al_j^{(t)}$, then the $(s+1)$st child of the node $i$ in $T=\phi_m(\hw)$ is the node $j$ if $t=0$, and a leaf otherwise.
\end{prop}

\begin{proof}
First note that for all $T\in\mT^{(m)}(n)$, the word $\psi_m(T)$ clearly satisfies the properties (a), (b), (c) of annotated $m$-sketches. Hence $\psi_m$ is a mapping from $\mT^{(m)}(n)$ to $\mD^{(m)}(n)$. 

Next we give an alternative description of $\psi_m$. Let $T\in\mT^{(m)}(n)$, and let $u_0\prec_T u_1\prec_T\cdots \prec_T u_{(m+1)n}$ be its vertices. 
Let $\tpsi_m(T)$ be the word $\tw=\tw_1\ldots \tw_{(m+1)n}$, where $\tw_p=\al_i$ if $u_{p-1}$ is the node labeled $i$, and $\tw_p=\be$ if $u_{p-1}$ is a leaf. We now show that $\hw=\psi_m(T)$ is the annotation of $\tw=\tpsi_m(T)$. It is easy to see that for all $q\in[0..(m+1)n]$, if the vertex $u_q$ is a node, then $u_{q+1}$ is its first child. Now let $p\in[(m+1)n]$ such that $\hw_p=\al^{(0)}_i$ for some $i\in[n]$. In this case $u_p$ is the first child of the node $i$, hence $u_{p-1}$ is the node $i$ and $\tw_p=\al_i$. Suppose now that $p\in[(m+1)n]$ is such that $\hw_p=\al^{(s)}_i$ for some $s\in[m],i\in[n]$. In this case $u_p$ is not a first child, thus $u_{p-1}$ is a leaf and $\tw_p=\be$. This proves that $\hw=\psi_m(T)$ is indeed the annotation of $\tw=\tpsi_m(T)$.

Next we prove that $\psi_m\circ\phi_m=\Id$ and $\phi_m\circ\psi_m=\Id$. Let $\hw\in\mT^{(m)}(n)$, and let $\hw^p$ be the prefix of length $p$. Let $T=\phi(\hw)$, and let $T^p=\tphi_m(\hw^p)$. 
We claim that the order in which the non-bud vertices are created in the sequence $\tphi_m(\hw^0),\tphi_m(\hw^1)\ldots,\tphi_m(\hw^{(m+1)n}),\phi_m(\hw)$ is the same as the order $\prec_T$. 
Indeed, for all $p\in[(m+1)n]$ the order $\prec_{T^p}$ on the vertices of $T^p$ coincide with the order $\prec_T$. So the first bud of $T^p$ is less than any bud of $T^{p+1}$ for the order $\prec_T$. This establish the claim. From the claim, and the alternative description of $\psi_m$ it follows directly that $\psi_m\circ\phi_m(\hw)=\hw$. And since $|\mD^{(m)}(n)|=|\mT^{(m)}(n)|$, $\phi_m$ and $\psi_m$ are inverse bijections.

Lastly, suppose that for $\hw\in\mD^{(m)}(n)$ we have $\hw_p=\al_i^{(s)}$ and $\hw_{p+1}=\al_j^{(t)}$. In this case, $u_p$ is the $(s+1)$st child of the node $i$ (by definition of $\psi_m$) and $u_p$ is the node $j$ if $s=0$ and a leaf otherwise (by definition of $\tpsi_m$).
\end{proof}

Let $\Phi_m=\phi_m\circ \sigma_m$ be the bijection from the the regions of $\mA_{[-m..m]}(n)$ to $\mT^{(m)}(n)$, and let $\Psi_m=\sigma^{-1}_m\circ \psi_m$ its inverse. 
\begin{lemma}\label{lem:Psi-characterization}
For $T\in\mT^{(m)}(n)$, $\Psi_m(T)$ is the region of $\mA_{[-m..m]}(n)$ made of the points $\xx=(x_1,\ldots,x_n)$ satisfying the following inequalities for all distinct integers $i,j\in[n]$, and all $s\in[0..m]$: $x_i-x_j<s$ if $i\prec_T v$ where $v$ is the $(s+1)$st child of the node $j$, and $x_i-x_j>s$ otherwise.
\end{lemma}

\begin{proof} 
Let $\xx\in\Psi_m(T)=\psi_m(\si_m^{-1}(T)$. By Property (iii) of $\sigma_m$, $x_i-x_j<s$ if and only if $\al_i^{(0)}$ appears before $\al_j^{(s)}$ in $\hw=\psi_m(T)$. By definition of $\psi_m$, this happens if and only if $u\prec_T v$, where $u$ is the first child of the node $i$. Moreover, since $u\neq v$ and $u$ is the successor of $i$ in the $\prec_T$ order, this happens if and only if $i\prec_T v$.
\end{proof}


\medskip
\subsection{Bijections for the Shi, semiorder, and Linial arrangements.}\label{sec:bij-m=1}~\\
In this section we give a bijection between the regions of $\mA_S(n)$, and the trees in $\mT_S$ for all $S\subseteq\{-1,0,1\}$. 
Up to symmetry, the interesting cases are $S=\{-1,0,1\}$ (Catalan arrangement), $S=\{0,1\}$ (Shi arrangement), $S=\{-1,1\}$ (semiorder arrangement), $S=\{1\}$ (Linial arrangement), and $S=\{0\}$ (braid arrangement).
We already treated the case $S=\{-1,0,1\}$ in the previous Section. 
For the Shi arrangements our bijection can be seen as a relative to~\cite{Athanasiadis:bijection-Shi} as discussed in Section \ref{sec:bij-link-Shi}. The bijection for the semiorder and Linial arrangements seem to be new.


The basic idea of our bijection for the Shi, semiorder, and Linial arrangements is to think of regions in these arrangements as union of regions of the Catalan arrangement. Then we will choose a canonical representative among these regions, so as to identify regions of the Shi, semiorder and Linial arrangements with certain canonical 1-sketches. We show that the bijection $\phi_1$ between $\mD^{(1)}(n)$ and $\mT^{(1)}(n)$ induces bijections between the canonical 1-sketches for $\mA_{\{0,1\}}(n)$, $\mA_{\{-1,1\}}(n)$, and $\mA_{\{1\}}(n)$ and the trees in $\mT_{\{0,1\}}(n)$, $\mT_{\{-1,1\}}(n)$, and $\mT_{\{1\}}(n)$ respectively. This is represented in Figure~\ref{fig:Catalan-arrang-trees+mini}. 
Moreover, the bijections induced by $\Phi_1=\phi_1\circ \si_1$ between regions and trees have simple inverses.

\begin{definition}
Let $\hw$ and $\hw'$ be annotated 1-sketches of size $n$.
We say that $\hw$ and $\hw'$ are related by a \emph{Shi move} if $\hw'$ is obtained from $\hw$ by swapping two consecutive letters $\bee_i$ and $\all_j$ with $i<j$. We say that $\hw$ and $\hw'$ are related by a \emph{semiorder move} if $\hw'$ is obtained from $\hw$ by swapping two consecutive letters $\all_i$ and $\all_j$ and also two consecutive letters $\bee_i$ and $\bee_j$ (for the same pair $\{i,j\}$). We say that $\hw$ and $\hw'$ are related by a \emph{Linial move} if they are related by either a Shi or semiorder move. Lastly, we say that $\hw$ and $\hw'$ are \emph{Shi equivalent} (resp. \emph{semiorder equivalent}, \emph{Linial equivalent}) if one can be obtained from the other by performing a series of Shi (resp. semiorder, Linial) moves. 
\end{definition}

Let $\hw,\hw'\in\mD^{(1)}(n)$ be annotated 1-sketches.
Let $\rho=\si^{-1}(\hw)$ and $\rho'=\si^{-1}(\hw')$ be the regions of $\mA_{\{-1,0,1\}}(n)$ corresponding to $\hw$ and $\hw'$. Observe that $\hw$ and $\hw'$ are related by the Shi move swapping $\all_i$ and $\bee_j$ if and only if the regions $\rho$ and $\rho'$ are separated only by the hyperplane $H_{i,j,-1}=\{x_i-x_j=-1\}$. 
Thus, $\hw$ and $\hw'$ are Shi equivalent if and only if one can go from $\rho$ to $\rho'$ only crossing hyperplanes of the forms $H_{i,j,-1}$ for $i<j$. In other words, $\hw$ and $\hw'$ are Shi equivalent if and only if $\rho$ and $\rho'$ are contained in the same region of the Shi arrangement $\mA_{\{0,1\}}(n)$. Similarly, $\hw$ and $\hw'$ are related by the semiorder move swapping $\all_i$ and $\all_j$ (and $\bee_i$ and $\bee_j$) if and only if the regions $\rho$ to $\rho'$ are separated only by the hyperplane $H_{i,j,0}=\{x_i-x_j=0\}$. Thus, $\hw$ and $\hw'$ are semiorder equivalent if and only if $\rho$ and $\rho'$ are contained in the same region of the semiorder arrangement $\mA_{\{-1,1\}}(n)$. Also, $\hw$ and $\hw'$ are Linial equivalent if and only if $\rho$ and $\rho'$ are contained in the same region of the Linial arrangement $\mA_{\{1\}}(n)$. To summarize:
\begin{lemma}\label{lem:moves=regions}
Let $\hw$ and $\hw'$ be annotated 1-sketches of size $n$, and let $\rho=\si^{-1}(\hw)$ and $\rho'=\si^{-1}(\hw')$ be the regions of $\mA_{\{-1,0,1\}}(n)$ corresponding to $\hw$ and $\hw'$. The annotated 1-sketches $\hw$ and $\hw'$ are Shi (resp. semiorder, Linial) equivalent if and only if $\rho$ and $\rho'$ are contained in the same region of $\mA_{\{0,1\}}(n)$ (resp. $\mA_{\{-1,1\}}(n)$, $\mA_{\{1\}}(n)$).
\end{lemma}

We consider the lexicographic order $\prec$ on $\mD^{(1)}(n)$ given by the following order on the alphabet: $\bee_1\prec \bee_2\prec \cdots\prec \bee_n\prec\all_1\prec\all_2\prec\cdots\prec\all_n$.
We say that an annotated 1-sketch $\hw$ is \emph{Shi locally-maximal} (resp. \emph{semiorder locally-maximal}, \emph{Linial locally-maximal}) if it is larger than any 1-sketch obtained from $\hw$ by a single Shi (resp. semiorder, Linial) move. We say that an annotated 1-sketch $\hw$ is \emph{Shi maximal} (resp. \emph{semiorder maximal}, \emph{Linial maximal}) if it is larger than any Shi (resp. semiorder, Linial) equivalent 1-sketch. The maximal 1-sketches are indicated in Figure~\ref{fig:Catalan-arrang-paren+mini}.

On the one hand, Lemma~\ref{lem:moves=regions} implies that regions of $\mA_{\{0,1\}}(n)$ (resp. $\mA_{\{-1,1\}}(n)$, $\mA_{\{1\}}(n)$) are in bijection with Shi (resp. semiorder, Linial) maximal 1-sketches in $\mD^{(1)}(n)$. 
On the other hand, locally-maximal 1-sketches are easy to characterize. The following result shows that the two notions actually coincide.

\begin{lemma}\label{lem:local-max-is-max} 
An annotated 1-sketch $\hw\in\mD^{(1)}(n)$ is Shi (resp. semiorder, Linial) maximal if and only if it is Shi (resp. semiorder, Linial) locally-maximal.
\end{lemma}

Before proving Lemma~\ref{lem:local-max-is-max} we explore its consequences.

\begin{corollary}\label{cor:bijection-Shi-SO-Linial}
The mapping $\Psi_1=\Phi_1^{-1}$ between $\mT^{(1)}(n)$ and the regions of $\mA_{\{-1,0,1\}}(n)$ induces a bijection $\Psi_{\{0,1\}}$ (resp. $\Psi_{\{-1,1\}}$, $\Psi_{\{1\}}$) between the trees in $\mT_{\{0,1\}}(n)$ (resp. $\mT_{\{-1,1\}}(n)$, $\mT_{\{1\}}(n)$) and the regions of $\mA_{\{0,1\}}(n)$ (resp. $\mA_{\{-1,1\}}(n)$, $\mA_{\{1\}}(n)$).
\begin{compactenum} 
\item For $T\in \mT_{\{0,1\}}(n)$ the region $\Psi_{\{0,1\}}(T)$ is defined by the following inequalities for all $1\leq i<j\leq n$:
$x_i-x_j<0$ iff $i\prec_T j$ (that is to say, node $i$ is less than node $j$ for the $\prec_T$ order), and $x_i-x_j<1$ iff $i\prec_T v$, where $v$ is the right child of $j$.
\item For $T\in \mT_{\{-1,1\}}(n)$ the region $\Psi_{\{-1,1\}}(T)$ is defined by the following inequalities for all $1\leq i<j\leq n$:
$x_i-x_j>-1$ iff $j\prec_T u$, and $x_i-x_j<1$ iff $i\prec_T v$, where $u$ is the right child of $i$ and $v$ is the right child of $j$.
\item For $T\in \mT_{\{1\}}(n)$ the region $\Psi_{\{1\}}(T)$ is defined by the following inequalities for all $1\leq i<j\leq n$: $x_i-x_j<1$ iff $i\prec_T v$, where $v$ is the right child of $j$.
\end{compactenum}
\end{corollary}

\fig{width=\linewidth}{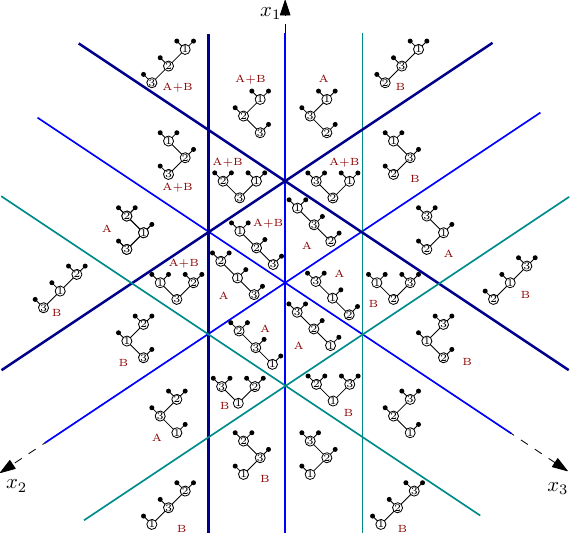}{The Catalan arrangement $\mA_{\{-1,0,1\}}(3)$, and the labeled binary trees corresponding to each region. The trees marked $A$ (including those denoted $A+B$) are in bijection with the regions of the Shi arrangement $\mA_{\{0,1\}}(3)$. The trees marked $B$ (including those denoted $A+B$) are in bijection with the regions of the semiorder arrangement $\mA_{\{-1,1\}}(3)$. The trees marked $A+B$ are in bijection with the regions of the Linial arrangement $\mA_{\{1\}}(3)$.}

Corollary~\ref{cor:bijection-Shi-SO-Linial} is illustrated in Figure~\ref{fig:Catalan-arrang-trees+mini}.

\begin{remark} The method used above works just as well for the braid arrangement: the induced bijection is between the regions of $\mA_{\{0\}}(n)$ and the binary trees with no right child. These trees (which look like paths) are the ones near the origin in Figure~\ref{fig:Catalan-arrang-trees+mini} (note that they are a subsets of the Shi trees). Of course, such a machinery is unnecessary for this simple case, which could be treated with $m=0$, but it illustrates the fact that arrangements corresponding to small values of $m$ are embedded in the bijective framework corresponding to larger values of $m$. 
\end{remark}

\begin{proof}[Proof of Corollary~\ref{cor:bijection-Shi-SO-Linial}]
It is clear that an annotated 1-sketch $\hw$ is Shi locally-maximal if and only if for all $i\in [n]$ the letter $\bee_i$ is not followed by a letter $\all_j$ with $j>i$. By Proposition~\ref{prop:bij-sketch-to-tree}, this means that $\hw$ is Shi locally-maximal if and only if for all $i\in [n]$ the right child of the node $i$ in $T=\phi_1(\hw)$ is not a node $j$ with $j>i$. In other words, $\hw$ is Shi locally-maximal if and only if $\phi_1(\hw)$ is in $\mT_{\{0,1\}}$. 
 
Similarly, an annotated 1-sketch $\hw$ is semiorder locally-maximal if and only if for all $i\in [n]$ the letters $\all_i$ and $\bee_i$ are not followed by the letters $\all_j$ and $\bee_j$ respectively, with $j>i$. Note that if $\all_i$ is followed by $\all_j$ and $\bee_i$ is followed by a $\be$-letter, then this letter is necessarily $\bee_j$. Thus, by Proposition~\ref{prop:bij-sketch-to-tree}, $\hw$ is semiorder locally-maximal if and only if no node $i\in [n]$ of $T$ has both a left child which is a node $j>i$ and a right child which is a leaf. Thus, $\hw$ is semiorder locally-maximal if and only if $\phi_1(\hw)$ is in $\mT_{\{-1,1\}}$. 

Lastly, an annotated 1-sketch $\hw$ is Linial locally-maximal if and only if it is both Shi locally-maximal and semiorder locally-maximal. Thus $\hw$ is Linial locally-maximal if and only if $\phi_1(\hw)$ is in $\mT_{\{0,1\}}\cap\mT_{\{-1,1\}}=\mT_{\{1\}}$. 

Moreover, the description of the bijection $\Psi_{S}$ is immediate from Lemma~\ref{lem:Psi-characterization} as the inequalities defining the region $\Psi_{S}(T)$ are a subset of the inequalities defining the region $\Psi_{1}(T)$ (the inequalities of the form $x_i-x_j<s$ or $x_i-x_j>s$ for $i<j$ and $s\in S$).
\end{proof}

\begin{proof}[Proof of Lemma~\ref{lem:local-max-is-max}]
We first treat the case of the Shi arrangement. Let $\hw=\hw_1\cdots\hw_{2n}$ and $\hw'=\hw_1'\cdots\hw_{2n}'$ be two 1-sketches. Suppose that $\hw$ and $\hw'$ are Shi equivalent. It is easy to see (by induction on the number of Shi moves), that 
\begin{compactitem}
\item[(a)] for all $i,j\in[n]$, $\all_i$ appears before $\all_j$ in $\hw$ if and only if $\all_i$ appears before $\all_j$ in $\hw'$, 
\item[(b)] for all $i>j\in[n]$, $\bee_i$ appears before $\all_j$ in $\hw$ if and only if $\bee_i$ appears before $\all_j$ in $\hw'$.
\end{compactitem}
Now suppose that $\hw$ is Shi locally-maximal and $\hw'$ is Shi maximal. We want to show $\hw=\hw'$. Suppose by contradiction that they are different, and let $p\in[2n]$ be such that $\hw_p\neq \hw_p'$ and $\hw_k=\hw_k'$ for all $k\in[p-1]$.
Since $\hw\prec \hw'$, either 
\begin{compactitem} 
\item[(i)] $\hw_p=\bee_i$ and $\hw_p'=\bee_j$ with $i<j$,
\item[(ii)] $\hw_p=\all_i$ and $\hw_p'=\all_j$ with $i<j$,
\item[(iii)] $\hw_p=\bee_i$ and $\hw_p'=\all_j$ for some $i,j$.
\end{compactitem}
However case (i) is impossible for 1-sketches: if $\hw_k=\hw_k'$ for all $k\in[p-1]$, $\hw_p=\bee_i$ and $\hw_p'=\bee_j$ then $i=j$. Moreover case (ii) is impossible by (a). Hence (iii) holds. Let $q>p$ be such that $\hw_{q}=\all_j$. 
By (a) and (b), we must have $\hw_{q-1}=\bee_k$ with $k<j$. But this contradicts the fact that $\hw$ is Shi locally-maximal. Hence $\hw=\hw'$ as wanted.

Next, we treat the case of the semiorder arrangement. Given $\hw\in\mD^{1}(n)$, we say that $i,j\in[n]$ are $\hw$-\emph{exchangeable} if in $\hw$ the letters $\all_i$ and $\all_j$ are separated only by $\al$-letters, and $\bee_i$ and $\bee_j$ are separated only by $\be$-letters. Let $\hw,\hw'\in\mD^{(1)}(n)$ be two semiorder equivalent 1-sketches. It is easy to see that $\hw'$ is obtained from $\hw$ by replacing the letters $\all_{i}$ and $\bee_{i}$ by $\all_{\pi(i)}$ and $\bee_{\pi(i)}$ for a permutation $\pi$ of $[n]$ such that for all $i\in[n]$, $i$ and $\pi(i)$ are $\hw$-exchangeable.
Now suppose that $\hw$ is semiorder locally-maximal and $\hw'$ is semiorder maximal. We want to show $\hw=\hw'$. Suppose by contradiction that they are different, and let $p\in[2n]$ be such that $\hw_p\neq \hw_p'$ and $\hw_k=\hw_k'$ for all $k\in[p-1]$.
Since $\hw\prec \hw'$, either (i), (ii) or (iii) holds. However (i) is impossible as before, and (iii) is impossible because the parenthesis systems underlying $\hw$ and $\hw'$ are equal. Hence, (ii) holds. 
By the remark above, $i,j$ are $\hw$-equivalent. Hence denoting $\hw_{p+d}=\all_j$, we get that for all $c\in[d]$ the letter $\hw_{p+c}$ has the form $\all_{i_c}$ for some $i_c$ which is $\hw$-exchangeable with $i$. Since $\hw$ is locally maximal, we have $i>i_1>\cdots>i_d=j$. This contradicts $i<j$, hence $\hw=\hw'$ as wanted.


Lastly, we treat the case of the Linial arrangement. Let $\hw,\hw'\in\mD^{(1)}(n)$ be two Linial equivalent 1-sketches. It is easy to see that (b) holds. Suppose now that $\hw$ is Linial locally-maximal and $\hw'$ is Linial maximal. We want to show $\hw=\hw'$. Suppose by contradiction that they are different, and let $p\in[2n]$ be such that $\hw_p\neq \hw_p'$ and $\hw_k=\hw_k'$ for all $k\in[p-1]$. Since $\hw\prec \hw'$, either (i), (ii) or (iii) holds. However (i) is impossible as before, so that $\hw'_p=\all_j$ for some $j\in[n]$. Let $d>0$ such that $\hw_{p+d}=\all_j$. 
Suppose first that $\hw_{p},\hw_{p+1},\ldots,\hw_{p+d}$ are all $\al$-letters. We denote $\hw_{p+c}=\all_{i_c}$ for all $c\in[0..d]$. 
In this case, $i_0<i_d=j$ (since $\hw\prec\hw'$), hence taking the least index $i_c$ we have $i_c<i_{c+1}$ and $i_c<j$. 
Since $\hw$ is Linial locally-maximal, the letter following $\bee_{i_c}$ has the form $\all_k$ with $k<i_c$ (otherwise it would be $\bee_{i_{c+1}}$ or $\all_k$ with $k>i_c$ and one could do an increasing Linial move)
Lastly, since $k<i_c,j$ and $\all_k$ is between $\bee_{i_c}$ and $\bee_{j}$ in $\hw$, property (b) implies that $\all_k$ is between $\bee_{i_c}$ and $\bee_{j}$ in $\hw'$. Hence $\all_{i_c}$ appears before $\all_j$ in $\hw'$. We reach a contradiction. It remains to treat the case where $\{\hw_{p},\hw_{p+1},\ldots,\hw_{p+d-1}\}$ contains a $\be$-letter. Let $\bee_i$ be the last $\be$-letter before $\all_j$ in $\hw$, and let $\all_{i_0}$ be the letter following $\bee_i$. By (b), we have $i<j$. Moreover, since $\hw$ is Linial locally-maximal, $i_0<i$. Since $i_{0}<j$ and all the letters between $\all_{i_0}$ and $\all_j$ are $\al$-letters, the same reasoning as before leads to a contradiction. Hence $\hw=\hw'$ as wanted.
\end{proof}

\medskip
\subsection{General bijection for transitive deformation of the braid arrangement.}\label{sec:bij-gle}~\\
In this section we generalize the strategy adopted in Section~\ref{sec:bij-gle} in order to establish bijections between regions of arbitrary transitive deformations of the braid arrangement, and trees.

We fix a positive integer $N$ and an ${N\choose 2}$-tuple of finite sets of integers $\SS=(S_{i,j})_{1\leq i<j\leq N}$. Recall that $\mA_{\SS}$ is the arrangement in $\RR^N$ made of the hyperplanes $H_{i,j,s}$ for all $1\leq i<j\leq N$ and all $s\in S_{i,j}$. Recall also that when $\SS$ is transitive, the regions of $\mA_{\SS}$ are equinumerous to the trees $\mT_\SS$ defined in Definition~\ref{def:mTSS}.
For $T\in\mT_\SS$, we denote by $\Psi_\SS(T)$ the set of points $(x_1,\ldots,x_N)$ satisfying the following inequalities for all $1\leq i<j\leq N$ and $s\in S_{i,j}$:
\begin{compactitem} 
\item for $s\geq 0$, $x_i-x_j<s$ if the node $i$ is less than the $(s+1)$st child of the node $j$ in the $\prec_T$ order, and $x_i-x_j>s$ otherwise,
\item for $s< 0$, $x_i-x_j>s$ if the node $j$ is less than the $(-s+1)$st child of the node $i$ in the $\prec_T$ order, and $x_i-x_j<s$ otherwise.
\end{compactitem}
Our goal is to establish the following result.

\begin{theorem}\label{thm:bij-gle}
If $\SS=(S_{i,j})_{1\leq i<j\leq N}$ is transitive (see Definition~\ref{def:transitive-multi}), then $\Psi_\SS$ is a bijection between the set $\mT_\SS$ of trees and the regions of $\mA_\SS$.
\end{theorem}


\begin{remark} The bijections $\Psi_\SS$ are compatible with refinements of arrangements. Indeed $\mA_{\SS'}$ is a refinement of $\mA_{\SS}$ if and only if $\SS'=(S_{i,j}')_{1\leq i<j\leq N}$, with $S_{i,j}\subseteq S_{i,j}'$ for all $i,j$. In this case, $\mT_\SS\subseteq \mT_{\SS'}$, and for all $T\in \mT_\SS$, $\Psi_{\SS'}(T)\subseteq \Psi_{\SS}(T)$. 
\end{remark}

Our strategy to prove Theorem~\ref{thm:bij-gle} is the same as in Section~\ref{sec:bij-m=1}. Let $m=\max(|s|,~s\in \cup S_{a,b})$, so that $\mT_\SS$ is a subarrangement of $\mA_{[-m..m]}(N)$. We will think of regions of $\mA_\SS$ as equivalence class of regions of $\mA_{[-m..m]}(N)$, and the bijection $\Phi_m$ defined in Section~\ref{sec:bij-prelim} will induce a bijection $\Phi_\SS$ between regions of $\mA_\SS$ and $\mT_\SS$.
 
\begin{definition}
Let $\hw,\hw'$ be annotated $m$-sketches of size $N$. 
Let $i,j\in[N]$ with $i<j$, and let $s\in[-m..m]$.
We say that $\hw$ and $\hw'$ are related by a \emph{$(i,j,s)$-move} if for all $k\in[0..m]\cap[-s..m-s]$ the pair of letters $\{\al_i^{(k)},\al_j^{(s+k)}\}$ are consecutive in $\hw$, and $\hw'$ is obtained from $\hw$ by swapping each of these pairs of letters.
A \emph{$\SS$-move} is any $(i,j,s)$-move with $1\leq i<j\leq N$, and $s\notin S_{i,j}$. 
We say that $\hw$ and $\hw'$ are \emph{$\SS$-equivalent} if one can be obtained from the other by performing a series of $\SS$-moves.
\end{definition}

\begin{example} Let $m=1$. The Shi moves defined in Section~\ref{sec:bij-m=1} are all the $(i,j,-1)$-moves, and the semiorder moves are all the $(i,j,0)$-moves. 
Hence the Shi (resp. semiorder, Linial) moves are the $\SS$-moves for the tuple $\SS=(S_{i,j})_{1\leq i<j\leq N}$ with $S_{i,j}=\{0,1\}$ (resp. $S_{i,j}=\{-1,1\}$, $S_{i,j}=\{1\}$) for all $1\leq i<j\leq N$. 
\end{example}

We consider the following order $\prec$ on the alphabet $\mA^{(m)}(N)$: $\al_i^{(s)}\prec \al_j^{(t)}$ if either $s>t$, or $s=t$ and $i<j$.
We now consider the lexicographic order $\prec$ on $\mD^{(m)}(n)$ corresponding the order $\prec$ on the letters.
An annotated $m$-sketch $\hw\in \mD^{(m)}(n)$ is \emph{$\SS$-locally-maximal} if it is greater than any $m$-sketch obtained from $\hw$ by a single $\SS$-move. It is \emph{$\SS$-maximal} if it is greater than any $\SS$-equivalent $m$-sketch. Lastly, a region $\rho$ of $\mA_{[-m..m]}(N)$ is said \emph{$\SS$-maximal} if the annotated $m$-sketch $\si_m(\rho)$ is $\SS$-maximal. We now establish two easy lemmas.

\begin{lemma}\label{lem:regions=max}
Each region of $\mA_\SS$ contains a unique $\SS$-maximal region of $\mA_{[-m..m]}(N)$. 
\end{lemma}

\begin{proof}
Let $\hw,\hw'\in \mD^{(m)}(N)$, and let $\rho=\si_m^{-1}(\hw)$ and $\rho'=\si_m^{-1}(\hw')$ be the associated regions of $\mA_{[-m..m]}(n)$. It is clear that $\hw$ are $\hw'$ are related by a $(i,j,s)$-move (for $1\leq i<j\leq N$, and $s\in[-m..m]$) if and only if the regions $\rho$ and $\rho'$ are separated only by the hyperplane $H_{i,j,s}$. Thus $\hw$ and $\hw'$ are $\SS$-equivalent if and only if $\rho$ and $\rho'$ are in the same regions of $\mA_\SS$. Thus, for in any region $R$ of $\mA_\SS$, exactly one of the regions $\rho$ of $\mA_{[-m..m]}(N)$ contained in $R$ is $\SS$-maximal.
\end{proof}

\begin{lemma}\label{lem:trees=local-max}
Let $\hw\in \mD^{(m)}(n)$. The sketch $\hw$ is $\SS$-locally-maximal if and only if the tree $\phi_m(\hw)$ is in $\mT_\SS$. In other words, $\phi_m$ induces a bijection between $\SS$-locally-maximal regions of $\mA_{[-m..m]}(N)$ and $\mT_\SS$.
\end{lemma}

\begin{proof}
Let $\hw\in \mD^{(m)}(n)$ and let $T=\phi_m(\hw)$. 
For $0\leq i<j\leq n$, and $s\in[m]$, a $(i,j,s)$-moves on $\hw$ is possible and gives an annotated $m$-sketch $\hw'\succ \hw$ if and only if in $\hw$ the letter $\al_j^{(s)}$ is immediately followed by $\al_i^{(0)}$, and for all $t\in[s+1..m]$, the letters $\al_j^{(t)}$ and $\al_i^{(t-s)}$ are consecutive.
By definition of annotations, this holds if and only if the letter $\al_j^{(s)}$ is immediately followed by $\al_i^{(0)}$, and for all $t\in[s+1..m]$, the letters $\al_j^{(t)}$ is immediately followed by a $\be$-letter.
By Proposition~\ref{prop:bij-sketch-to-tree}, this holds if and only if in the tree $T$ the node $i$ is the $(s+1)$st child of the node $j$, and the right siblings of $i$ are leaves (so that $i=\cadet(j)$ and $\ls(i)=s$).

Similarly, for $0\leq i<j\leq n$, and $s\in[-m..0]$, a $(i,j,s)$-moves on $\hw$ is possible and gives an annotated $m$-sketch $\hw'\succ \hw$ if and only if in $\hw$ the letter $\al_i^{(-s)}$ is immediately followed by $\al_j^{(0)}$ and for all $t\in[-s+1..m]$, the letters $\al_i^{(t)}$ is immediately followed by a $\be$-letter. By Proposition~\ref{prop:bij-sketch-to-tree}, this holds if and only if in the tree $T$ the node $j$ is the $(s+1)$st child of the node $i$, and the right siblings of $j$ are leaves (so that $j=\cadet(i)$ and $\ls(j)=-s$).

Thus $\hw$ is $\SS$-locally-maximal if and only if the tree $T$ satisfies the following property for all $0\leq i<j\leq n$: 
if $i=\cadet(j)$ then $\ls(i)\in S_{i,j}\cup \{0\}$, and if $j=\cadet(i)$ then $-\ls(j)\in S_{i,j}$. This holds if and only if $T$ is in $\mT_\SS$. 
\end{proof}

We now complete the proof of Theorem~\ref{thm:bij-gle}. From Lemma~\ref{lem:Psi-characterization}, it is clear that for any tree $T\in\mT^{(m)}(N)$, $\Psi_\SS(T)$ is the region of $\mA_\SS$ containing the region $\Psi_m(T)$ of $\mA_{[-m..m]}(N)$. Hence, by Lemma~\ref{lem:trees=local-max}, the mapping $\Psi_\SS$ is a surjection between the trees in $\mT_\SS$ and the regions of $\mA_\SS$ containing at least one $\SS$-locally-maximal region of $\mA_{[-m..m]}(N)$. And since any $\SS$-maximal region is $\SS$-locally-maximal, Lemma~\ref{lem:trees=local-max} ensures that $\Psi_\SS$ is a surjection between the trees in $\mT_\SS$ and the regions of $\mA_\SS$ (all this holds even if $\SS$ is not transitive). Now assuming that $\SS$ is transitive, Theorem~\ref{thm:unsigned-count-multi}, ensures that the regions of $\mA_\SS$ are equinumerous to the trees in $\mT_\SS$, so $\Psi_\SS$ is actually a bijection.


\medskip

\section{Concluding remarks}\label{sec:conclusion}
We conclude with some additional links to the literature and some open questions.

\subsection{Bijections for the Shi arrangement.}\label{sec:bij-link-Shi}~\\
We now explain how our bijection $\Psi_{\{0,1\}}$ for the Shi arrangement $\mA_{\{0,1\}}(n)$ relates to the existing bijections described in~\cite{Athanasiadis:bijection-Shi} and~\cite{Stanley:hyperplane-tree-inversions} between regions of $\mA_{\{0,1\}}(n)$ and parking functions of size $n$. The correspondence is represented in Figure \ref{fig:Shi-bijections}. Recall that a \emph{parking function of size $n$} is a $n$-tuple $(p_1,\ldots,p_n)$ of integers in $[0..n-1]$ such that for all $k\in[n]$, $\ds ~k\leq\left|\{i\in[n]~|~p_i<k\}\right|$. 

\fig{width=\linewidth}{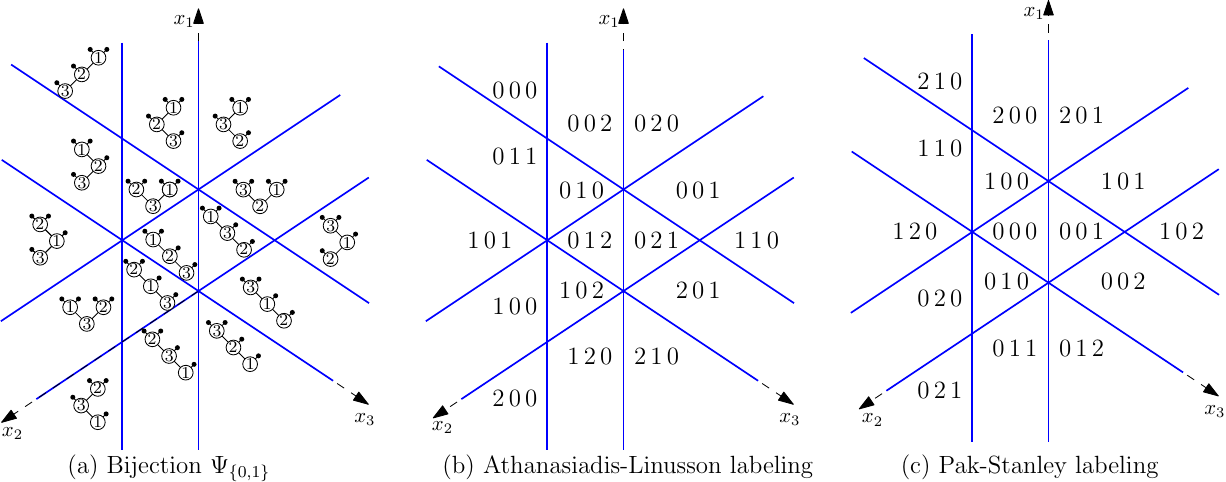}{The bijection $\Psi_{\{0,1\}}$, the \emph{Athanasiadis-Linusson labeling} and the \emph{Pak-Stanley labeling}.}

The first bijection discovered for the Shi arrangement is the so-called \emph{Pak-Stanley labeling} of the regions described in~\cite{Stanley:hyperplane-tree-inversions} (and earlier in~\cite[Section 5]{Stanley:hyperplane-interval-orders-overview} where Igor Pak is credited for suggesting the labeling in the case $m=1$, without proof). This bijection associates to a region $\rho$ of $\mA_{\{0,1\}}(n)$ the parking function $(p_1,\ldots,p_n)$, where for all $i\in[n]$,
$$p_i=\left|\{k\in[i-1]~|~ x_k<x_i\}\right|\,+\,\left|\{k\in[i+1..n]~|~x_k+1<x_i\}\right|,$$
where $(x_1,\ldots,x_n)$ is any point in the region $\rho$. This is represented in Figure \ref{fig:Shi-bijections}(b).
It follows directly from the definition of $\Psi_{\{0,1\}}$ that for any tree $T\in\mT_{\{0,1\}}(n)$, the Pak-Stanley labeling of the region $\rho=\Psi_{\{0,1\}}(T)$ is the parking function $\la_1(T)=(p_1,\ldots,p_n)$ given by 
$$p_i=|\{k\in[i-1]~|~ \textrm{node } k\prec_T \textrm{node }i\}|\,+\, |\{k\in[i+1..n]~|~ \textrm{right child of node } k\preceq_T \textrm{node }i\}|.$$

Another bijection for the Shi arrangement was established by Athanasiadis and Linusson in~\cite{Athanasiadis:bijection-Shi}. This bijection has two steps. The first step associates to each region $\rho$ of $\mA_{\{0,1\}}(n)$ a \emph{diagram} $\delta(\rho)$. The second step associates to the diagram $\de(\rho)$ a partition function that we call \emph{Athanasiadis-Linusson labeling} of $\rho$.
A reader familiar with~\cite{Athanasiadis:bijection-Shi} will have no difficulty seeing that the diagram $\delta(\rho)$ is closely related to the Shi-maximal 1-sketch that we associated to $\rho$ in Section \ref{sec:bij-m=1}. This induces a correspondence between our bijection and the Athanasiadis-Linusson labeling that we now state (the easy proof is omitted). For $T\in\mT_{\{0,1\}}(n)$, the Athanasiadis-Linusson labeling of the region $\rho=\Psi_{\{0,1\}}(T)$ is the parking function $\la_2(T)=(p_1,\ldots,p_n)$ obtained as follows. For all $i\in[n]$, we consider the path of vertices $v_1,v_2,\ldots,v_\ell$, where $v_1$ is the node $i$, $v_\ell$ is a leaf, and $v_{k+1}$ is the right child of $v_k$ for all $k\in[\ell-1]$. Then, $p_i$ is the number of leaves greater than $v_\ell$ for the $\prec_T$ order. This is represented in Figure \ref{fig:Shi-bijections}(c).

It is not very hard to see that the correspondence $\la_2$ is a bijection, and this is why the bijection in~\cite{Athanasiadis:bijection-Shi} can be considered a close relative of $\Psi_{\{0,1\}}$. However, it is less clear why the correspondence $\la_1$ is a bijection.

\medskip
\subsection{Regions of the Linial arrangements and binary search trees}\label{sec:bij-link-Linial}~\\
We now discuss the Linial arrangement $\mA_{\{1\}}(n)$. Stanley had conjectured that the regions of $\mA_{\{1\}}(n)$ were equinumerous to \emph{binary search trees with $n$ nodes}, that is, trees in $\mT^{(1)}(n)$ satisfying the Condition (iii) of Figure~\ref{fig:condition-Shi-so-Linial}. This fact was proved independently in~\cite{Postnikov:coxeter-hyperplanes} and~\cite{Athanasiadis:finite-field-method}. In~\cite{Postnikov:coxeter-hyperplanes,Postnikov:Thesis} Postnikov and Stanley listed several combinatorial classes equinumerous to the regions of $\mA_{\{1\}}(n)$, and some bijections between them. But, up to now, no bijection was known between these classes and the regions of $\mA_{\{1\}}(n)$. We remedy to this situation by giving bijections between regions of $\mA_{\{1\}}(n)$, the set $\mT_{\{1\}}(n)$ (which was not in the list), and the set $\mB(n)$ of binary search trees with $n$ nodes (which was in the list). The bijection between the regions of $\mA_{\{1\}}(n)$ and $\mT_{\{1\}}(n)$ was established in Section~\ref{sec:bij-m=1} (see also Figure~\ref{fig:Linial-bijection}). We now describe a recursive bijection $\theta$, represented in Figure~\ref{fig:bij-Linial-binarysearch}, between $\mT_{\{1\}}(n)$ and $\mB(n)$. 
\fig{width=.8\linewidth}{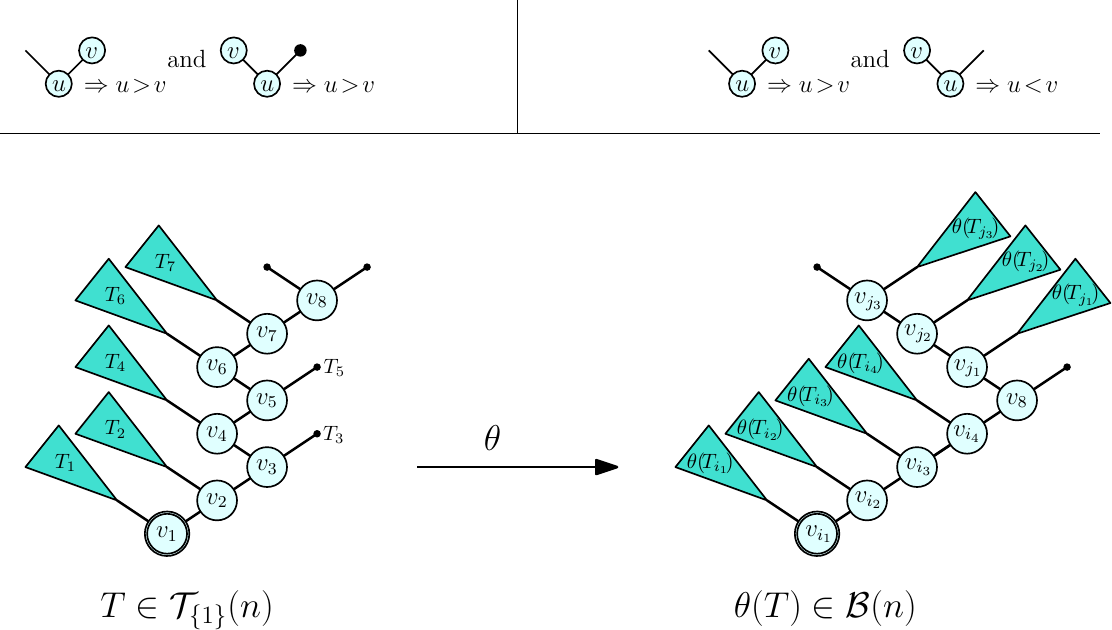}{The recursive bijection $\theta$ from $\mT_{\{1\}}(n)$ to $\mB(n)$. In this example, exactly two of the trees $\theta(T_{j_1}),\theta(T_{j_2}),\theta(T_{j_3})$ have no node, while the third tree is $\theta(T_p)$ for the only integer $p\in\{1,2,4,6,7\}$ such that $T_p$ has at least one node, and the root of $\theta(T_p)$ is less than $v_p$.}


For the tree $\tau_0\in\mT_{\{1\}}(0)$ made of one leaf, we define $\theta(\tau_0)=\tau_0\in\mB(0)$. We now consider $n>0$ and suppose that $\theta$ is a well defined bijection from $\mT_{\{1\}}(k)$ to $\mB(k)$ for all $k<n$. By extension, we may assume that $\theta$ is defined on all order-preserving relabeling of trees in $\mT_{\{1\}}(k)$ for all $k<n$ (with $\theta$ preserving the set of labels). Let $T$ be a tree in $\mT_{\{1\}}(n)$, and let $v_1$ be its root. Let $v_2,v_3,\ldots, v_{k+1}$ be defined by $v_{i+1}=\cadet(v_{i})$ for all $i\in [k]$, and the fact that both children of $v_{k+1}$ are leaves. For $i\in[k]$, let $T_i$ be the subtree of $T$ rooted at the child of $v_i$ which is not $v_{i+1}$; see Figure~\ref{fig:bij-Linial-binarysearch}. We denote by $I$ the subset of $[k]$ such that either $T_i$ is a reduced to a leaf which is the left child of $v_i$, or the root of $\theta(T_i)$ is a node which is greater than $v_i$. Let $i_1<\cdots<i_a=k+1$ be the elements of $I\cup \{k+1\}$ and let $k+1=j_1>\cdots>j_b$ be the elements of $[k+1]\setminus I$. We then define $\theta(T)$ as follows:
\begin{compactitem}
\item $v_{i_1}$ is the root,
\item for all $p\in[a-1]$, the node $v_{i_p}$ has right child $v_{i_{p+1}}$ and left child the root of the subtree $\theta(T_{i_p})$,
\item for all $p\in[b]$, the node $v_{j_p}$ has left child $v_{j_{p+1}}$ (or a leaf for $p=b$), and left child the root of the subtree $\theta(T_{j_p})$.
\end{compactitem}
It is easy to see that $\theta(T)$ is in $\mB(n)$ (since $v_1>v_2>\cdots>v_k$).
It is also easy to see, by induction on $n$, that $\theta$ is a bijection (one of the useful observations to invert $\theta$ is that $\theta$ transforms subtrees $T_j$ which are right leaves into subtrees which are right leaves). The bijection $\theta$ is applied to a tree in $\mT_{\{1\}}(10)$ in Figure \ref{fig:exp-bijection-Linial-binary}.
\fig{width=.7\linewidth}{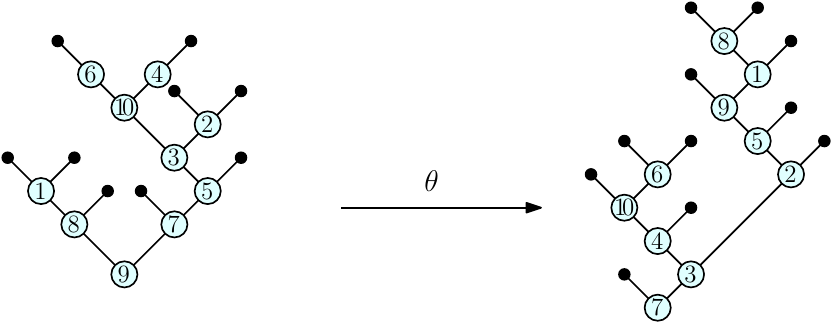}{The bijection $\theta$ from $\mT_{\{1\}}(n)$ to $\mB(n)$.}

\medskip
\subsection{Open questions.}~\\
The braid arrangement is associated to the root system $A_{n-1}$, in the sense that the hyperplane have the form $<\al,\xx>=0$ for the positive roots $\al$ of $A_{n-1}$. The (deformations of) arrangements corresponding to other root systems are known to share some of the properties of (deformations of) the braid arrangement (see e.g.~\cite{Athanasiadis:arrangement-coxeter,Postnikov:coxeter-hyperplanes,Shi:nb-regions-Weyl}). Thus, a natural question is whether the results of the present paper can be extended to this more general setting. Another direction for future research is to use the bijections presented here in order to obtain more refined counting formulas for the regions of deformed braid arrangements, by taking into account additional parameters of these regions (in the spirit of e.g.~\cite{Stanley:hyperplane-tree-inversions,Armstrong:Shi-Ish}). We now state two open questions.

It was shown in Section~\ref{sec:bij-gle}, that when the tuple $\SS$ is not transitive, the mapping $\Psi_\SS$ still gives a surjection between the trees in $\mT_\SS$ and the regions of $\mA_\SS$. Indeed, $\Psi_\SS$ gives a bijection between the subset of trees corresponding to $\SS$-maximal regions and the regions of $\mA_\SS$.
\begin{question} 
For a non-transitive tuple $\SS$, is it possible to characterize a subset $\widetilde{\mT}_\SS$ of $\mT_\SS$ in bijection with the regions of $\mA_\SS$ via $\Psi_\SS$? 
\end{question}
Let us consider for example the non-transitive set $S=\{-2,0,2\}$. The set $\mT_{\{-2,0,2\}}(n)$ contains all the trees in $\mT^{(2)}(n)$ such that if the rightmost child of any node $v$ is a leaf, then the middle child is also a leaf. However, because $\mA_{\{-2,0,2\}}(n)$ is just a dilation of the Catalan arrangement $\mA_{\{1,0,1\}}(n)$, we know that the regions of $\mA_{\{-2,0,2\}}(n)$ are in bijection with the set $\mT^{(1)}(n)$, or equivalently, the set $\widetilde{\mT}_{\{-2,0,2\}}(n)$ of trees in $\mT_{\{-2,0,2\}}(n)$ such that the middle child of any node is a leaf. 
In general, one could hope to find the desired subset $\widetilde{\mT}_\SS$ of $\mT_\SS$ either starting from the counting formulas in terms of boxed-trees (Theorem~\ref{thm:signed-count-multi}) and applying some sign-reversing involutions, or by using more direct bijective considerations.

A related problem is to find a more illuminating proof of our bijective results (Theorem~\ref{thm:bij-gle}). In the case of the Shi, semiorder, and Linial arrangement we gave a direct proof involving Lemma~\ref{lem:local-max-is-max} showing that locally-maximal regions are maximal.
The argument given there can actually be extended to the $m$-Shi, $m$-semiorder and $m$-Linial arrangements discussed in Section~\ref{sec:known-arrangements}. However it is unclear whether such an approach would work in the general case (hence removing the need of using Theorem~\ref{thm:unsigned-count-multi}).
\begin{question} 
Is there a direct, preferably geometric, proof that $\SS$-locally-maximal regions are $\SS$-maximal whenever $\SS$ is transitive?
\end{question}

\noindent \textbf{Acknowledgements:} We thank the referees for several helpful suggestions. We thank Priyavrat Deshpande and Krishna Menon for pointing out a mistake in Section \ref{sec:bij-prelim}.

\bibliographystyle{plain} 
\bibliography{biblio-hyperplanes}

\end{document}